\documentclass[10pt]{article}

\usepackage[letterpaper, top=25.4mm, bottom=25.4mm, left=25.4mm, right=25.4mm, includefoot]{geometry}

\usepackage[utf8]{inputenc}
\usepackage[T1]{fontenc}
\usepackage{lmodern}
\usepackage{alphabeta}

\usepackage{titling}
\usepackage{xspace}
\usepackage{soul} % for \st to strike out text
\usepackage{ifthen}

\usepackage{ifthen}

\usepackage{tikz}
\usepackage{xparse}

\newcommand{\colorname}[1]{%
  \ifnum#1=1 red\else%
    \ifnum#1=2 blue\else%
      \ifnum#1=3 green\else%
        \ifnum#1=4 magenta\else%
          \ifnum#1=5 orange\else%
            black%
          \fi%
        \fi%
      \fi%
    \fi%
  \fi%
}

\NewDocumentCommand{\colordisk}{mm}{%
\raisebox{-.6mm}{
  \begin{tikzpicture}
    \fill[\colorname{#1}]
      (0,0) arc[start angle=180, end angle=360, radius=0.172] -- (0.5,0) -- cycle;
    \fill[\colorname{#2}]
      (0.344,0) arc[start angle=0, end angle=180, radius=0.172] -- (-0.5,0) -- cycle;
  \end{tikzpicture}
}}

\newcommand{\pica}[4]{%
\mbox{
\{\!\!\!\!\!\!\!\!\!%
\mbox{\colordisk{#1}{#2}\!\!\!\!\!\!\!\!\!\!\colordisk{#3}{#4}}\!\!\!\!,\!\!\!\!\!\!\!\!\!
\mbox{\colordisk{#1}{#4}\!\!\!\!\!\!\!\!\!\!\colordisk{#3}{#2}}\!\!\!\!,\!\!\!\!\!\!\!\!\!
\mbox{\colordisk{#1}{#3}\!\!\!\!\!\!\!\!\!\!\colordisk{#2}{#4}}\!\!\!\!\}}}

\newcommand{\descriptionfont}[1]{{\bfseries\sffamily #1}}

\usepackage[inline]{enumitem}

\usepackage{dsfont}
\usepackage{setspace}

\usepackage{bm}
\usepackage{microtype}
\usepackage{amsmath}
\usepackage{amssymb}
\usepackage{amsfonts}
\usepackage{mathtools}
\usepackage[mathscr]{euscript}

\usepackage[hyphens]{url}
\usepackage{amsthm}

\usepackage{tikz}
\usepackage{xcolor}
\usepackage{subcaption}
\usepackage{graphicx}
\usepackage{wrapfig}

\usepackage{hyperref}
\usepackage{zref-clever}

\usepackage{nicefrac}
\usepackage[textwidth=2.2cm]{todonotes}

\colorlet{myGreen}{green!50!black}
\colorlet{myLightgreen}{green}
\colorlet{myRed}{red!90!black}
\definecolor{myBlue}{rgb}{0.25, 0.0, 1.0}
\definecolor{myLightBlue}{rgb}{0.39, 0.58, 0.93}
\colorlet{myViolet}{myBlue!55!myRed}
\definecolor{myOrange}{rgb}{1.0, 0.66, 0.07}

\definecolor{CornflowerBlue}{rgb}{0.39, 0.58, 0.93}
\definecolor{DarkGoldenrod}{rgb}{0.72, 0.53, 0.04}
\definecolor{BritishRacingGreen}{rgb}{0.0, 0.26, 0.15}
\definecolor{DarkMagenta}{rgb}{0.55, 0.0, 0.55}
\definecolor{AO}{rgb}{0.0, 0.5, 0.0}
\definecolor{BostonUniversityRed}{rgb}{0.8, 0.0, 0.0}
\definecolor{myRed}{rgb}{0.8, 0.0, 0.0}
\definecolor{DarkMidnightBlue}{rgb}{0.0, 0.2, 0.4}
\definecolor{DarkTangerine}{rgb}{1.0, 0.66, 0.07}
\definecolor{AppleGreen}{rgb}{0.55, 0.71, 0.0}
\definecolor{BrightUbe}{rgb}{0.82, 0.62, 0.91}
\definecolor{Amethyst}{rgb}{0.6, 0.4, 0.8}
\definecolor{DarkGray}{rgb}{0.52, 0.52, 0.51}
\definecolor{Gray}{rgb}{0.66, 0.66, 0.66}
\definecolor{BananaYellow}{rgb}{1.0, 0.88, 0.21}
\definecolor{Amber}{rgb}{1.0, 0.75, 0.0}
\definecolor{LightGray}{rgb}{0.83, 0.83, 0.83}
\definecolor{PrincetonOrange}{rgb}{1.0, 0.56, 0.0}
\definecolor{DeepCarrotOrange}{rgb}{0.91, 0.41, 0.17}
\definecolor{CarrotOrange}{rgb}{0.93, 0.57, 0.13}
\definecolor{MidnightBlue}{rgb}{0.1, 0.1, 0.44}
\definecolor{Magenta}{rgb}{0.50, 0.0, 0.50}
\definecolor{BrightPink}{rgb}{1.0, 0.0, 0.5}
\definecolor{BrilliantRose}{rgb}{1.0, 0.33, 0.64}
\definecolor{ChromeYellow}{rgb}{1.0, 0.65, 0.0}
\definecolor{HotMagenta}{rgb}{1.0, 0.11, 0.81}

\definecolor{DarkTangerine}{rgb}{1.0, 0.66, 0.07}
\definecolor{darkyellow}{rgb}{.7, .6, 0.0}
\definecolor{CornflowerBlue}{rgb}{0.39, 0.58, 0.93}
\definecolor{DarkGoldenrod}{rgb}{0.72, 0.53, 0.04}
\definecolor{BritishRacingGreen}{rgb}{0.0, 0.26, 0.15}
\definecolor{AO}{rgb}{0.0, 0.5, 0.0}
\definecolor{MidnightBlack}{rgb}{0.1,0.1,.34}
\definecolor{MidnightBlue}{rgb}{0.1,0.1,0.43}
\definecolor{Black}{rgb}{0,0, 0}
\definecolor{Blue}{rgb}{0, 0 ,1}
\definecolor{Red}{rgb}{1, 0 ,0}
\definecolor{White}{rgb}{1, 1, 1}
\definecolor{DeepMagenta}{rgb}{0.8, 0.0, 0.8}
\definecolor{grey}{rgb}{.6, .6, .6}
\definecolor{darkgrey}{rgb}{.33, .33, .33}
\definecolor{Mygreen}{rgb}{.0, .7, .0}
\definecolor{Yellow}{rgb}{.55,.55,0}
\definecolor{Mustard}{rgb}{1.0, 0.86, 0.35}
\definecolor{applegreen}{rgb}{0.55, 0.71, 0.0}
\definecolor{darkturquoise}{rgb}{0.0, 0.81, 0.82}
\definecolor{celestialblue}{rgb}{0.29, 0.59, 0.82}
\definecolor{green_yellow}{rgb}{0.68, 1.0, 0.18}
\definecolor{crimsonglory}{rgb}{0.75, 0.0, 0.2}
\definecolor{darkmagenta}{rgb}{0.30, 0.0, 0.30}
\definecolor{magenta}{rgb}{0.50, 0.0, 0.50}
\definecolor{internationalorange}{rgb}{1.0, 0.31, 0.0}
\definecolor{darkorange}{rgb}{1.0, 0.55, 0.0}
\definecolor{ao}{rgb}{0.0, 0.5, 0.0}
\definecolor{awesome}{rgb}{1.0, 0.13, 0.32}
\definecolor{darkcyan}{rgb}{0.0, 0.50, 0.50}
\definecolor{violet}{rgb}{0.93, 0.51, 0.93}
\definecolor{brown}{rgb}{0.65, 0.16, 0.16}
\definecolor{orange}{rgb}{1.0, 0.65, 0.0}
\definecolor{DarkGreen}{rgb}{0,.5,0}
\definecolor{BostonUniversityRed}{rgb}{0.8, 0.0, 0.0}
\definecolor{violet}{RGB}{155,110,240}

\newcommand{\darkgreen}[1]{{\color{DarkGreen}#1}}
\newcommand{\midnightblue}[1]{{\color{MidnightBlue}#1}}
\newcommand{\yellow}[1]{{\color{Yellow}#1}}
\newcommand{\black}[1]{{\color{Black}#1}}
\newcommand{\other}[1]{{\color{Other}#1}}
\newcommand{\darkmagenta}[1]{{\color{darkmagenta}#1}}
\newcommand{\darkmidnightblue}[1]{{\color{DarkMidnightBlue}#1}}
\newcommand{\magenta}[1]{{\color{magenta}#1}}
\newcommand{\mnb}[1]{{\color{MidnightBlue}#1}}
\newcommand{\blue}[1]{{\color{Blue}#1}}
\newcommand{\red}[1]{{\color{Red}#1}}
\newcommand{\green}[1]{{\color{Mygreen}#1}}
\newcommand{\ao}[1]{{\color{awesome}#1}}
\newcommand{\white}[1]{{\color{White}#1}}
\newcommand{\grey}[1]{{\color{grey}#1}}
\newcommand{\darkgrey}[1]{{\color{darkgrey}#1}}
\newcommand{\darkcyan}[1]{{\color{darkcyan}#1}}
\newcommand{\darkorange}[1]{{\color{darkorange}#1}}
\newcommand{\violet}[1]{{\color{violet}#1}}
\newcommand{\brown}[1]{{\color{brown}#1}}
\newcommand{\orange}[1]{{\color{orange}#1}}
\newcommand{\darkyellow}[1]{{\color{darkyellow}#1}}

\vfuzz \hfuzz
\setlist[itemize]{topsep=0pt,partopsep=0pt,itemsep=0pt,parsep=0pt}
\setlist[itemize,1]{label={\small\textbullet}}
\setlist[itemize,2]{label={\tiny\textbullet}}
\setlist[itemize,3]{label=$\cdot$}
\setlist[enumerate]{topsep=0pt,partopsep=0pt,itemsep=0pt,parsep=0pt}
\setlist[enumerate,1]{label=\roman*)}
\setlist[enumerate,2]{label=\alph*)}
\setlist[enumerate,3]{label=\arabic*)}

\hypersetup{
colorlinks=true,
linkcolor=AO!65!black,
citecolor=AO!65!black,
urlcolor=AppleGreen!65!black,
bookmarksopen=true,
bookmarksnumbered,
bookmarksopenlevel=2,
bookmarksdepth=3
}

\newenvironment{claimproof}[1][Proof of claim]
  {\renewcommand{\qedsymbol}{$\dashv$}\begin{proof}[#1]}
  {\end{proof}}

\theoremstyle{definition}

\newtheorem{environment}{Environment}[section]

\newtheorem{lemma}[environment]{Lemma}
\AddToHook{env/lemma/begin}{\zcsetup{countertype={environment=lemma}}}
\zcRefTypeSetup{lemma}{
Name-sg = Lemma ,
name-sg = Lemma ,
Name-pl = Lemmas ,
name-pl = Lemmas ,
}

\newtheorem*{lemma*}{Lemma}
\AddToHook{env/lemma*/begin}{\zcsetup{countertype={environment=lemma*}}}
\zcRefTypeSetup{lemma*}{
Name-sg = Lemma ,
name-sg = Lemma ,
Name-pl = Lemmas ,
name-pl = Lemmas ,
}

\newtheorem{proposition}[environment]{Proposition}
\AddToHook{env/proposition/begin}{\zcsetup{countertype={environment=proposition}}}
\zcRefTypeSetup{proposition}{
Name-sg = Proposition ,
name-sg = Proposition ,
Name-pl = Propositions ,
name-pl = Propositions ,
}

\newtheorem{corollary}[environment]{Corollary}
\AddToHook{env/corollary/begin}{\zcsetup{countertype={environment=corollary}}}
\zcRefTypeSetup{corollary}{
Name-sg = Corollary ,
name-sg = Corollary ,
Name-pl = Corollaries ,
name-pl = Corollaries ,
}

\newtheorem{theorem}[environment]{Theorem}
\AddToHook{env/theorem/begin}{\zcsetup{countertype={environment=theorem}}}
\zcRefTypeSetup{theorem}{
Name-sg = Theorem ,
name-sg = Theorem ,
Name-pl = Theorems ,
name-pl = Theorems ,
}

\newtheorem*{theorem*}{Theorem}
\AddToHook{env/theorem*/begin}{\zcsetup{countertype={environment=theorem*}}}
\zcRefTypeSetup{theorem*}{
Name-sg = Theorem ,
name-sg = Theorem ,
Name-pl = Theorems ,
name-pl = Theorems ,
}

\AddToHook{env/conjecture/begin}{\zcsetup{countertype={environment=conjecture}}}
\zcRefTypeSetup{conjecture}{
Name-sg = Conjecture ,
name-sg = Conjecture ,
Name-pl = Conjectures ,
name-pl = Conjectures ,
}

\newtheorem*{hypothesis*}{Hypothesis}
\AddToHook{env/hypothesis*/begin}{\zcsetup{countertype={environment=hypothesis*}}}
\zcRefTypeSetup{hypothesis*}{
Name-sg = Hypothesis ,
name-sg = Hypothesis ,
Name-pl = Hypotheses ,
name-pl = Hypotheses ,
}

\newtheorem{observation}[environment]{Observation}
\AddToHook{env/observation/begin}{\zcsetup{countertype={environment=observation}}}
\zcRefTypeSetup{observation}{
Name-sg = Observation ,
name-sg = Observation ,
Name-pl = Observations ,
name-pl = Observations ,
}

\AddToHook{env/example/begin}{\zcsetup{countertype={environment=example}}}
\zcRefTypeSetup{example}{
Name-sg = Example ,
name-sg = Example ,
Name-pl = Examples ,
name-pl = Examples ,
}

\AddToHook{env/remark/begin}{\zcsetup{countertype={environment=remark}}}
\zcRefTypeSetup{remark}{
Name-sg = Remark ,
name-sg = Remark ,
Name-pl = Remarks ,
name-pl = Remarks ,
}

\zcRefTypeSetup{equation}{
Name-sg = Equation ,
name-sg = Equation ,
Name-pl = Equations ,
name-pl = Equations ,
}

\zcRefTypeSetup{chapter}{
Name-sg = Chapter ,
name-sg = Chapter ,
Name-pl = Chapters ,
name-pl = Chapters ,
}

\zcRefTypeSetup{section}{
Name-sg = Section ,
name-sg = Section ,
Name-pl = Sections ,
name-pl = Sections ,
}

\zcRefTypeSetup{algorithm}{
Name-sg = Algorithm ,
name-sg = Algorithm ,
Name-pl = Algorithms ,
name-pl = Algorithms ,
}

\AddToHook{env/notation/begin}{\zcsetup{countertype={environment=notation}}}
\zcRefTypeSetup{notation}{
Name-sg = Notation ,
name-sg = Notation ,
Name-pl = Notations ,
name-pl = Notations ,
}

\AddToHook{env/question/begin}{\zcsetup{countertype={environment=question}}}
\zcRefTypeSetup{question}{
Name-sg = Question ,
name-sg = Question ,
Name-pl = Questions ,
name-pl = Questions ,
}

\newtheorem{problem}[environment]{Problem}
\AddToHook{env/problem/begin}{\zcsetup{countertype={environment=problem}}}
\zcRefTypeSetup{problem}{
Name-sg = Problem ,
name-sg = Problem ,
Name-pl = Problems ,
name-pl = Problems ,
}

\AddToHook{env/definition/begin}{\zcsetup{countertype={environment=definition}}}
\zcRefTypeSetup{definition}{
Name-sg = Definition ,
name-sg = Definition ,
Name-pl = Definitions ,
name-pl = Definitions ,
}

\zcRefTypeSetup{figure}{
Name-sg = Figure ,
name-sg = Figure ,
Name-pl = Figures ,
name-pl = Figures ,
}

\newtheoremstyle{beautypalour}% ⟨name⟩
{3pt}% ⟨Space above⟩1
{3pt}% ⟨Space below⟩1
{}% ⟨Body font⟩
{}% ⟨Indent amount⟩2
{\itshape}% ⟨Theorem head font⟩
{.}% ⟨Punctuation after theorem head⟩
{.5em}% ⟨Space after theorem head⟩3
{}% ⟨Theorem head spec (can be left empty, meaning ‘normal’)⟩
\theoremstyle{beautypalour}

\newtheorem{claim}{Claim}[environment]
\AddToHook{env/claim/begin}{\zcsetup{countertype={environment=claim}}}
\zcRefTypeSetup{claim}{
Name-sg = Claim ,
name-sg = Claim ,
Name-pl = Claims ,
name-pl = Claims ,
}

\usetikzlibrary{calc}
\usetikzlibrary{fit}
\usetikzlibrary{decorations}
\usetikzlibrary{decorations.pathmorphing}
\usetikzlibrary{decorations.text}
\usetikzlibrary{shapes,hobby}
\usetikzlibrary{positioning}

\tikzset{
	position/.style args={#1:#2 from #3}{
		at=($(#3)+(#1:#2)$)
	}
}

\tikzset{
  v:main/.style = {draw, circle, scale=0.8, thick,fill=black,inner sep=0.7mm},
  v:ghost/.style = {inner sep=0pt,scale=1},
  >={latex},
  e:marker/.style = {line width=8.5pt,line cap=round,opacity=0.35,color=DarkGoldenrod},
  e:main/.style = {line width=1pt},
}

\newcommand{\Acal}{\mathcal{A}}
\newcommand{\Bcal}{\mathcal{B}}
\newcommand{\Ccal}{\mathcal{C}}
\newcommand{\Dcal}{\mathcal{D}}
\newcommand{\Ecal}{\mathcal{E}}
\newcommand{\Fcal}{\mathcal{F}}
\newcommand{\Gcal}{\mathcal{G}}
\newcommand{\Hcal}{\mathcal{H}}
\newcommand{\Ical}{\mathcal{I}}
\newcommand{\Jcal}{\mathcal{J}}
\newcommand{\Kcal}{\mathcal{K}}
\newcommand{\Lcal}{\mathcal{L}}
\newcommand{\Mcal}{\mathcal{M}}
\newcommand{\Ncal}{\mathcal{N}}
\newcommand{\Ocal}{\mathcal{O}}
\newcommand{\Pcal}{\mathcal{P}}
\newcommand{\Qcal}{\mathcal{Q}}
\newcommand\Rcal{\mathcal{R}}
\newcommand{\Scal}{\mathcal{S}}
\newcommand{\Tcal}{\mathcal{T}}
\newcommand{\Ucal}{\mathcal{U}}
\newcommand{\Vcal}{\mathcal{V}}
\newcommand{\Wcal}{\mathcal{W}}
\newcommand{\Xcal}{\mathcal{X}}
\newcommand{\Ycal}{\mathcal{Y}}
\newcommand{\Zcal}{\mathcal{Z}}

\newcommand{\Abbb}{\mathbb{A}}
\newcommand{\Bbbb}{\mathbb{B}}
\newcommand{\Cbbb}{\mathbb{C}}
\newcommand{\Dbbb}{\mathbb{D}}
\newcommand{\Ebbb}{\mathbb{E}}
\newcommand{\Fbbb}{\mathbb{F}}
\newcommand{\Gbbb}{\mathbb{G}}
\newcommand{\Hbbb}{\mathbb{H}}
\newcommand{\Ibbb}{\mathbb{I}}
\newcommand{\Jbbb}{\mathbb{J}}
\newcommand{\Kbbb}{\mathbb{K}}
\newcommand{\Lbbb}{\mathbb{L}}
\newcommand{\Mbbb}{\mathbb{M}}
\newcommand{\Nbbb}{\mathbb{N}}
\newcommand{\Obbb}{\mathbb{O}}
\newcommand{\Pbbb}{\mathbb{P}}
\newcommand{\Qbbb}{\mathbb{Q}}
\newcommand{\Rbbb}{\mathbb{R}}
\newcommand{\Sbbb}{\mathbb{S}}
\newcommand{\Tbbb}{\mathbb{T}}
\newcommand{\Ubbb}{\mathbb{U}}
\newcommand{\Vbbb}{\mathbb{V}}
\newcommand{\Wbbb}{\mathbb{W}}
\newcommand{\Xbbb}{\mathbb{X}}
\newcommand{\Ybbb}{\mathbb{Y}}
\newcommand{\Zbbb}{\mathbb{Z}}

\newcommand{\finsub}{\subseteq_{\text{fin}}}% I \finsub \Nbbb_{\geq1}: I is a finite set of colors

\RequirePackage{stmaryrd}
\usepackage{textcomp}
\DeclareUnicodeCharacter{2286}{\subseteq}
\DeclareUnicodeCharacter{2192}{\ifmmode\to\else\textrightarrow\fi}
\DeclareUnicodeCharacter{2203}{\ensuremath\exists}
\DeclareUnicodeCharacter{183}{\cdot}
\DeclareUnicodeCharacter{2200}{\forall}
\DeclareUnicodeCharacter{2264}{\leq}
\DeclareUnicodeCharacter{2265}{\geq}
\DeclareUnicodeCharacter{8614}{\mathbin{\mapsto}}
\DeclareUnicodeCharacter{8656}{\Leftarrow}
\DeclareUnicodeCharacter{8657}{\Uparrow}
\DeclareUnicodeCharacter{8658}{\Rightarrow}
\DeclareUnicodeCharacter{8659}{\Downarrow}
\DeclareUnicodeCharacter{8669}{\rightsquigarrow}
\newcommand{\eqdef}{\stackrel{{\scriptsize\rm def}}{=}}
\DeclareUnicodeCharacter{8797}{\eqdef}
\DeclareUnicodeCharacter{8870}{\vdash}
\DeclareUnicodeCharacter{8873}{\Vdash}
\DeclareUnicodeCharacter{22A7}{\models}
\DeclareUnicodeCharacter{9121}{\lceil}
\DeclareUnicodeCharacter{9123}{\lfloor}
\DeclareUnicodeCharacter{9124}{\rceil}
\DeclareUnicodeCharacter{2208}{\in}
\DeclareUnicodeCharacter{9126}{\rfloor}
\DeclareUnicodeCharacter{9655}{\triangleright}
\DeclareUnicodeCharacter{9665}{\triangleleft}
\DeclareUnicodeCharacter{9671}{\diamond}
\DeclareUnicodeCharacter{9675}{\circ}
\DeclareUnicodeCharacter{10178}{\bot}
\DeclareUnicodeCharacter{10214}{} % needs stmaryrd
\DeclareUnicodeCharacter{10215}{} % needs stmaryrd
\DeclareUnicodeCharacter{10229}{\longleftarrow}
\DeclareUnicodeCharacter{10230}{\longrightarrow}
\DeclareUnicodeCharacter{10231}{\longleftrightarrow}
\DeclareUnicodeCharacter{10232}{\Longleftarrow}
\DeclareUnicodeCharacter{10233}{\Longrightarrow}
\DeclareUnicodeCharacter{10234}{\Longleftrightarrow}
\DeclareUnicodeCharacter{10236}{\longmapsto}
\DeclareUnicodeCharacter{10238}{\Longmapsto} % needs stmaryrd
\DeclareUnicodeCharacter{10503}{\Mapsto}    % needs stmaryrd
\DeclareUnicodeCharacter{10971}{\mathrel{\not\hspace{-0.2em}\cap}}
\DeclareUnicodeCharacter{65294}{\ldotp}
\DeclareUnicodeCharacter{65372}{\mid}

\newcommand{\YMB}{\raisebox{-3pt}{\begin{tikzpicture}
  \pgftext{\includegraphics[width=12pt]{Figures/ymb.pdf}};
\end{tikzpicture}}}
\newcommand{\YGB}{\raisebox{-3pt}{\begin{tikzpicture}
  \pgftext{\includegraphics[width=12pt]{Figures/ygb.pdf}};
\end{tikzpicture}}}
\newcommand{\YMG}{\raisebox{-3pt}{\begin{tikzpicture}
  \pgftext{\includegraphics[width=12pt]{Figures/ymg.pdf}};
\end{tikzpicture}}}
\newcommand{\GMB}{\raisebox{-3pt}{\begin{tikzpicture}
  \pgftext{\includegraphics[width=12pt]{Figures/gmb.pdf}};
\end{tikzpicture}}}

\newcommand{\NodeYM}{\raisebox{-3pt}{\begin{tikzpicture}
  \pgftext{\includegraphics[width=11.5pt]{Figures/ym.pdf}};
\end{tikzpicture}}}
\newcommand{\NodeYB}{\raisebox{-3pt}{\begin{tikzpicture}
  \pgftext{\includegraphics[width=11.5pt]{Figures/yb.pdf}};
\end{tikzpicture}}}
\newcommand{\NodeYG}{\raisebox{-3pt}{\begin{tikzpicture}
  \pgftext{\includegraphics[width=11.5pt]{Figures/yg.pdf}};
\end{tikzpicture}}}
\newcommand{\NodeYP}{\raisebox{-3pt}{\begin{tikzpicture}
  \pgftext{\includegraphics[width=11.5pt]{Figures/yp.pdf}};
\end{tikzpicture}}}

\newcommand{\NodeMB}{\raisebox{-3pt}{\begin{tikzpicture}
  \pgftext{\includegraphics[width=11.5pt]{Figures/mb.pdf}};
\end{tikzpicture}}}
\newcommand{\NodeMG}{\raisebox{-3pt}{\begin{tikzpicture}
  \pgftext{\includegraphics[width=11.5pt]{Figures/mg.pdf}};
\end{tikzpicture}}}
\newcommand{\NodeMP}{\raisebox{-3pt}{\begin{tikzpicture}
  \pgftext{\includegraphics[width=11.5pt]{Figures/mp.pdf}};
\end{tikzpicture}}}

\newcommand{\NodeBG}{\raisebox{-3pt}{\begin{tikzpicture}
  \pgftext{\includegraphics[width=11.5pt]{Figures/bg.pdf}};
\end{tikzpicture}}}
\newcommand{\NodeBP}{\raisebox{-3pt}{\begin{tikzpicture}
  \pgftext{\includegraphics[width=11.5pt]{Figures/bp.pdf}};
\end{tikzpicture}}}

\newcommand{\NodeGP}{\raisebox{-3pt}{\begin{tikzpicture}
  \pgftext{\includegraphics[width=11.5pt]{Figures/gp.pdf}};
\end{tikzpicture}}}

\newcommand{\bpw}{\mathsf{bpw}\xspace}%biconnected pathwidth
\newcommand{\pw}{\mathsf{pw}\xspace}%pathwidth
\newcommand{\tw}{\mathsf{tw}\xspace}%treewidth
\newcommand{\td}{\mathsf{td}\xspace}%treedepth
\newcommand{\size}{\mathsf{size}\xspace}%size
\newcommand{\obs}{\mathsf{obs}\xspace}%obs
\newcommand{\cobs}{\mathsf{cobs}\xspace}%obs
\newcommand{\hw}{\mathsf{hw}\xspace}%Hadwiger number
\newcommand{\torso}{\mathsf{torso}\xspace}%torso
\newcommand{\aall}{\Acal_{\text{all}}\xspace}
\newcommand{\ttw}{\mathsf{ttw}\xspace}
\newcommand{\expp}{\mathsf{exp}\xspace}
\newcommand{\funp}{\mathsf{f}\xspace}
\newcommand{\poly}{\mathsf{poly}\xspace}
\newcommand{\tfolio}{\mathsf{tfolio}\xspace}
\newcommand{\rep}{\mathsf{rep}\xspace}
\newcommand{\bg}{\mathsf{bg}\xspace}
\newcommand{\sbg}{\mathsf{sbg}\xspace}
\newcommand{\MSO}{\mbox{\sf MSO}\xspace}
\newcommand{\CMSO}{\mbox{\sf CMSO}\xspace}
\newcommand{\FO}{{\sf FO}\xspace}
\newcommand{\DP}{{\sf dp}\xspace}
\newcommand{\SDP}{{\sf sdp}\xspace}
\newcommand{\FOLDP}{\FOL[\tau{\normalfont+}\DP]\xspace}
\newcommand{\NTMC}{{\sf EM}\xspace}
\newcommand{\bool}{{\bf PB}\xspace}
\newcommand{\labels}[1]{\label{#1}}
\newcommand{\sig}{{\sf sig}\xspace}
\newcommand{\bd}{{\sf bd}\xspace}
\newcommand{\yes}{{\sf yes}\xspace}
\newcommand{\no}{{\sf no}\xspace}
\newcommand{\ann}{{\sf ann}\xspace}
\newcommand{\inter}{{\sf int}\xspace}
\newcommand{\remove}[1]{}
\newcommand{\removed}[1]{#1}
\newcommand{\bigmid}{\;\big|\;\xspace}
\newcommand{\capall}{\pmb{\pmb{\bigcap}}\xspace}
\newcommand{\disk}{{\sf disk}\xspace}
\newcommand{\dehe}{{\sf dehe}\xspace}
\newcommand{\cae}{{\sf cae}\xspace}
\newcommand{\comp}{{\sf comp}\xspace}
\newcommand{\col}{{\sf col}\xspace}
\newcommand{\perim}{{\sf perim}\xspace}
\newcommand{\bnd}{{\sf bnd}\xspace}
\newcommand{\thw}{\mathsf{thw}}
\newcommand{\marg}[1]{\marginpar{\footnotesize #1}}
\newcommand{\pretp}{\preceq_{\sf tm}}
\newcommand{\prem}{\preceq_{\sf m}}
\newcommand{\defterm}[1]{\emph{#1}}
\newcommand{\bigO}[1]{\mathcal O(#1)}
\newcommand{\NP}{{\sf NP}\xspace}
\newcommand{\FPT}{{\sf FPT}\xspace}
\newcommand{\XP}{{\sf XP}\xspace}
\newcommand{\W}{{\sf W}\xspace}
\newcommand{\frR}{{\mathfrak{R}}}
\newcommand{\bidim}{{\sf bidim}\xspace}
\newcommand{\sbidim}{{\sf sbidim}\xspace}
\newcommand{\excl}{{\sf excl}\xspace}
\newcommand{\imprint}{{\sf imp}}
\newcommand{\tritri}{\mbox{\footnotesize $\,\blacktriangleright\,$}}
\newcommand{\p}{\mathsf{p}\xspace}
\newcommand{\q}{\mathsf{q}\xspace}
\newcommand{\lt}[1]{\textlatin{#1}}
\newcommand{\cy}[1]{\foreignlanguage{russian}{#1}}
\newcommand{\gr}[1]{\foreignlanguage{greek}{#1}}
\newcommand{\injection}{{\bf inj}\xspace}
\newcommand{\repact}{\mathcal{L}}
\newcommand{\gall}{\mathcal{G}_{{\text{\rm  \textsf{all}}}}}
\newcommand{\opt}{\mathsf{opt}\xspace}
\newcommand{\vc}{\mathsf{vc}\xspace}
\newcommand{\rg}{\midnightblue{\textsl{rg}}\xspace}
\newcommand{\sg}{\midnightblue{\textsl{sg}}\xspace}
\newcommand{\rc}{\midnightblue{\textsl{rc}}\xspace}

\newcommand{\wall}{\midnightblue{\textsl{wall}}\xspace}
\newcommand{\apex}{\midnightblue{\textsl{apex}}\xspace}
\newcommand{\breadth}{\midnightblue{\textsl{breadth}}\xspace}
\newcommand{\link}{\midnightblue{\textsl{link}}\xspace}
\newcommand{\depth}{\midnightblue{\textsl{depth}}\xspace}
\newcommand{\vortex}{\midnightblue{\textsl{vortex}}\xspace}
\newcommand{\hfw}{\midnightblue{\textsl{hfw}}\xspace}
\newcommand{\genus}{\midnightblue{\textsl{genus}}\xspace}

\newcommand{\say}[1]{``#1''}
\newcommand{\bdim}{{\sf bdim}\xspace}%Bidimensionality

\newcommand{\adhesion}{\midnightblue{\textsl{adhesion}}\xspace}
\newcommand{\palette}{\mathsf{palette}\xspace}
\newcommand{\cover}{\mathsf{cover}\xspace}
\newcommand{\folio}{\mathsf{folio}\xspace}
\newcommand{\pack}{\mathsf{pack}\xspace}
\providecommand{\eg}{\mathsf{eg}}% Euler genus
\newcommand{\rh}{\mathsf{rh}\xspace}
\newcommand{\srh}{\mathsf{srh}\xspace}
\newcommand{\srtw}{\mathsf{srtw}\xspace}
\newcommand{\sbsg}{\mathsf{sbsg}\xspace}
\newcommand{\rtw}{\mathsf{rtw}\xspace}

\usetikzlibrary{decorations.pathreplacing}

\graphicspath{{Figures/}}
\makeatletter
\providecommand*{\input@path}{}
\edef\input@path{{Figures/}}
\newcommand{\figmode}{}
\newcommand{\figscale}[2]{\expandafter\gdef\csname uf@s@#1\endcsname{#2}}
\newcommand{\usefigure}[2]{%
  \edef\uf@z{#2}%
  \edef\uf@m{\figmode}%
  \ifx\uf@m\@empty\else\let\uf@z\uf@m\fi
  \ifnum\uf@z=1\relax
    \input{#1.tex}%
  \else
    \@ifundefined{uf@s@#1}{\def\uf@f{1}}{\edef\uf@f{\csname uf@s@#1\endcsname}}%
    \scalebox{\uf@f}{\includegraphics{#1.jpg}}%
  \fi}
\makeatother
\figscale{fig_1segregatedGridsPack}{0.99979}
\figscale{fig_C4Interval}{0.99984}
\figscale{fig_Case4EP}{0.99987}
\figscale{fig_ColorfulMinorIntro}{0.99937}
\figscale{fig_EPIntro}{0.99986}
\figscale{fig_K3K13Interval}{0.99980}
\figscale{fig_OneEdge}{0.99978}
\figscale{fig_Oq1}{0.99967}
\figscale{fig_Oq2}{0.99945}
\figscale{fig_SegregatedGrids}{0.99911}
\figscale{fig_Tgrids}{0.99918}
\figscale{fig_Thosting}{0.99994}
\figscale{fig_UsGrids}{0.99999}
\figscale{fig_UsGrids_change}{0.99961}
\figscale{fig_componentwiseBicolordGrids}{0.99947}
\figscale{fig_construction1}{0.99909}
\figscale{fig_construction2}{0.99909}
\figscale{fig_halfIntegral_segregated}{0.99973}
\figscale{fig_vortexConstruction}{0.99944}
\figscale{manu_ole}{0.99975}
\definecolor{clY}{RGB}{255,180,  0}     % yellow / amber
\definecolor{clM}{RGB}{255,  0,180}     % magenta
\definecolor{clB}{RGB}{  0,210,255}     % blue
\definecolor{clG}{RGB}{ 48,202,  0}     % green
\definecolor{clP}{RGB}{149,  0,255}     % purple
\definecolor{clYpale}{RGB}{255,210,103} % pale amber (highlight boxes)
\definecolor{clMpale}{RGB}{255, 99,210} % pale magenta
\definecolor{clGpale}{RGB}{136,255, 99} % pale green
\definecolor{plategray}{gray}{0.827}

\newcommand{\halfv}[5]{%
  \begin{scope}
    \clip (#1,#2) circle (#3);
    \fill[#4] (#1-#3,#2-#3) rectangle (#1,#2+#3);
    \fill[#5] (#1,#2-#3) rectangle (#1+#3,#2+#3);
  \end{scope}
  \draw[line width=1.5pt] (#1,#2-#3) -- (#1,#2+#3);
  \draw[line width=1.5pt] (#1,#2) circle (#3);}

\newcommand{\thirdv}[6]{%
  \fill[#5] (#1,#2) -- ++(-30:#3) arc[start angle=-30,end angle= 90,radius=#3] -- cycle;
  \fill[#4] (#1,#2) -- ++( 90:#3) arc[start angle= 90,end angle=210,radius=#3] -- cycle;
  \fill[#6] (#1,#2) -- ++(210:#3) arc[start angle=210,end angle=330,radius=#3] -- cycle;
  \foreach \a in {90,210,330}{\draw[line width=1.5pt] (#1,#2) -- ++(\a:#3);}
  \draw[line width=1.5pt] (#1,#2) circle (#3);}
\definecolor{clY}{RGB}{255,180,  0}
\definecolor{clM}{RGB}{255,  0,180}
\definecolor{clB}{RGB}{  0,210,255}
\definecolor{clG}{RGB}{ 48,202,  0}
\definecolor{clP}{RGB}{149,  0,255}

\newcommand{\NodeTwoC}[2]{\raisebox{-3pt}{\begin{tikzpicture}[x=1pt,y=1pt]
  \begin{scope}
    \clip (0,0) circle (5.50);
    \fill[#1] (-5.5,0) rectangle (5.5,5.5);
    \fill[#2] (-5.5,-5.5) rectangle (5.5,0);
  \end{scope}
  \draw[line width=.50pt] (-5.07,0) -- (5.07,0);
  \draw[line width=.86pt] (0,0) circle (5.07);
\end{tikzpicture}}}

\newcommand{\NodeThreeC}[3]{\raisebox{-3pt}{\begin{tikzpicture}[x=1pt,y=1pt]
  \fill[#2] (0,0) -- (-30:5.74) arc[start angle=-30,end angle= 90,radius=5.74] -- cycle;
  \fill[#1] (0,0) -- ( 90:5.74) arc[start angle= 90,end angle=210,radius=5.74] -- cycle;
  \fill[#3] (0,0) -- (210:5.74) arc[start angle=210,end angle=330,radius=5.74] -- cycle;
  \foreach \a in {90,210,330}{\draw[line width=.53pt] (0,0) -- (\a:5.29);}
  \draw[line width=.90pt] (0,0) circle (5.29);
\end{tikzpicture}}}

\renewcommand{\NodeYM}{\NodeTwoC{clY}{clM}}
\renewcommand{\NodeYB}{\NodeTwoC{clY}{clB}}
\renewcommand{\NodeYG}{\NodeTwoC{clY}{clG}}
\renewcommand{\NodeYP}{\NodeTwoC{clY}{clP}}
\renewcommand{\NodeMB}{\NodeTwoC{clM}{clB}}
\renewcommand{\NodeMG}{\NodeTwoC{clM}{clG}}
\renewcommand{\NodeMP}{\NodeTwoC{clM}{clP}}
\renewcommand{\NodeBG}{\NodeTwoC{clB}{clG}}
\renewcommand{\NodeBP}{\NodeTwoC{clB}{clP}}
\renewcommand{\NodeGP}{\NodeTwoC{clG}{clP}}

\renewcommand{\YMB}{\NodeThreeC{clY}{clM}{clB}}
\renewcommand{\YGB}{\NodeThreeC{clY}{clG}{clB}}
\renewcommand{\YMG}{\NodeThreeC{clY}{clM}{clG}}
\renewcommand{\GMB}{\NodeThreeC{clG}{clM}{clB}}
 \newcommand{\rev}[1]{%
\todo[linecolor=red,bordercolor=white,color=BrightPink,
 backgroundcolor=Blue!20,size=\scriptsize]{%
 \textsf{\scriptsize REV: #1}}}

 \newcommand{\ans}[1]{%
 \todo[linecolor=red,bordercolor=white,color=BrilliantRose,
 backgroundcolor=BrilliantRose!50,size=\scriptsize]{%
 \textsf{\scriptsize ANS: #1}}}

 \newboolean{appear}%%%%%%%%%% Change the appear to "true"/"false" in order to see/not see all the comments!

\setboolean{appear}{true}

 \ifthenelse{\boolean{appear}}{%

 \newcommand{\SW}[1]{%
 \todo[linecolor=red,bordercolor=white,color=Magenta,
 backgroundcolor=CornflowerBlue!40,textcolor=BrightPink,
 size=\footnotesize]{\textsf{SW: #1}}}

 \newcommand{\SWin}[1]{%
 \todo[inline,bordercolor=white,color=Magenta,
 backgroundcolor=CornflowerBlue!40,textcolor=BrightPink,
 size=\footnotesize]{\textsf{SW: #1}}}

 \newcommand{\EP}[1]{%
 \todo[linecolor=red,bordercolor=white,color=DeepCarrotOrange,
 backgroundcolor=BananaYellow!40,textcolor=DeepCarrotOrange,
 size=\footnotesize]{\textsf{EP: #1}}}

 \newcommand{\EPin}[1]{%
 \todo[inline,bordercolor=white,color=DeepCarrotOrange,
 backgroundcolor=BananaYellow!40,textcolor=DeepCarrotOrange,
 size=\footnotesize]{\textsf{EP: #1}}}

 \newcommand{\sed}[1]{%
 \todo[linecolor=red,bordercolor=white,color=BrilliantRose,
 backgroundcolor=BrilliantRose!50,size=\footnotesize]{%
 \textsf{\footnotesize SED: #1}}}

 \newcommand{\VP}[1]{%
 \todo[linecolor=red,bordercolor=white,color=purple!40,
 backgroundcolor=cyan!20,textcolor=magenta,
 size=\footnotesize]{\sf Vagelis: #1}}

 \newcommand{\VPin}[1]{%
 \todo[inline,bordercolor=white,color=purple!40,
 backgroundcolor=cyan!20,textcolor=magenta,
 size=\footnotesize]{\sf Vagelis: #1}}

 \newcommand{\claude}[1]{%
 \todo[linecolor=AO,bordercolor=white,color=AO,
 backgroundcolor=AppleGreen!25,textcolor=BritishRacingGreen,
 size=\scriptsize]{%
 \raggedright\hyphenpenalty=0\exhyphenpenalty=0\tolerance=9999
 \emergencystretch=2em
 \textsf{\scriptsize CLAUDE: #1}}}

 \newcommand{\claudein}[1]{%
 \todo[inline,bordercolor=white,color=AO,
 backgroundcolor=AppleGreen!25,textcolor=BritishRacingGreen,
 size=\scriptsize]{%
 \textsf{\scriptsize CLAUDE: #1}}}

}{%

 \newcommand{\SW}[1]{}
 \newcommand{\SWin}[1]{}
 \newcommand{\EP}[1]{}
 \newcommand{\EPin}[1]{}
 \newcommand{\sed}[1]{}
 \newcommand{\VP}[1]{}
 \newcommand{\VPin}[1]{}
 \newcommand{\claude}[1]{}
 \newcommand{\claudein}[1]{}

}

\ifthenelse{\boolean{appear}}{\newcommand{\rred}[1]{\red{#1}}}{\newcommand{\rred}[1]{#1}}

\definecolor{framepurple}{RGB}{160,32,240}
\definecolor{clYbg}{rgb}{1,0.968,0.890}
\definecolor{hc0}{rgb}{0.75,0.396,1.0}
\definecolor{kvpurple}{RGB}{160,32,240}
\definecolor{KVc0}{rgb}{0.63,0.12,0.94}
\definecolor{KVc1}{rgb}{0,0,0}
\definecolor{KVc2}{rgb}{1,0,0.71}
\definecolor{KVc3}{rgb}{0,0.82,1}
\definecolor{2K2Vortexc0}{rgb}{0.97,0.66,0.44}
\definecolor{2K2Vortexc1}{rgb}{0.96,0.46,0.13}
\definecolor{2K2Vortexc2}{rgb}{1,1,1}
\definecolor{2K2Vortexc3}{rgb}{0,0,0}
\definecolor{2K2Vortexc4}{rgb}{0.03,0.02,0.02}
\definecolor{2K2Vortexc5}{rgb}{0.47,0.32,0.63}
\definecolor{2K2Vortexc6}{rgb}{0.89,0.23,0.59}
\definecolor{2K2Vortexc7}{rgb}{0.26,0.78,0.95}
\definecolor{2K2Vortexc8}{rgb}{0.99,0.7,0.08}
\definecolor{2K2Vortexc9}{rgb}{0.29,0.72,0.28}
\definecolor{Figure11c0}{rgb}{0.32,0.31,0.31}
\definecolor{Figure11c1}{rgb}{0.58,0.59,0.6}
\definecolor{Figure11c2}{rgb}{0.26,0.25,0.25}
\definecolor{Figure11c3}{rgb}{0.26,0.25,0.25}
\definecolor{Figure11c4}{rgb}{0.89,0.23,0.59}
\definecolor{Figure11c5}{rgb}{1,1,1}
\definecolor{Figure11c6}{rgb}{0,0,0}
\definecolor{Figure11c7}{rgb}{0.62,0.8,0.31}
\definecolor{Figure12c0}{rgb}{0.31,0.3,0.3}
\definecolor{Figure12c1}{rgb}{0.31,0.31,0.3}
\definecolor{Figure12c2}{rgb}{0.61,0.43,0.69}
\definecolor{Figure12c3}{rgb}{0.99,0.7,0.08}
\definecolor{Figure12c4}{rgb}{0.26,0.78,0.95}
\definecolor{Figure12c5}{rgb}{0.89,0.23,0.59}
\definecolor{Figure12c6}{rgb}{0,0,0}
\definecolor{Figure12c7}{rgb}{1,1,1}
\definecolor{Figure12c8}{rgb}{0.47,0.32,0.63}
\definecolor{Figure12c9}{rgb}{0.62,0.8,0.31}
\definecolor{Figure13c0}{rgb}{0,0,0}
\definecolor{Figure13c1}{rgb}{1,1,1}
\definecolor{Figure13c2}{rgb}{0.47,0.32,0.63}
\definecolor{Figure13c3}{rgb}{0.31,0.3,0.3}
\definecolor{Figure13c4}{rgb}{0.26,0.78,0.95}
\definecolor{Figure13c5}{rgb}{0.31,0.3,0.3}
\definecolor{Figure13c6}{rgb}{0.89,0.23,0.59}
\definecolor{Figure13c7}{rgb}{0.62,0.8,0.31}
\definecolor{Vortex1c0}{rgb}{0.32,0.31,0.31}
\definecolor{Vortex1c1}{rgb}{0.99,0.8,0.7}
\definecolor{Vortex1c2}{rgb}{0.58,0.59,0.6}
\definecolor{Vortex1c3}{rgb}{0.26,0.25,0.25}
\definecolor{Vortex1c4}{rgb}{0.26,0.25,0.25}
\definecolor{Vortex1c5}{rgb}{0.89,0.23,0.59}
\definecolor{Vortex1c6}{rgb}{0.94,0.35,0.13}
\definecolor{Vortex1c7}{rgb}{1,1,1}
\definecolor{Vortex1c8}{rgb}{0,0,0}
\definecolor{Vortex1c9}{rgb}{0.62,0.8,0.31}
\definecolor{Vortex2c0}{rgb}{0.31,0.3,0.3}
\definecolor{Vortex2c1}{rgb}{0.31,0.31,0.3}
\definecolor{Vortex2c2}{rgb}{0.61,0.43,0.69}
\definecolor{Vortex2c3}{rgb}{0.99,0.8,0.7}
\definecolor{Vortex2c4}{rgb}{0.94,0.35,0.13}
\definecolor{Vortex2c5}{rgb}{0.99,0.7,0.08}
\definecolor{Vortex2c6}{rgb}{0.26,0.78,0.95}
\definecolor{Vortex2c7}{rgb}{0.89,0.23,0.59}
\definecolor{Vortex2c8}{rgb}{0,0,0}
\definecolor{Vortex2c9}{rgb}{1,1,1}
\definecolor{Vortex2c10}{rgb}{0.47,0.32,0.63}
\definecolor{Vortex2c11}{rgb}{0.62,0.8,0.31}
\definecolor{Vortex3c0}{rgb}{0.31,0.3,0.3}
\definecolor{Vortex3c1}{rgb}{0.31,0.31,0.3}
\definecolor{Vortex3c2}{rgb}{0.61,0.43,0.69}
\definecolor{Vortex3c3}{rgb}{0.99,0.8,0.7}
\definecolor{Vortex3c4}{rgb}{0.94,0.35,0.13}
\definecolor{Vortex3c5}{rgb}{0.99,0.7,0.08}
\definecolor{Vortex3c6}{rgb}{0.26,0.78,0.95}
\definecolor{Vortex3c7}{rgb}{0.89,0.23,0.59}
\definecolor{Vortex3c8}{rgb}{0,0,0}
\definecolor{Vortex3c9}{rgb}{1,1,1}
\definecolor{Vortex3c10}{rgb}{0.47,0.32,0.63}
\definecolor{Vortex3c11}{rgb}{0.62,0.8,0.31}
\definecolor{cM}{RGB}{255,  0,180}  % magenta
\definecolor{cC}{RGB}{  0,210,255}  % cyan
\definecolor{cO}{RGB}{255,180,  0}  % orange
\definecolor{cbranch}{RGB}{217,217,217}  % branch-set halo
\definecolor{uA}{RGB}{234, 51,176}  % colour 1  (pink)
\definecolor{uB}{RGB}{ 95,207,250}  % colour 2  (light blue)
\definecolor{uC}{RGB}{244,183, 63}  % colour 3  (amber)
\definecolor{uD}{RGB}{101,199, 59}  % colour 4  (green)
\definecolor{uE}{RGB}{155,110,240}  % colour 5  (violet)
\definecolor{pkamber}{RGB}{255,180,  0}
\definecolor{pkgray}{gray}{0.851}
\definecolor{pkmagenta}{RGB}{255,  0,181}
\definecolor{pkcyan}{RGB}{  0,210,255}
\definecolor{mo0}{rgb}{0.827,0.827,0.827}
\definecolor{mo1}{rgb}{0.0,0.0,0.545}
\definecolor{mo2}{rgb}{0.843,0.271,0.674}
\definecolor{mo3}{rgb}{0.49,0.8,0.965}
\definecolor{mo4}{rgb}{0.921,0.729,0.341}
\definecolor{mo5}{rgb}{0.494,0.772,0.322}
\definecolor{mo6}{rgb}{1.0,0.0,0.0}
\definecolor{HotMagenta}{RGB}{255, 29,206}      % color 1  (magenta)
\definecolor{CornflowerBlue}{RGB}{100,149,237}  % color 2  (blue)
\definecolor{ChromeYellow}{RGB}{255,167,  0}    % color 3  (yellow)
\definecolor{AppleGreen}{RGB}{141,182,  0}      % color 4  (green)
\definecolor{Violet}{RGB}{127,  0,255}          % color 5  (violet)
\colorlet{violet}{Violet}
\colorlet{uA}{HotMagenta}\colorlet{uB}{CornflowerBlue}\colorlet{uC}{ChromeYellow}
\colorlet{uD}{AppleGreen}\colorlet{uE}{Violet}
\colorlet{clM}{HotMagenta}\colorlet{clB}{CornflowerBlue}\colorlet{clY}{ChromeYellow}
\colorlet{clG}{AppleGreen}\colorlet{clP}{Violet}
\colorlet{cM}{HotMagenta}\colorlet{cC}{CornflowerBlue}\colorlet{cO}{ChromeYellow}
\colorlet{pkmagenta}{HotMagenta}\colorlet{pkcyan}{CornflowerBlue}\colorlet{pkamber}{ChromeYellow}
\colorlet{mo2}{HotMagenta}\colorlet{mo3}{CornflowerBlue}\colorlet{mo4}{ChromeYellow}
\colorlet{mo5}{AppleGreen}
\colorlet{clMpale}{HotMagenta!45}\colorlet{clYpale}{ChromeYellow!45}\colorlet{clGpale}{AppleGreen!45}

\title{The Erd\H{o}s-P\'osa Property for Colorful Minors\thanks{This paper is an extended version of some part of the combinatorial results of \cite{ProtopapasTW2025Colorful}, presented  in ICALP 2026.}}

 \author{%
 Evangelos Protopapas\thanks{Faculty of Mathematics, Informatics and Mechanics, University of Warsaw, Poland.\\ Email: \texttt{eprotopapas@mimuw.edu.pl}}~$^{,}$\thanks{Supported by the ERC project BUKA (n°\! 101126229) and the French-German Collaboration ANR/DFG Project UTMA (ANR-20-CE92-0027).}\and 
 Dimitrios M. Thilikos\thanks{LIRMM, Univ Montpellier, CNRS, Montpellier, France.\\ Email: \texttt{sedthilk@thilikos.info}.}~$^{,}$\thanks{
Supported by the Franco-Norwegian project PHC AURORA 2024-25 (Projet n°\! 51260WL) and the French National Research Agency (ANR) under project GODASse ANR-24-CE48-4377 and under the France 2030 grant reference number ANR-24-RRII-0002 operated by the Inria Quadrant Program.}\and 
 Sebastian Wiederrecht\thanks{KAIST, School of Computing, Daejeon, South Korea. \\ Email:
 \texttt{wiederrecht@kaist.ac.kr}}~$^{,}$\thanks{Supported by the Institute for Basic Science (IBS-R029-C1)}}

\date{}

\begin{document}

\maketitle

\begin{abstract}
\noindent A \emph{colorful graph} is a pair $(G,\chi)$ where $G$ is a graph and $\chi$ assigns to each vertex of $G$ a   finite set of colors.
The \emph{colorful minor} relation enhances the minor relation by merging color sets along contractions and by allowing the removal of colors; it generalizes rooted minors and models problems on graphs with several, possibly overlapping, annotated vertex sets.
A graph has the Erd\H{o}s-P\'osa property for minors if and only if it is planar, by a classical theorem of Robertson and Seymour.
In this work we determine, for the colorful minor relation, exactly which colorful graphs have the Erd\H{o}s-P\'osa property.
Our characterization takes three equivalent forms.
The first is \emph{structural}: the colorful graphs with the property are those that can be drawn with all their colored vertices on one face and whose colors are, in a precise sense, laid out along that face without interleaving.
The second is given by an \emph{obstruction set}: they are those excluding every member of an explicit infinite family $\Ocal,$ of which only $\bigO{|I|^{4}}$ members have  colors that are  a subset of $I,$ for every finite set $I$ of colors.
The third is \emph{grid-like}: they are exactly the colorful minors of unions of particular families of segregated grids, the colorful analogues of the grids that drive the classical proof.
\end{abstract}

\noindent\textbf{Keywords:} Graph Minors, Colorful Minors, Annotated Graphs, Rooted Minors, Erd\H{o}s-P\'osa property, Structural Graph Theory, Obstruction sets.

\thispagestyle{empty}

\newpage
\thispagestyle{empty}

\tableofcontents
\thispagestyle{empty}
\newpage

\section{Introduction}
Packing and covering are among the oldest opposing questions in combinatorics and the theorems that tie them together are among the most useful. This paper is about one such tie, in a setting where the objects being packed and covered are not merely subgraphs but subgraphs anchored to prescribed parts of the host. A graph may come with a set of terminals to be connected, a set of faces to be covered, a set of vertices that a cycle must visit; in each case the structures one wants to find are constrained not only in shape but in where and how  they attach. Colorful graphs are a way of capturing all such constraints at once: each vertex carries a finite set of colors, several annotated sets may overlap and a pattern is found in a host only if every color it prescribes is traced back. The question we settle is which patterns, in this generality, admit a packing-covering duality.

\subsection{The Erd\H{o}s-P{\'o}sa property}
A family of subgraphs of a graph may be viewed from two perspectives: one may look for many copies of its members that are pairwise disjoint, a \emph{packing}, or for few vertices that meet all such copies, a \emph{covering}.
A covering is at least as large as a packing and the family is said to have the \emph{Erd\H{o}s-P{\'o}sa property} when the reverse inequality holds up to a function, that is, when a bound on the packing number implies a bound on the covering number.
The name comes from the theorem of Erd\H{o}s and P{\'o}sa \cite{ErdosP1965independent}, later sharpened by Simonovits \cite{Simonovits1967New}, asserting that a graph either contains $k$ pairwise vertex-disjoint cycles or has a set of $\bigO{k\log k}$ vertices meeting all of its cycles.
Since a cycle is a $K_{3}$-minor model, the natural generalization is to ask, for a fixed graph $H,$ when the $H$-minor models of a graph enjoy the same duality.
Robertson and Seymour answered this in the fifth paper of the Graph Minors series \cite{RobertsonS1986Grapha}: the models of $H$ have the Erd\H{o}s-P{\'o}sa property if and only if $H$ is planar.
The two directions of their answer are of very different natures.
The positive one is a consequence of the Grid Theorem, as a graph of large treewidth contains a large grid and hence many disjoint copies of any fixed planar $H,$ while a graph of small treewidth can be treated  along a tree decomposition.
The negative one is a construction: for non-planar $H$ one embeds $H$ in a surface of minimum Euler genus  and superimposes many copies of the resulting drawing, so that the copies interlock and no small vertex set can destroy them all.
Both mechanisms reappear, in a colorful form, in the present paper.

The property has since been studied in a large number of variants (see \cite{RaymondT2017Recent} for an overview).
Among the directions closest to this paper, the \emph{half-integral} relaxation, in which a packing may use every vertex twice, was conjectured by Thomas to hold for every $H$ and was confirmed by Liu \cite{Liu2022Packing}, who proved it even for topological minors; the threshold at which half-integrality becomes necessary has begun to be delineated by Paul, Protopapas, Thilikos, and Wiederrecht \cite{PaulPTW2024Delineating,PaulPTW2024Obstructions}.

\subsection{Annotated versions}
The variant that this paper generalizes asks for models that are anchored: a set of vertices of the host is distinguished and the minor models are required to interact with it.
Kakimura, Kawarabayashi, and Marx \cite{kakimuraKM2011} proved that cycles through a prescribed vertex set have the Erd\H{o}s-P{\'o}sa property, and Pontecorvi and Wollan \cite{PontecorviW2012Disjoint} gave a version with a bound independent of the prescribed set.
For minors rather than cycles, Kwon and Marx \cite{KwonM2019ErdosPosa} studied the property for minor models with prescribed vertex sets, and Bruhn, Joos, and Schaudt \cite{BruhnJS2021Labelled} characterized the $2$-connected labeled minors that have it.
All of these are instances of a single setting, the colorful minor relation introduced by Protopapas, Thilikos, and Wiederrecht \cite{ProtopapasTW2025Colorful}. A colorful graph carries, at each of its vertices, a finite set of colors. A colorful minor must trace every color of the pattern back to the host and the annotated settings above are the cases in which a single color is used.
The purpose of this paper is to determine, in the colorful minor setting, exactly which patterns have the Erd\H{o}s-P{\'o}sa property.

\subsection{Colorful graphs}
A \emph{colorful graph} is a pair $(G,\chi)$ where $G$ is a finite graph and\footnote{For $n,m\in\mathbb{Z},$ we denote $\{ x \mid 1\leq x \leq n\text{, }x\in\mathbb{Z}\}$ by $[n]$ and $\{ x \mid n \leq x \leq m\text{, }x\in\mathbb{Z} \}$ by $[n,m].$ Also, we write $I\finsub \Nbbb_{\geq1}$ in order to denote that $I$ is a finite set of colors, that is, a finite subset of $\Nbbb_{\geq1}$.} $\chi$
is a function mapping each vertex $v$ of $G$ to a \textsl{finite} subset of $\Nbbb_{\geq1}$; that is, colors are positive integers, so that every finite set of colors is contained in $[q]$ for some $q\in\Nbbb.$  
We refer to the set $\chi(v)$ 
as the \emph{palette}  of $v$ in $(G,\chi)$.
Also we define the \emph{palette} of $(G,\chi)$ 
as the set $\chi(G)\coloneqq \bigcup_{v\in V(G)}\chi(v)$. Clearly, the palette of every colorful graph is always a finite set of colors.
Given a vertex set $X\subseteq V(G),$ 
we define $\chi(X)\coloneqq \bigcup_{v \in X} \chi(v)$.

For illustrating the colors, we use \textcolor{HotMagenta}{magenta} for $1,$ \textcolor{CornflowerBlue}{blue} for $2$, \textcolor{ChromeYellow}{yellow} for $3,$ \textcolor{AppleGreen}{green} for $4,$ and \textcolor{violet}{violet} for $5$.

Given some $I\finsub \Nbbb_{\geq1}$, a colorful graph $(G,\chi)$ is an \emph{$I$-colorful graph} if $\chi(G)= I$ and is called \emph{$I$-restricted} if either $I\neq\emptyset$ and  $\chi(G)\subsetneq I$ 
or $I=\emptyset$ and $V(G)=\emptyset$. 
A colorful graph $(G,\chi)$ is said to be \emph{$I$-rainbow}\footnote{Notice that -- in slight violation of the English language -- we use the word \say{rainbow} as an adjective here.} if $\chi(v)=I$ for all $v\in V(G).$ 
We denote by $\rho_{I}$ the coloring given by $\rho_{I}(v)\coloneqq I$ for every vertex $v,$ so that $(G,\rho_{I})$ is the $I$-rainbow colorful graph on $G.$
That way, every graph $G$ can be seen as the $\emptyset$-rainbow colorful graph $(G,\rho_{\emptyset}).$

A colorful graph $(H,\psi)$ is a \emph{(colorful) subgraph} of a colorful graph $(G,\chi),$ denoted by $(H,\psi)\subseteq (G,\chi),$ if $H$ is a subgraph of $G$ and $\psi(v)\subseteq \chi(v)$ for all $v\in V(H).$
If $(G,\chi)$ is a colorful graph and $H\subseteq G$ is a subgraph of $G$ we write, in slight abuse of notation, $(H,\chi)$ for the colorful subgraph of $(G,\chi)$ where all vertices of $H$ receive precisely the colors they received in $(G,\chi),$
i.e., we use $(H,\chi)$ for the more precise but   overloaded notation $(H,\chi|_{V(H)})$.
Given a graph class $\Gcal$ and some colorful graph 
$(G,\chi)$ we say that $(G,\chi)$ belongs to $\Gcal$ if $G$ belongs to $\Gcal$. For instance by saying that $(G,\chi)$ is planar/outerplanar/a clique  we mean that $G$ is planar/outerplanar/a clique.

\subsection{Colorful minors}
A colorful graph $(H,\psi)$ is a \emph{colorful minor} of a colorful graph $(G,\chi)$ if $(H,\psi)$ can be obtained from $(G,\chi)$ by means of the following operations:
\begin{itemize}
 \item deleting an edge $e\in E(G),$
 \item deleting a vertex $v\in V(G)$ and restricting the domain of $\chi$ to $V(G)\setminus\{ v\},$
 \item for some $v\in V(G)$ and $i\in \chi(v)$ overwrite $\chi(v)\coloneqq \chi(v)\setminus\{ i\}$ (we refer to this operation as \emph{removing a color} (\emph{from $v$})), and
 \item contracting an edge $uv\in E(G),$ that is introducing a new vertex $x_{uv}$ with neighborhood $N_G(u)\cup N_G(v),$ setting $\chi(x_{uv})\coloneqq \chi(u) \cup \chi(v),$ and then deleting both $u$ and $v$ from the resulting colorful graph.
\end{itemize}
See \zcref{fig_ColorfulMinorIntro} for an example.
Notice that colorful graphs whose palette is a subset of $\{1\}$ are exactly annotated graphs and the notions of colorful minors and rooted minors coincide on annotated graphs.
We use $(H,\psi)≤(G,\chi)$ in order to denote the 
fact that $(H,\psi)$ is a colorful minor of $(G,\chi)$.

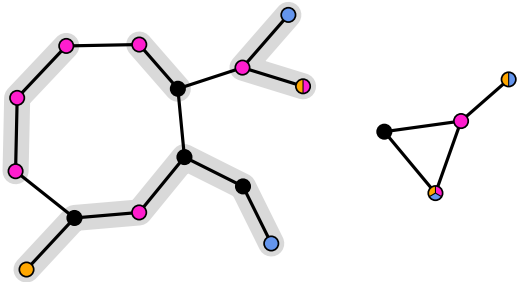
\begin{figure}[ht]
 \vspace{-8pt}
 \centering
\scalebox{.8}{%
\def\vr{6.75}                          % vertex radius
\def\vw{1.6}                           % vertex outline width

\newcommand{\halfvtx}[3]{%  #1 = coordinate, #2 = left colour, #3 = right colour
  \begin{scope}
    \clip (#1) circle (\vr);
    \fill[#2] ([shift={(-\vr,-\vr)}]#1) rectangle ([shift={(0,\vr)}]#1);
    \fill[#3] ([shift={(0,-\vr)}]#1) rectangle ([shift={(\vr,\vr)}]#1);
  \end{scope}
  \draw[line width=1pt] ([shift={(0,-\vr)}]#1) -- ([shift={(0,\vr)}]#1);
  \draw[line width=\vw pt] (#1) circle (\vr);}
\newcommand{\thirdvtx}[4]{% #1 = coordinate, #2 = upper right, #3 = upper left, #4 = bottom
  \fill[#2] (#1) -- ++(-30:\vr) arc[start angle=-30,end angle= 90,radius=\vr] -- cycle;
  \fill[#3] (#1) -- ++( 90:\vr) arc[start angle= 90,end angle=210,radius=\vr] -- cycle;
  \fill[#4] (#1) -- ++(210:\vr) arc[start angle=210,end angle=330,radius=\vr] -- cycle;
  \foreach \a in {90,210,330}{\draw[line width=1pt] (#1) -- ++(\a:\vr);}
  \draw[line width=\vw pt] (#1) circle (\vr);}

\scalebox{.5}{
\begin{tikzpicture}[x=1pt,y=1pt,
    branch/.style={cbranch, line width=24.5pt, line cap=round, line join=round},
    edg/.style   ={line width=3pt, line cap=round},
    vtx/.style   ={circle, draw, line width=\vw pt, inner sep=0pt, minimum size=13.5pt}]

  \coordinate (a1) at (105.6,227.7);  \coordinate (a2) at ( 59.4,178.8);
  \coordinate (a3) at ( 57.9,109.8);
  \coordinate (b1) at (174.3,229.0);  \coordinate (b2) at (210.6,187.4);
  \coordinate (c1) at (314.6,256.9);  \coordinate (c2) at (271.4,207.3);
  \coordinate (c3) at (328.4,189.7);
  \coordinate (d1) at ( 68.2, 17.3);  \coordinate (d2) at (113.3, 65.8);
  \coordinate (d3) at (174.2, 71.0);  \coordinate (d4) at (216.8,123.1);
  \coordinate (d5) at (271.9, 95.5);  \coordinate (d6) at (298.5, 41.8);
  \coordinate (h1) at (404.9,147.0);  \coordinate (h2) at (453.0, 89.2);
  \coordinate (h3) at (477.1,157.0);  \coordinate (h4) at (521.9,196.2);

  \draw[branch] (a1) -- (a2) -- (a3);
  \draw[branch] (b1) -- (b2);
  \draw[branch] (c1) -- (c2) -- (c3);
  \draw[branch] (d1) -- (d2) -- (d3) -- (d4) -- (d5) -- (d6);

  \draw[edg] (a1) -- (a2) -- (a3);
  \draw[edg] (b1) -- (b2);
  \draw[edg] (c1) -- (c2) -- (c3);
  \draw[edg] (d1) -- (d2) -- (d3) -- (d4) -- (d5) -- (d6);
  \draw[edg] (a1) -- (b1);  \draw[edg] (a3) -- (d2);
  \draw[edg] (b2) -- (d4);  \draw[edg] (b2) -- (c2);
  \draw[edg] (h1) -- (h2) -- (h3) -- (h1);  \draw[edg] (h3) -- (h4);

  \foreach \v in {a1,a2,a3,b1,c2,d3,h3}{\node[vtx,fill=cM] at (\v) {};}
  \foreach \v in {c1,d6}            {\node[vtx,fill=cC] at (\v) {};}
  \foreach \v in {d1}               {\node[vtx,fill=cO] at (\v) {};}
  \foreach \v in {b2,d2,d4,d5,h1}   {\node[vtx,fill=black] at (\v) {};}
  \halfvtx{c3}{cO}{cM}
  \halfvtx{h4}{cO}{cC}
  \thirdvtx{h2}{cM}{cO}{cC}

\end{tikzpicture}
}
\medskip
}
 \caption{Two colorful graphs $(G,\chi)$ and $(H,\psi)$ such that $(H,\psi)$ is a colorful minor of $(G,\chi).$
 The gray subgraphs of $(G,\chi)$ indicate the connected vertex sets that have to be contracted in order to form $(H,\psi).$ Notice that it is necessary to remove some colors from some of the vertices of $(G,\chi).$}
 \label{fig_ColorfulMinorIntro}
\end{figure}

\subsection{The Erd\H{o}s-P{\'o}sa property}\label{subsec_EP_def}
Let $(H,\psi)$ and $(G,\chi)$ be colorful graphs and let $k\in\Nbbb$.
\begin{itemize}
\item A \emph{packing of $(H,\psi)$ in $(G,\chi)$ of size $k$}
is a family $(G_{1},\ldots,G_{k})$ of pairwise vertex-disjoint subgraphs of $G$
such that $(H,\psi)$ is a colorful minor of $(G_{i},\chi)$ for every $i\in[k]$.
Note that a packing is an indexed family and not a set of subgraphs; the distinction only matters when $V(H)=\emptyset,$ in which case the subgraph of $G$ with no vertices is vertex-disjoint from itself and $(G,\chi)$ has packings of $(H,\psi)$ of every size.

\item A \emph{covering of $(H,\psi)$ in $(G,\chi)$ of size $k$}
is a set $S\subseteq V(G)$ such that $|S|\leq k$ and $(G-S,\chi)$ does not contain $(H,\psi)$ as a colorful minor.
\end{itemize}

We say that the colorful graph $(H,\psi)$ has the \emph{Erd\H{o}s-P{\'o}sa property} \cite{ErdosP1965independent}, in short the \emph{EP-property}, if 
there is a function $f\colon\Nbbb\to\Nbbb$ such that, for every $k\in\Nbbb$
and every colorful graph $(G,\chi)$, if $(G,\chi)$ has no
packing of $(H,\psi)$ of size $k$, then $(G,\chi)$ has a covering of $(H,\psi)$ of size at most $f(k)$.
We refer to the function $f$ as \emph{the gap} of the EP-property for $(H,\psi)$.
We also write $\pack_{H,\psi}(G,\chi)$ for the maximum size of a packing of $(H,\psi)$ in $(G,\chi)$ and $\cover_{H,\psi}(G,\chi)$ for the minimum size of a covering of $(H,\psi)$ in $(G,\chi)$.
In these terms, $(H,\psi)$ has the EP-property if and only if there is a function $f\colon\Nbbb\to\Nbbb$ such that $\cover_{H,\psi}(G,\chi)\leq f(\pack_{H,\psi}(G,\chi))$ for every colorful graph $(G,\chi)$.\medskip

In this paper we give a complete characterization of the colorful graphs that have the Erd\H{o}s-P{\'o}sa property.
For $\emptyset$-colorful graphs, i.e., for graphs without colors, such a characterization is provided by the result of Robertson and Seymour \cite{RobertsonS1986Grapha} asserting that an $\emptyset$-colorful graph has the EP-property if and only if it is planar.
Our characterization extends this result to all colorful graphs.
As we will see, in the presence of colors the boundary  
changes in subtle ways that are worth highlighting.

\section{Basic definitions and statement of the main result}

\paragraph{Domains and their obstructions.}

For every set $\Zcal$ of colorful graphs and every $I\finsub \Nbbb_{\geq1}$, we define $\Zcal|_I$ as 
the restriction of $\Zcal$  to its colorful graphs whose 
palette is a subset of $I$, i.e.,  $\Zcal|_I\coloneqq \{(G,\chi)\mid (G,\chi)\in\Zcal \text{~and~} \chi(G)\subseteq I\}$.

In this paper we study classes of colorful graphs that are closed under colorful minors, i.e., they contain all colorful minors of their members.
For simplicity, we refer to every such class as a \emph{domain}.
Notice that  if $\Zcal$ is a domain then so is $\Zcal|_I$, 
for every $I\finsub \Nbbb_{\geq1}$.

Given some domain $\Dcal$, we use $\obs(\Dcal)$ as the 
set of all colorful graphs that do not belong in $\Dcal$
and whose  proper colorful minors belong in $\Dcal$. 
We refer to the colorful graphs in $\obs(\Dcal)$ as \emph{obstructions} 
of $\Dcal$ and to the set $\obs(\Dcal)$  as the \emph{obstruction set} of $\Dcal$. 
Notice that $\obs(\Dcal)$ is not necessarily a finite set, however, given that we assumed that each colorful graph carries 
a palette of a finite number of colors, this set is always countable.
For an example of an infinite obstruction, if $\Dcal$ is the domain of all $\emptyset$-colorful graphs, i.e., the set of all graphs without colors, then 
$\obs(\Dcal)=\{(K_{1},\chi_{i})\mid i\in\Nbbb\}$ where 
$\chi_{i}$ assigns to the unique vertex in $K_{1}$ the color $i$.
On the other side, if $\Dcal$ is the domain of all colorful graphs, then $\obs(\Dcal)$ is finite as it is  the empty set.

The above imply that $\obs(\Dcal)$ characterizes $\Dcal$ completely, in the sense that $\Dcal$ consists precisely of the colorful graphs that contain no member of $\obs(\Dcal)$ as a colorful minor.

\subsection{Torso treewidth} For a vertex set $X\subseteq V(G)$ in a graph $G$ we define the \emph{torso} of $X$ in $G,$ denoted by $\torso(G,X),$ as the graph obtained from $G[X]$ by turning the neighborhood of $J$ in $X$ into a clique for every component $J$ of $G-X.$

A \emph{tree-decomposition} for a graph $G$ is a pair $\Tcal=(T,\beta)$ where $T$ is a tree and $\beta\colon V(T)\to2^{V(G)},$ called the \emph{bags} of $\Tcal,$ assigns a set of vertices of $G$ to every node of $T$ such that $\bigcup_{t\in V(T)}\beta(t)=V(G),$ for every $e\in E(G)$ there is $t\in V(T)$ with $e\subseteq \beta(t),$ and for every $v\in V(G)$ the set $\{ t\in V(T) \mid v\in\beta(t)\}$ is connected.
The \emph{adhesion} of $\Tcal$ is $\max_{dt\in E(T)}|\beta(d)\cap \beta(t)|$ and the width of $\Tcal$ is defined as $\max_{t\in V(T)}|\beta(t)|-1.$
The \emph{treewidth} of a graph $G,$ denoted by $\tw(G),$ is the smallest integer $k$ such that there is a tree-decomposition of width at most $k$ for $G.$

Given some $I\finsub \Nbbb_{\geq1}$,  the \emph{$I$-torso treewidth} of a colorful graph $(G,\chi)$ is the smallest integer $k$ such that there exists a set $X\subseteq V(G)$ such that the treewidth of the torso of $X$ in $G$ is at most $k$ and for every component $J$ of $G-X$ the colorful graph $(J,\chi)$ is $I$-restricted.

\subsection{Segregated grids}
Consider some $I\finsub \Nbbb_{\geq1}$ where $q\coloneqq |I|>0$.
Let also $k\in\Nbbb_{≥1}$.
An \emph{$(I,k)$-segregated grid} is an $I$-colorful graph $(G,\chi)$ where $G$ is the $(qk\times qk)$-grid where the vertices of the first column can be numbered as $v_1,\dots,v_{qk}$ in order of their appearance and we have $\chi(u)=\emptyset$ for all $u\in V(G)\setminus\{ v_1,\dots,v_{qk}\}$ and there exists a bijection $\iota\colon[q]\to I$ such that for every $i\in [q]$ and every $j\in[(i-1)k+1,ik],$ $\chi(v_j)=\{ \iota(i)\}.$
See \zcref{fig_SegregatedGrids} for an example.

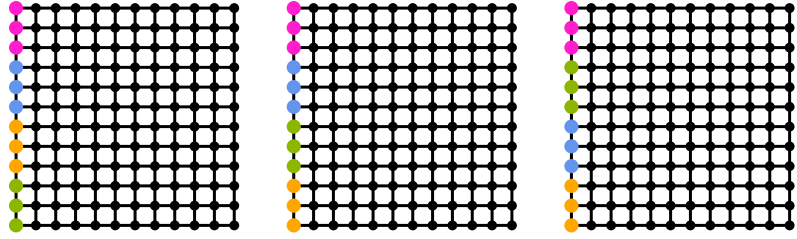
\begin{figure}[ht]
 \vspace{0pt}
 \centering
 \scalebox{1}{
 \begin{tikzpicture}

 \pgfdeclarelayer{background}
		\pgfdeclarelayer{foreground}
			
		\pgfsetlayers{background,main,foreground}
			
 \begin{pgfonlayer}{main}
 \node (M) [v:ghost] {};

 \end{pgfonlayer}

 \begin{pgfonlayer}{background}
 \pgftext{\scalebox{0.3763}{%
\begin{tikzpicture}[x=19.84pt,y=19.76pt,line width=3pt,
    dot/.style ={circle,fill,inner sep=0pt,minimum size=9.92pt},
    cdot/.style={circle,fill=#1,inner sep=0pt,minimum size=14.18pt}]
  \foreach \off/\seq in {0/{clM,clM,clM,clB,clB,clB,clY,clY,clY,clG,clG,clG},
                        14/{clM,clM,clM,clB,clB,clB,clG,clG,clG,clY,clY,clY},
                        28/{clM,clM,clM,clG,clG,clG,clB,clB,clB,clY,clY,clY}}{%
    \foreach \r in {0,...,11}{\draw (\off,\r) -- (\off+11,\r);}
    \foreach \c in {0,...,11}{\draw (\off+\c,0) -- (\off+\c,11);}
    \foreach \c in {1,...,11}{\foreach \r in {0,...,11}{\node[dot] at (\off+\c,\r) {};}}
    \foreach \cl [count=\i from 0] in \seq {\node[cdot=\cl] at (\off,11-\i) {};}}
\end{tikzpicture}
}} at (M.center);
 \end{pgfonlayer}
 
 \begin{pgfonlayer}{foreground}
 \end{pgfonlayer}

 \end{tikzpicture}}

 \vspace{-1pt}
 \caption{Three non-isomorphic $([4],3)$-segregated grids.}
 \vspace{-4pt}
 \label{fig_SegregatedGrids}
\end{figure}

The $(\emptyset,k)$-segregated grid is the $(k\times k)$-grid without colors. 
Notice that there are $\lceil |I|!/2\rceil$ different 
$(I,k)$-segregated grids.

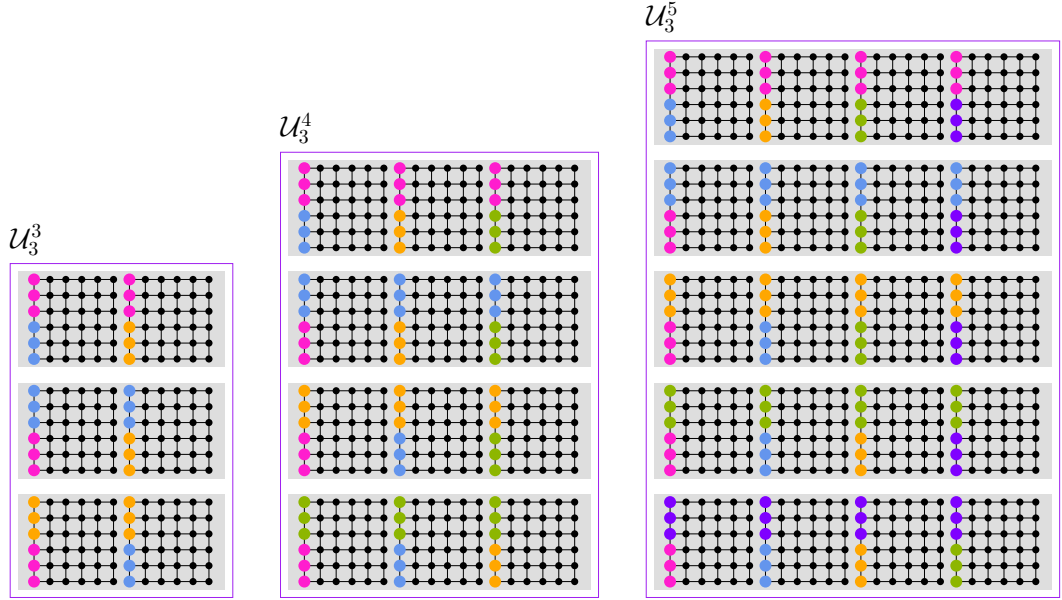
\begin{figure}[ht]
\begin{center}
\scalebox{.45}{%
\scalebox{1.87}{%
\newcommand{\ugrid}[4]{%
  \foreach \r in {0,...,5}{\draw (#1,#2+\r) -- (#1+5,#2+\r);}%
  \foreach \c in {0,...,5}{\draw (#1+\c,#2) -- (#1+\c,#2+5);}%
  \foreach \c in {1,...,5}{\foreach \r in {0,...,5}{\node[dot] at (#1+\c,#2+\r) {};}}%
  \foreach \r in {0,...,5}{\ifnum\r<3 \node[cdot=#4] at (#1,#2+\r) {};%
                           \else      \node[cdot=#3] at (#1,#2+\r) {};\fi}}
\newcommand{\ugridplain}[2]{%
  \foreach \r in {0,...,5}{\draw (#1,#2+\r) -- (#1+5,#2+\r);}%
  \foreach \c in {0,...,5}{\draw (#1+\c,#2) -- (#1+\c,#2+5);}%
  \foreach \c in {0,...,5}{\foreach \r in {0,...,5}{\node[dot] at (#1+\c,#2+\r) {};}}}
\newcommand{\uplate}[3]{\fill[plate] (#1-1,#2-.5) rectangle (#1+6*#3,#2+5.5);}
\newcommand{\uframe}[5]{%
  \draw[framepurple] (#1-1.5,#2-1) rectangle (#1+6*#3+.5,#2+7*#4-1);%
  \node[anchor=base west, inner sep=0pt] at (#1-1.5,#2+7*#4) {\Large$#5$};}
\resizebox{\textwidth}{!}{%
\begin{tikzpicture}[
    x=8pt, y=8pt, line width=.4pt,
    dot/.style  ={circle, fill, inner sep=0pt, minimum size=3.6pt},
    cdot/.style ={circle, fill=#1, inner sep=0pt, minimum size=6pt},
    plate/.style={fill=plategray, fill opacity=.75}]

  \def\xA{0}\def\xB{11}\def\xC{28}\def\xD{51}   % horizontal position of the four columns

  \foreach \py/\topc/\bots in {2/uA/{uB,uC},1/uB/{uA,uC},0/uC/{uA,uB}}{%
    \pgfmathtruncatemacro{\Y}{7*\py}\uplate{\xB}{\Y}{2}%
    \foreach \cb [count=\g from 0] in \bots {%
      \pgfmathtruncatemacro{\GX}{\xB+6*\g}\ugrid{\GX}{\Y}{\topc}{\cb}}}
  \uframe{\xB}{0}{2}{3}{\mathcal{U}^{3}_{3}}

  \foreach \py/\topc/\bots in {3/uA/{uB,uC,uD},2/uB/{uA,uC,uD},
                               1/uC/{uA,uB,uD},0/uD/{uA,uB,uC}}{%
    \pgfmathtruncatemacro{\Y}{7*\py}\uplate{\xC}{\Y}{3}%
    \foreach \cb [count=\g from 0] in \bots {%
      \pgfmathtruncatemacro{\GX}{\xC+6*\g}\ugrid{\GX}{\Y}{\topc}{\cb}}}
  \uframe{\xC}{0}{3}{4}{\mathcal{U}^{4}_{3}}

  \foreach \py/\topc/\bots in {4/uA/{uB,uC,uD,uE},3/uB/{uA,uC,uD,uE},
                               2/uC/{uA,uB,uD,uE},1/uD/{uA,uB,uC,uE},
                               0/uE/{uA,uB,uC,uD}}{%
    \pgfmathtruncatemacro{\Y}{7*\py}\uplate{\xD}{\Y}{4}%
    \foreach \cb [count=\g from 0] in \bots {%
      \pgfmathtruncatemacro{\GX}{\xD+6*\g}\ugrid{\GX}{\Y}{\topc}{\cb}}}
  \uframe{\xD}{0}{4}{5}{\mathcal{U}^{5}_{3}}
\end{tikzpicture}}}
}
\end{center}
\caption{The sets   $\Ucal_{3}^{3},$ $\Ucal_{3}^{4},$
and $\Ucal_{3}^{5}.$}
\label{fig_UsGrids}
\end{figure}

\subsection{The domain \texorpdfstring{$\Ucal$}{U}}
Consider a $q\in\Nbbb_{≥3}$  and a $k\in\Nbbb_{≥1}$.
For every $i\in [q]$, we define the colorful graph $U^{q}_{i,k}$ as the disjoint union
of the $(I,k)$-segregated grids for all $I\in \{\{i,h\}\mid h\in [q]\setminus\{i\}\}$.
Moreover, for every $Z\subseteq[q]$ with $|Z|=3$, we define the colorful graph $T^{q}_{Z,k}$ as the disjoint union
of the $(I,k)$-segregated grids for all $I\in \{P\subseteq Z\mid |P|=2\}\cup\{\{d\}\mid d\in [q]\}$.

Given some $q\in\Nbbb_{≥3}$  and a $k\in\Nbbb_{≥1}$ we also define  
$\Ucal_{k}^{q}=\{U^{q}_{i,k}\mid i\in[q]\}\cup\{T^{q}_{Z,k}\mid Z\subseteq[q],\ |Z|=3\}$.
See \zcref{fig_UsGrids} for some examples of the colorful graphs $U^{q}_{i,k}$ and \zcref{fig_Tgrids} for examples of the colorful graphs $T^{q}_{Z,k}$,   for various values of $q$ and $k$.
Notice that $|\Ucal_{k}^{q}|=q+\binom{q}{3}$.
For every  $q\in \Nbbb_{≥3}$, we define the domain $\Ucal_{q}$ as  follows 
\begin{align}
\Ucal_{q}  =  \{(G,\chi)\mid \text{$(G,\chi)$ is a colorful minor of some  colorful graph in $\bigcup_{k\in\Nbbb_{≥1}}\Ucal_{k}^{q}$}\}
\label{eq_U_d}
\end{align}
and we set 
\begin{align}
\Ucal =   \bigcup_{q\in\Nbbb_{≥3}}\Ucal_{q}
\end{align}
As we prove in \zcref{subsect_U_crucial}, the domain $\Ucal$ is precisely the domain of all crucial colorful graphs. We postpone the technical definition of this particular domain to \zcref{subsect_obs_cru}. We need this notion for stating the main result of this paper. Apart from this, 
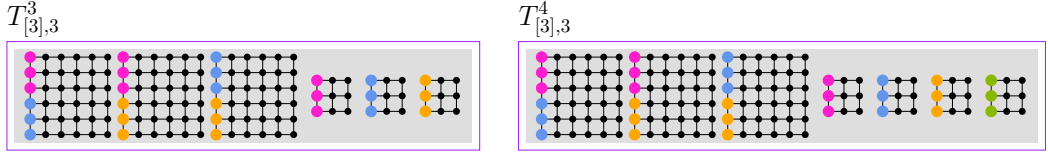
\begin{figure}[ht]
\centering
\scalebox{.39}{%
\scalebox{1.87}{%
\newcommand{\tugrid}[4]{%
  \foreach \r in {0,...,5}{\draw (#1,#2+\r) -- (#1+5,#2+\r);}%
  \foreach \c in {0,...,5}{\draw (#1+\c,#2) -- (#1+\c,#2+5);}%
  \foreach \c in {1,...,5}{\foreach \r in {0,...,5}{\node[dot] at (#1+\c,#2+\r) {};}}%
  \foreach \r in {0,...,5}{\ifnum\r<3 \node[cdot=#4] at (#1,#2+\r) {};%
                           \else      \node[cdot=#3] at (#1,#2+\r) {};\fi}}
\newcommand{\tsgrid}[3]{%
  \foreach \r in {0,...,2}{\draw (#1,#2+\r) -- (#1+2,#2+\r);}%
  \foreach \c in {0,...,2}{\draw (#1+\c,#2) -- (#1+\c,#2+2);}%
  \foreach \c in {1,2}{\foreach \r in {0,...,2}{\node[dot] at (#1+\c,#2+\r) {};}}%
  \foreach \r in {0,...,2}{\node[cdot=#3] at (#1,#2+\r) {};}}
\newcommand{\tplate}[3]{\fill[plate] (#1-1,#2-.5) rectangle (#1+#3,#2+5.5);}
\newcommand{\tframe}[4]{%
  \draw[framepurple] (#1-1.5,#2-1) rectangle (#1+#3+.5,#2+6);%
  \node[anchor=base west, inner sep=0pt] at (#1-1.5,#2+7.1) {\Large$#4$};}
\begin{tikzpicture}[
    x=8pt, y=8pt, line width=.4pt,
    dot/.style  ={circle, fill, inner sep=0pt, minimum size=3.6pt},
    cdot/.style ={circle, fill=#1, inner sep=0pt, minimum size=6pt},
    plate/.style={fill=plategray, fill opacity=.75}]
  \tplate{0}{0}{28.5}
  \tugrid{0}{0}{uA}{uB}
  \tugrid{6}{0}{uA}{uC}
  \tugrid{12}{0}{uB}{uC}
  \tsgrid{18.5}{1.5}{uA}
  \tsgrid{22}{1.5}{uB}
  \tsgrid{25.5}{1.5}{uC}
  \tframe{0}{0}{28.5}{T^{3}_{[3],3}}
  \tplate{33}{0}{32}
  \tugrid{33}{0}{uA}{uB}
  \tugrid{39}{0}{uA}{uC}
  \tugrid{45}{0}{uB}{uC}
  \tsgrid{51.5}{1.5}{uA}
  \tsgrid{55}{1.5}{uB}
  \tsgrid{58.5}{1.5}{uC}
  \tsgrid{62}{1.5}{uD}
  \tframe{33}{0}{32}{T^{4}_{[3],3}}
\end{tikzpicture}}
}
\caption{The colorful graphs $T^{3}_{[3],3}$ and $T^{4}_{[3],3}$.}
\label{fig_Tgrids}
\end{figure}
we need two more ingredients that we define below. We consider a set $\Ocal$ of colorful graphs and, for every $I\finsub \Nbbb_{\geq1}$, we set $\Ocal_{I}\coloneqq \Ocal|_{I}$, that is, $\Ocal_{I}$ consists of the colorful graphs of $\Ocal$ whose palette is a subset of $I$.
Each $\Ocal_{I}$ is non-empty and finite.

This set is  schematically described 
in \zcref{fig_EPIntro} and is formally defined in \zcref{subsect_obs_cru}.

We  stress that the size of $\Ocal_{I}$ grows polynomially in the number of colors $q\coloneqq |I|$. 
In fact, 
$$|\Ocal_{[q]}|=  \frac{1}{24} \left( 3q^4 - 14q^3 + 129q^2 - 22q + 48 \right)\in\bigO{q^4}.$$

\begin{figure}[ht]
 \vspace{-8pt}
 \centering
 \scalebox{.94}{\begin{tikzpicture}

 \pgfdeclarelayer{background}
		\pgfdeclarelayer{foreground}
			
		\pgfsetlayers{background,main,foreground}
			
 \begin{pgfonlayer}{main}
 \node (M) [v:ghost] {};
 
 \node (MBottom) [v:ghost,position=270:30mm from M] {};

 \node (R) [v:ghost,position=0:71mm from MBottom] {};

 \node (q4) [v:ghost,position=90:4mm from R] {$p=4$};
 \node (q3) [v:ghost,position=90:11mm from R] {$p=3$};
 \node (q2) [v:ghost,position=90:22mm from R] {$p=2$};
 \node (q1) [v:ghost,position=90:36mm from R] {$p=1$};
 \node (q0) [v:ghost,position=90:51mm from R] {$p=0$};

 \end{pgfonlayer}

 \begin{pgfonlayer}{background}
 \pgftext{\scalebox{0.4210}{%
\begin{tikzpicture}[x=1pt,y=1pt,line width=2pt,
    v/.style ={circle,draw,line width=1.5pt,inner sep=0pt,minimum size=8.4pt,fill=#1},
    b/.style ={circle,draw,line width=1pt,inner sep=0pt,minimum size=7.5pt,fill=black}]
  \fill[black!10] (-0.2,2.3) rectangle (897.8,290.8);
  \fill[black!16] (-0.2,2.3) rectangle (889.7,191.2);
  \fill[black!22] (-0.2,2.3) rectangle (885,93.2);
  \fill[black!28] (-0.2,2.3) rectangle (882,46);
  \draw (11.3,347) -- (51,381.4);
  \draw (51,381.4) -- (89,348.3);
  \draw (89,348.3) -- (75.9,310.5);
  \draw (75.9,310.5) -- (28.5,309.4);
  \draw (28.5,309.4) -- (11.3,347);
  \draw (28.2,309.2) -- (50.7,381.5);
  \draw (50.7,381.5) -- (75.9,310.6);
  \draw (75.9,310.6) -- (11.7,347.1);
  \draw (11.7,347.1) -- (89.2,348.3);
  \draw (89.2,348.3) -- (28.2,309.2);
  \draw (150.7,381.3) -- (114.6,359.5);
  \draw (114.6,359.5) -- (114.8,322.5);
  \draw (114.8,322.5) -- (150.8,302.1);
  \draw (150.8,302.1) -- (184.7,322.8);
  \draw (184.7,322.8) -- (185,359.7);
  \draw (185,359.7) -- (150.7,381.3);
  \draw (150.6,381.2) -- (150.5,302);
  \draw (114.6,322.2) -- (184.8,359.7);
  \draw (114.7,360.1) -- (184.7,323.6);
  \draw (15.6,203.7) -- (50.1,277.9);
  \draw (50.1,277.9) -- (84.2,203.7);
  \draw (84.2,203.7) -- (15.6,203.7);
  \draw (15.9,203.8) -- (50.1,237.6);
  \draw (50.1,237.6) -- (50.2,277.4);
  \draw (50.1,237.4) -- (83.1,203.2);
  \draw (111.5,201) -- (151.5,279.9);
  \draw (151.5,279.9) -- (188.5,200.6);
  \draw (188.5,200.6) -- (111.5,201);
  \draw (111.5,201) -- (151.9,226.8);
  \draw (151.9,226.8) -- (188.6,200.4);
  \draw (151.4,226.7) -- (151.5,280);
  \draw (111.5,200.9) -- (151.5,249.5);
  \draw (151.5,249.5) -- (188.4,200.8);
  \draw (210.8,200.9) -- (287.9,201.2);
  \draw (287.9,201.2) -- (251.2,281);
  \draw (251.2,281) -- (210.8,200.9);
  \draw (251.1,281.4) -- (250.5,226.2);
  \draw (210.8,200.8) -- (250.7,226.1);
  \draw (250.7,226.1) -- (287.6,201);
  \draw (348.5,281.4) -- (348.2,201);
  \draw (348.4,281.4) -- (313.9,240);
  \draw (313.9,240) -- (348.4,200.9);
  \draw (348.4,200.9) -- (382.2,240.1);
  \draw (382.2,240.1) -- (348.4,281.4);
  \draw (49.8,178.7) -- (15.6,104.6);
  \draw (15.6,104.6) -- (83.9,104.4);
  \draw (83.9,104.4) -- (49.8,178.7);
  \draw (49.7,178.5) -- (49.8,137.2);
  \draw (49.8,137.2) -- (16.9,104.2);
  \draw (49.8,136.8) -- (83.3,103.7);
  \draw (210.8,103.5) -- (250.9,180.6);
  \draw (250.9,180.6) -- (287.6,103.5);
  \draw (287.6,103.5) -- (210.8,103.5);
  \draw (250.8,180.3) -- (250.5,127);
  \draw (210.5,103.1) -- (250.5,128.6);
  \draw (250.5,128.6) -- (287.7,103.2);
  \draw (210.8,102.8) -- (250.5,151.8);
  \draw (250.5,151.8) -- (287.6,102.4);
  \draw (350,183.4) -- (349.6,128);
  \draw (310.1,102.8) -- (350.2,128.2);
  \draw (350.2,128.2) -- (387,102.8);
  \draw (310.1,102.9) -- (386.8,103);
  \draw (309.7,103.1) -- (350.1,184.1);
  \draw (350.1,184.1) -- (386.8,103.1);
  \draw (547,182.4) -- (546.5,102.5);
  \draw (447.7,182.4) -- (447.3,102.5);
  \draw (546.9,182.7) -- (512.4,140.3);
  \draw (512.4,140.3) -- (546.7,102);
  \draw (447.6,182.7) -- (413.2,140.3);
  \draw (413.2,140.3) -- (447.5,102);
  \draw (546.9,101.8) -- (580.7,140.7);
  \draw (580.7,140.7) -- (547.1,182.6);
  \draw (447.7,101.8) -- (481.4,140.7);
  \draw (481.4,140.7) -- (447.9,182.6);
  \draw (646.1,174.7) -- (612,140.8);
  \draw (612,140.8) -- (646.2,106.1);
  \draw (646.2,106.1) -- (680.4,140.5);
  \draw (680.4,140.5) -- (646.1,174.7);
  \draw (744.5,177.2) -- (710.7,106.9);
  \draw (710.7,106.9) -- (778.5,106.9);
  \draw (778.5,106.9) -- (744.5,177.2);
  \draw (836.4,176.9) -- (836.3,140.5);
  \draw (836.6,140.1) -- (805.7,107.4);
  \draw (836.6,140.3) -- (867.9,107.3);
  \draw (148.9,178.7) -- (114.7,104.6);
  \draw (114.7,104.6) -- (183,104.4);
  \draw (183,104.4) -- (148.9,178.7);
  \draw (148.8,178.5) -- (148.9,137.2);
  \draw (148.9,137.2) -- (115.9,104.2);
  \draw (148.9,136.8) -- (182.4,103.7);
  \node[b] at (50.9,380.8) {};
  \node[b] at (11.8,346.9) {};
  \node[b] at (28.7,309.1) {};
  \node[b] at (75.8,310.1) {};
  \node[b] at (150.7,302.1) {};
  \node[b] at (151.7,249.1) {};
  \node[b] at (250.9,151.6) {};
  \node[b] at (251,226.2) {};
  \node[b] at (350.1,128.1) {};
  \node[b] at (251,281) {};
  \node[b] at (350.1,182.8) {};
  \node[b] at (111.8,200.8) {};
  \node[b] at (210.9,103.4) {};
  \node[b] at (211,200.8) {};
  \node[b] at (310.1,102.7) {};
  \node[b] at (188.3,200.8) {};
  \node[b] at (287.4,103.4) {};
  \node[b] at (287.5,200.8) {};
  \node[b] at (386.6,102.7) {};
  \node[b] at (348.5,200.8) {};
  \node[b] at (546.9,102) {};
  \node[b] at (447.7,102) {};
  \node[b] at (836.6,140.1) {};
  \node[b] at (348.5,281.3) {};
  \node[b] at (546.9,182.4) {};
  \node[b] at (447.7,182.4) {};
  \node[b] at (150.7,380.8) {};
  \node[b] at (89.3,347.9) {};
  \node[b] at (114.9,322.5) {};
  \node[b] at (184.5,322.5) {};
  \node[b] at (114.9,359.5) {};
  \node[b] at (184.5,359.5) {};
  \node[v=clM] at (50,277.6) {};
  \node[v=clB] at (50,178.4) {};
  \node[v=clM] at (50,237.4) {};
  \node[v=clM] at (50,138.2) {};
  \node[v=clM] at (16,203.9) {};
  \node[v=clM] at (16,104.7) {};
  \thirdv{50.6}{69.6}{7.75}{clM}{clB}{clY}
  \node[v=clM] at (84,203.9) {};
  \node[v=clM] at (84,104.7) {};
  \node[v=clM] at (151.8,226.6) {};
  \node[v=clB] at (251,127.2) {};
  \node[v=clM] at (250.4,249.8) {};
  \node[v=clM] at (349.5,151.3) {};
  \node[v=clM] at (348.2,239.9) {};
  \node[v=clM] at (546.6,140.5) {};
  \node[v=clB] at (447.4,140.5) {};
  \node[v=clM] at (382.1,239.9) {};
  \node[v=clB] at (580.5,140.5) {};
  \node[v=clM] at (481.3,140.5) {};
  \node[v=clM] at (680.3,140.5) {};
  \halfv{744.3}{177.2}{4.21}{clM}{clB}
  \halfv{836.4}{177.2}{4.21}{clM}{clB}
  \halfv{778.4}{107.1}{4.21}{clM}{clB}
  \halfv{867.7}{107.1}{4.21}{clM}{clB}
  \halfv{710.3}{107.1}{4.21}{clM}{clB}
  \halfv{31.4}{23.4}{7.69}{clM}{clB}
  \halfv{67.3}{23.4}{7.69}{clY}{clG}
  \halfv{805.3}{107.1}{4.21}{clM}{clB}
  \node[v=clB] at (646.3,174.5) {};
  \node[v=clM] at (314.1,239.9) {};
  \node[v=clB] at (512.5,140.5) {};
  \node[v=clM] at (413.3,140.5) {};
  \node[v=clM] at (612.3,140.5) {};
  \node[v=clB] at (646.3,106.5) {};
  \node[v=clM] at (250.4,201.4) {};
  \node[v=clB] at (349.5,102.9) {};
  \node[v=clM] at (151.8,279.7) {};
  \node[v=clM] at (251,180.3) {};
  \node[v=clM] at (149,178.4) {};
  \node[v=clB] at (149,138.2) {};
  \node[v=clB] at (115.1,104.7) {};
  \node[v=clB] at (183.1,104.7) {};
\end{tikzpicture}
}} at (M.center);
 \end{pgfonlayer}
 
 \begin{pgfonlayer}{foreground}
 \end{pgfonlayer}

 \end{tikzpicture}
 }

 \caption{The obstructions to the Erd\H{o}s-P{\'o}sa property:
 For $q\in\mathbb{N},$ the family $\Ocal_{[q]}$ consists, for $p\in[0,4],$ of all the colorful graphs displayed in the corresponding level above, where the levels for $p\in[4]$ are copied for each of the $\binom{q}{p}$ possible choices of $p$ colors out of $[q]$.}
 \label{fig_EPIntro}
\end{figure}
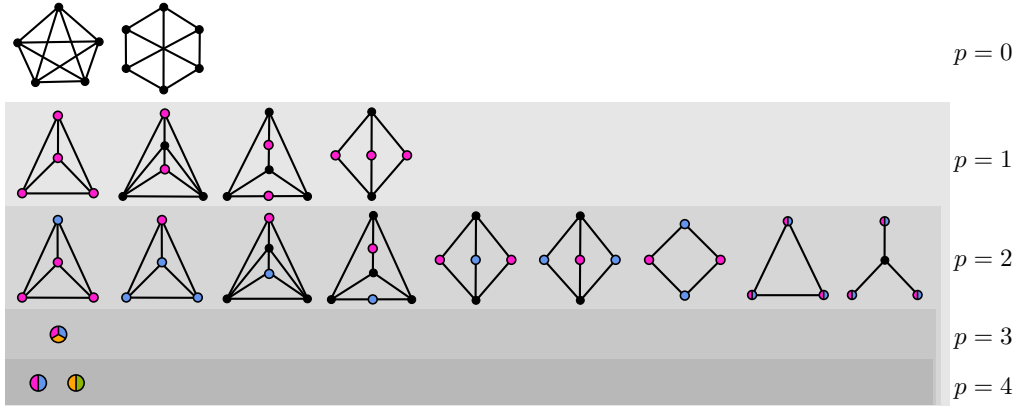

\subsection{Delineating the Erd\H{o}s-P{\'o}sa property}\label{subsec_Results}

Our main result is a complete characterization of the colorful graphs that have the Erd\H{o}s-P{\'o}sa property.

\begin{theorem}\label{th_EP_single}
For every colorful graph $(H,\psi)$ with $V(H)\neq\emptyset,$ the following are equivalent.
\begin{enumerate}
\item\label{it_s_ep} $(H,\psi)$ has the Erd\H{o}s-P{\'o}sa property,
\item\label{it_s_cru} $(H,\psi)$ is crucial,
\item\label{it_s_U} $(H,\psi)$ belongs to $\Ucal$, and
\item\label{it_s_O} $(H,\psi)$ does not contain any member of $\Ocal$ as a colorful minor.
\end{enumerate}
\end{theorem}

The equivalence of \ref{it_s_cru}, \ref{it_s_U}, and \ref{it_s_O} is established in \zcref{sec_obstructions} (see \zcref{obs_CrucialMinimal,lemma_U_equals_Q}) and the implication from \ref{it_s_cru} to \ref{it_s_ep} is proved in \zcref{sec_EP}.
The remaining implication, from \ref{it_s_ep} to \ref{it_s_cru}, is proved in \zcref{sec_EP_to_crucial}.

\begin{corollary}\label{cor_EPminorClosed}
The class of colorful graphs with at least one vertex that have the
Erd\H{o}s-P{\'o}sa property is closed under colorful minors.
\end{corollary}

\begin{proof}
Immediate from the equivalence of \ref{it_s_ep} and \ref{it_s_cru} in
\zcref{th_EP_single} together with \zcref{obs_CrucialClosed}.
\end{proof}

We know of no direct proof of \zcref{cor_EPminorClosed} that avoids \zcref{th_EP_single}; see the discussion at the beginning of \zcref{sec_EP_to_crucial}.

As a special case, \zcref{th_EP_single} has an interesting specialization for the $[q]$-rainbow colorful graphs, $q\in \Nbbb$.
We call a graph \emph{a linear forest} if it is the disjoint union of paths.

\begin{corollary}\label{cor_final_EP}
Let  $q\in\Nbbb$ and let $(H,\psi)$ be a  $[q]$-rainbow colorful graph where $|V(H)|>0$.
The following are equivalent.
\begin{itemize}
\item $(H,\psi)$ has the Erd\H{o}s-P{\'o}sa property,
\item $q≤2$ and, 
\begin{itemize}
\item if $q=0$ then $H$ is a  planar graph,
\item if $q=1$ then $H$ is an outerplanar graph,
\item if $q=2$ then $H$ is a linear forest, 
\end{itemize} 
\item $q≤2$ and $H$ excludes $K_{5-q}$ and $K_{3-q,3}$ as minors.
\end{itemize}
\end{corollary}

The proof of \zcref{cor_final_EP} has been deferred in \pageref{page_cor_final_EP}.
\subsection{The structural input: torso treewidth}\label{subsec_structuralInput}

It is easy to verify that, for every $I\finsub \Nbbb_{\geq1}$, the $I$-torso treewidth 
of $(I,k)$-segregated grids is lower bounded by an increasing  function of $k$ (see \cite{ProtopapasTW2025Colorful}). 
It was proved in \cite{ProtopapasTW2025Colorful} that the converse 
also holds, in the sense that if all $(I,k)$-segregated grids
are excluded as colorful minors then the $I$-torso treewidth is upper bounded. 
The structural ingredient that we import from \cite{ProtopapasTW2025Colorful} is the following.

\begin{proposition}[\cite{ProtopapasTW2025Colorful}]\label{prop_torsoGrid}
There exists a function $\sg\colon\mathbb{N}^2\to\mathbb{N}$ such that, for every $I\finsub \Nbbb_{\geq1}$ and $k\in\Nbbb_{\geq 1},$ if an $I$-colorful graph $(G,\chi)$
has $I$-torso treewidth $≥\sg(|I|,k)$, then $(G,\chi)$ contains as a colorful minor some $(I,k)$-segregated grid.
\end{proposition}

Given that the $I$-torso treewidth of $(I,k)$-segregated grids is unbounded as $k$ grows,
\zcref{prop_torsoGrid} can be seen as a min-max duality 
between $I$-torso treewidth and segregated grids.
Clearly, for $I=\emptyset$ this is the celebrated grid theorem \cite{RobertsonS1986Grapha}.
For the case $|I|=1,$ the $I$-torso treewidth of $(G,\chi)$ has been independently defined by Jansen and Swennenhuis \cite{JansenS2024SteinerTree} and by Hodor, La, Micek, and Rambaud \cite{HodorLMR24quick} under different names.
When $|I|=1$, \zcref{prop_torsoGrid} also follows from the results in \cite{MarxSW2017Rooted} and \cite{HodorLMR24quick}.

\section{Crucial colorful graphs and their obstructions}

\label{sec_obstructions}

We first define, in \zcref{subsec_Crucial}, the sets
$\Ocal^{0},$ $\tilde{\Ocal}^{1},$ $\tilde{\Ocal}^{2},$ $\Ocal^{3},$ and $\Ocal^{4}$; we then introduce in \zcref{subsec_crucial_define} the class of \textsl{crucial} colorful graphs.
Finally, in \zcref{subsect_obs_cru}, we define $$\Ocal\coloneqq \Ocal^{0} \cup \tilde{\Ocal}^{1} \cup \tilde{\Ocal}^{2} \cup \Ocal^{3} \cup \Ocal^{4}$$ 
and we prove that if a colorful graph excludes all the graphs in $\Ocal$ as colorful minors, then it is crucial.
In fact we prove that $\Ocal$ is the obstruction set of 
the domain of the crucial colorful graphs.
This is the structural input of \zcref{sec_EP}, together with the results of \zcref{subsect_U_crucial} on how the members of $\Ucal$ sit inside segregated grids.
In \zcref{subsect_U_crucial} we prove that the domain $\Ucal$ of \zcref{subsec_Results} is precisely the domain of the crucial colorful graphs and that every $I$-colorful member of $\Ucal$ is a colorful minor of every large enough $(I,k)$-segregated grid.

\subsection{The sets \texorpdfstring{$\Ocal^{0},$ $\tilde{\Ocal}^{1},$ $\tilde{\Ocal}^{2},$ $\Ocal^{3},$ and $\Ocal^{4}$}{O0, O1, O2, O3, and O4}}\label{subsec_Crucial}

 \paragraph{Planarity.}
We start by defining the set $\Ocal^{0}\coloneqq \{(K_{5},\rho_{\emptyset}),(K_{3,3},\rho_{\emptyset})\}.$
The following is the Kuratowski-Pontryagin-Wagner characterization of planar graphs \cite{Wagner1937Komplexe}.

\begin{observation}\label{obs_Wagner}%
A colorful graph $(G,\chi)$ is planar if and only if  $(G,\chi)$ excludes  every member of $\Ocal^{0}$ as a colorful minor.
\end{observation}

 \paragraph{Color facial graphs.}
By excluding the members of $\Ocal^{0}$ we first control the global underlying structure.
Next, we begin to further restrict the structure with respect to the colors.
We say that a colorful graph $(H,\psi)$ is \emph{color-facial} if it has a plane embedding where all colored vertices are on the boundary of a \textsl{single} face.

\begin{observation}\label{obs_colorFacialClosed}%
The class of color-facial colorful graphs is closed under colorful minors.
\end{observation}

\begin{figure}[ht]
 \centering
 \scalebox{1}{
 \begin{tikzpicture}

 \pgfdeclarelayer{background}
		\pgfdeclarelayer{foreground}
			
		\pgfsetlayers{background,main,foreground}
			
 \begin{pgfonlayer}{main}
 \node (M) [v:ghost] {};

 \end{pgfonlayer}

 \begin{pgfonlayer}{background}
 \pgftext{\scalebox{0.4780}{%
\begin{tikzpicture}[x=1pt,y=1pt,line width=2pt,
    v/.style ={circle,draw,line width=1.5pt,inner sep=0pt,minimum size=8.4pt,fill=#1},
    b/.style ={circle,draw,line width=1pt,inner sep=0pt,minimum size=7.5pt,fill=black}]
  \draw[clYpale,fill=clYbg,line width=2pt,rounded corners=7pt,
        dash pattern=on 6.5pt off 4.5pt] (1.6,5) rectangle (99.9,105.1);
  \draw (15.5,332.6) -- (50,406.8);
  \draw (50,406.8) -- (84.1,332.7);
  \draw (84.1,332.7) -- (15.5,332.6);
  \draw (15.5,227.7) -- (50,302);
  \draw (50,302) -- (84.1,227.8);
  \draw (84.1,227.8) -- (15.5,227.7);
  \draw (15.8,332.8) -- (50,366.5);
  \draw (50,366.5) -- (50,406.4);
  \draw (15.8,227.9) -- (50,261.6);
  \draw (50,261.6) -- (50,301.5);
  \draw (50,366.4) -- (82.9,332.1);
  \draw (50,261.5) -- (82.9,227.2);
  \draw (111.4,329.9) -- (151.4,408.8);
  \draw (151.4,408.8) -- (188.4,329.6);
  \draw (188.4,329.6) -- (111.4,329.9);
  \draw (111.4,225) -- (151.4,304);
  \draw (151.4,304) -- (188.4,224.7);
  \draw (188.4,224.7) -- (111.4,225);
  \draw (111.3,329.9) -- (151.7,355.7);
  \draw (151.7,355.7) -- (188.4,329.3);
  \draw (111.3,225) -- (151.7,250.8);
  \draw (151.7,250.8) -- (188.4,224.5);
  \draw (151.3,355.7) -- (151.3,408.9);
  \draw (151.3,250.8) -- (151.3,304);
  \draw (111.4,329.8) -- (151.3,378.4);
  \draw (151.3,378.4) -- (188.3,329.8);
  \draw (111.4,225) -- (151.3,273.6);
  \draw (151.3,273.6) -- (188.3,224.9);
  \draw (210.7,329.8) -- (287.7,330.1);
  \draw (287.7,330.1) -- (251.1,410);
  \draw (251.1,410) -- (210.7,329.8);
  \draw (210.7,225) -- (287.7,225.2);
  \draw (287.7,225.2) -- (251.1,305.1);
  \draw (251.1,305.1) -- (210.7,225);
  \draw (251,410.3) -- (250.4,355.1);
  \draw (251,305.4) -- (250.4,250.2);
  \draw (210.7,329.7) -- (250.6,355.1);
  \draw (250.6,355.1) -- (287.5,330);
  \draw (210.7,224.8) -- (250.6,250.2);
  \draw (250.6,250.2) -- (287.5,225.1);
  \draw (348.4,410.4) -- (348.1,330);
  \draw (348.4,305.5) -- (348.1,225.1);
  \draw (348.3,410.4) -- (313.8,368.9);
  \draw (313.8,368.9) -- (348.3,329.9);
  \draw (348.3,329.9) -- (382.1,369.1);
  \draw (382.1,369.1) -- (348.3,410.4);
  \draw (348.3,305.5) -- (313.8,264);
  \draw (313.8,264) -- (348.3,225);
  \draw (348.3,225) -- (382.1,264.2);
  \draw (382.1,264.2) -- (348.3,305.5);
  \draw (49.7,194.9) -- (15.5,120.9);
  \draw (15.5,120.9) -- (83.8,120.7);
  \draw (83.8,120.7) -- (49.7,194.9);
  \draw (49.6,194.7) -- (49.7,153.4);
  \draw (49.7,153.4) -- (16.7,120.5);
  \draw (49.7,153.1) -- (83.2,120);
  \draw (210.6,119.8) -- (250.8,196.9);
  \draw (250.8,196.9) -- (287.5,119.7);
  \draw (287.5,119.7) -- (210.6,119.8);
  \draw (250.6,196.5) -- (250.4,143.3);
  \draw (210.4,119.3) -- (250.4,144.8);
  \draw (250.4,144.8) -- (287.6,119.5);
  \draw (210.7,119.1) -- (250.3,168.1);
  \draw (250.3,168.1) -- (287.5,118.7);
  \draw (349.9,199.7) -- (349.5,144.3);
  \draw (310,119.1) -- (350.1,144.5);
  \draw (350.1,144.5) -- (386.9,119);
  \draw (310,119.2) -- (386.7,119.3);
  \draw (309.5,119.4) -- (350,200.3);
  \draw (350,200.3) -- (386.7,119.4);
  \draw (546.8,198.7) -- (546.4,118.7);
  \draw (447.6,198.7) -- (447.2,118.7);
  \draw (546.8,198.9) -- (512.3,156.6);
  \draw (512.3,156.6) -- (546.6,118.3);
  \draw (447.5,198.9) -- (413,156.6);
  \draw (413,156.6) -- (447.4,118.3);
  \draw (546.8,118) -- (580.5,156.9);
  \draw (580.5,156.9) -- (547,198.9);
  \draw (447.6,118) -- (481.3,156.9);
  \draw (481.3,156.9) -- (447.7,198.9);
  \draw (148.8,194.3) -- (114.6,120.3);
  \draw (114.6,120.3) -- (182.9,120);
  \draw (182.9,120) -- (148.8,194.3);
  \draw (50.1,91.2) -- (16,17.2);
  \draw (16,17.2) -- (84.2,16.9);
  \draw (84.2,16.9) -- (50.1,91.2);
  \draw (148.7,194.1) -- (148.8,152.8);
  \draw (148.8,152.8) -- (115.9,119.8);
  \draw (50,91) -- (50.1,49.7);
  \draw (50.1,49.7) -- (17.2,16.7);
  \draw (148.8,152.5) -- (182.3,119.4);
  \draw (50.1,49.4) -- (83.6,16.3);
  \node[b] at (151.6,377.4) {};
  \node[b] at (151.6,272.5) {};
  \node[b] at (250.8,167.3) {};
  \node[b] at (250.9,354.6) {};
  \node[b] at (250.9,249.7) {};
  \node[b] at (350,143.7) {};
  \node[b] at (250.9,409.3) {};
  \node[b] at (250.9,304.4) {};
  \node[b] at (350,198.5) {};
  \node[b] at (111.7,329.2) {};
  \node[b] at (111.7,224.3) {};
  \node[b] at (210.8,119) {};
  \node[b] at (210.9,329.2) {};
  \node[b] at (210.9,224.3) {};
  \node[b] at (310,118.3) {};
  \node[b] at (188.2,329.2) {};
  \node[b] at (188.2,224.3) {};
  \node[b] at (287.3,119) {};
  \node[b] at (287.5,329.2) {};
  \node[b] at (287.5,224.3) {};
  \node[b] at (386.6,118.3) {};
  \node[b] at (348.4,329.2) {};
  \node[b] at (348.4,224.3) {};
  \node[b] at (546.8,117.6) {};
  \node[b] at (447.6,117.6) {};
  \node[b] at (348.4,409.6) {};
  \node[b] at (348.4,304.7) {};
  \node[b] at (546.8,198.1) {};
  \node[b] at (447.6,198.1) {};
  \node[v=clM] at (49.9,405.9) {};
  \node[v=clB] at (49.9,301) {};
  \node[v=clB] at (49.9,194) {};
  \node[v=clM] at (49.9,365.7) {};
  \node[v=clB] at (49.9,260.8) {};
  \node[v=clM] at (49.9,153.8) {};
  \node[v=clM] at (15.9,332.2) {};
  \node[v=clB] at (15.9,227.3) {};
  \node[v=clM] at (15.9,120.3) {};
  \node[v=clM] at (83.9,332.2) {};
  \node[v=clB] at (83.9,227.3) {};
  \node[v=clM] at (83.9,120.3) {};
  \node[v=clM] at (151.7,354.9) {};
  \node[v=clB] at (151.7,250) {};
  \node[v=clB] at (250.9,142.8) {};
  \node[v=clM] at (250.3,378.1) {};
  \node[v=clB] at (250.3,273.2) {};
  \node[v=clM] at (349.4,166.9) {};
  \node[v=clM] at (348.1,368.2) {};
  \node[v=clB] at (348.1,263.3) {};
  \node[v=clM] at (546.5,156.1) {};
  \node[v=clB] at (447.3,156.1) {};
  \node[v=clM] at (382,368.2) {};
  \node[v=clB] at (382,263.3) {};
  \node[v=clB] at (580.5,156.1) {};
  \node[v=clM] at (481.2,156.1) {};
  \node[v=clM] at (314,368.2) {};
  \node[v=clB] at (314,263.3) {};
  \node[v=clB] at (512.4,156.1) {};
  \node[v=clM] at (413.2,156.1) {};
  \node[v=clM] at (250.3,329.7) {};
  \node[v=clB] at (250.3,224.9) {};
  \node[v=clB] at (349.4,118.6) {};
  \node[v=clM] at (151.7,408) {};
  \node[v=clB] at (151.7,303.2) {};
  \node[v=clM] at (250.9,196) {};
  \node[v=clM] at (148.9,194) {};
  \node[v=clM] at (50.3,90.9) {};
  \node[v=clB] at (148.9,153.8) {};
  \node[v=clM] at (50.3,50.7) {};
  \node[v=clB] at (115,120.3) {};
  \node[v=clB] at (16.3,17.2) {};
  \node[v=clB] at (183,120.3) {};
  \node[v=clB] at (84.3,17.2) {};
\end{tikzpicture}
}} at (M.center);
 \end{pgfonlayer}
 
 \begin{pgfonlayer}{foreground}
 \end{pgfonlayer}

 \end{tikzpicture}}

 \caption{The $15$ colorful graphs in $\Ocal^{1}|_{[2]}.$ The unique colorful graph in $\Ocal^{1}|_{[2]}\setminus \tilde{\Ocal}^{1}$ is highlighted in the dashed square.
 The first row depicts the four colorful graphs in $\Ocal^{1}|_{[1]}$.}
 \label{fig_Oq1}%
\end{figure}
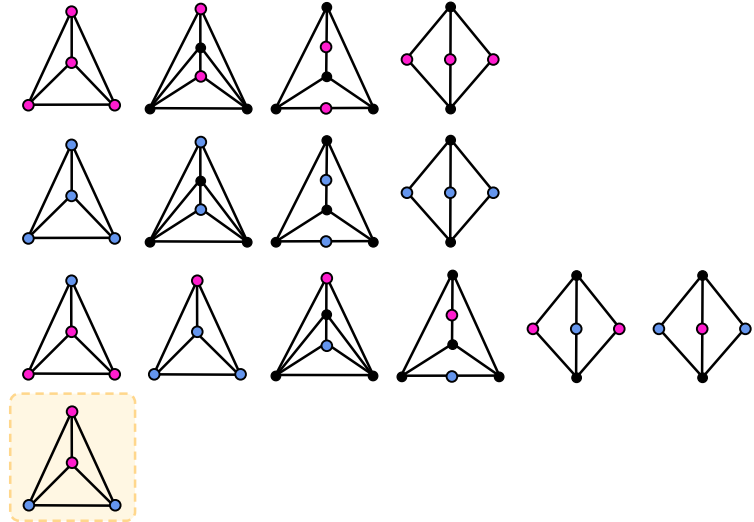

We denote by $K_{5}^{-}$ (resp. $K_{3,3}^-$) the graph obtained from $K_{5}$ (resp. $K_{3,3}$) by deleting a single edge.
We define $\Ocal^{1}$ to be the set of all the colorful graphs $(K_{5}^{-},\pi_{1}),$ $(K_{4},\pi_{2}), (K_{3,3}^{-},\pi_{3}),$ and $(K_{2,3},\pi_{4}),$ where for each $i\in[4]$ the $\pi_i$ are defined so that,
\begin{itemize}
\item $\pi_{1}$ assigns one or two colors to the two vertices of $K_{5}^{-}$ that have degree $3$ and no vertex receives two colors at once -- all other vertices are assigned $\emptyset,$
\item $\pi_{2}$ assigns one, two, three, or four colors to the vertices of $K_{4}$ such that every vertex receives exactly one color,
\item $\pi_{3}$ assigns one or two colors to the two vertices of $K_{3,3}^{-}$ of degree $2$ such that no vertex receives $2$ colors at once -- all other vertices are assigned $\emptyset,$ and
\item $\pi_{4}$ assigns one, two, or three colors to the three vertices of $K_{2,3}$ of degree $2$ such that no vertex receives $2$ colors at once -- all other vertices are assigned $\emptyset.$
\end{itemize}

\begin{observation}\label{obs_Oq1anti_chain}
The family $\Ocal^{0} \cup\Ocal^{1}$ is an anti-chain for the colorful minor relation.
\end{observation}
 
\begin{lemma}\label{lemma_colorFacial}%
A colorful graph $(H,\psi)$ is color-facial if and only if it does not contain a member of $\Ocal^{0} \cup \Ocal^{1}$ as a colorful minor.
\end{lemma}

\begin{proof}
Let $H^+$ be the graph obtained from $H$ by adding a new vertex $v$ and making it adjacent to all vertices $u\in V(H)$ with $\psi(u)\neq\emptyset.$

($\Rightarrow$) Suppose $(H,\psi)$ is color-facial.
Since the property of being color-facial is closed under colorful minors, every colorful minor of $(H,\psi)$ should also be color-facial.
Notice that none of the colorful graphs in $\Ocal^{0} \cup \Ocal^{1}$ are color-facial, therefore all of them are excluded from $(H,\psi)$ as colorful minors.

($\Leftarrow$)
We prove the contrapositive and assume that $(H,\psi)$ is not color-facial.
Then $H^+$ is not planar, since otherwise, by considering a plane embedding of $H^+$ and removing $v,$ all colored vertices of $(H,\psi)$ would lie on the boundary of a single face, that is, $(H,\psi)$ would be color-facial, a contradiction.
As $H^+$ is not planar, \zcref{obs_Wagner} implies that it contains a subgraph $S$ that is either a subdivision of $K_{5}$ or a subdivision of $K_{3,3}.$
{We call a vertex of $S$ a \emph{branch vertex} if it has degree at least $3$ in $S,$ otherwise we call it a \emph{subdivision vertex}.}
If $v$ is not a vertex of $S,$ then $S$ should be a subgraph of $H$ as well, therefore $H$ contains $K_{5}$ or $K_{3,3}$ as a minor, which means that $(H,\psi)$ contains some of the graphs in $\Ocal^{0}$ as a colorful minor.
Therefore we may assume that $v\in V(S)$ and distinguish cases:

Suppose first that $S$ is a subdivision of $K_{5}$ and that $v$ is a subdivision vertex of $S,$ respectively a branch vertex of $S.$
Then the $2,$ respectively $4,$ neighbors of $v$ in $S$ have a non-empty palette and, after deleting $v$ from $S$ and contracting the remaining subdivision paths, $(H,\psi)$ contains some $(K_{5}^{-},\pi_{1}),$ respectively some $(K_{4},\pi_{2}),$ as a colorful minor.

Suppose now that $S$ is a subdivision of $K_{3,3}$ and that $v$ is a subdivision vertex of $S,$ respectively a branch vertex of $S.$
Then the $2,$ respectively $3,$ neighbors of $v$ in $S$ have a non-empty palette and, in the same way, $(H,\psi)$ contains some $(K^{-}_{3,3},\pi_{3}),$ respectively some $(K_{2,3},\pi_{4}),$ as a colorful minor.
\end{proof}

We define $\tilde{\Ocal}^{1}\subseteq \Ocal^{1}$ by discarding all colorful graphs $(H,\psi)\in\Ocal^{1}$ with $|\psi(V(H))|\geq 3,$ and every $(K_{4},\pi_{2})$ whose vertex set can be partitioned into two sets $X_1$ and $X_2$ of size $2$ such that there are two distinct colors $i$ and $j$ with $\pi_{2}(X_1)=\{i\}$ and $\pi_{2}(X_2)=\{j\}.$
For example, the only discarded colorful graph with palette a subset of $[2]$ is the one in the dashed square in \zcref{fig_Oq1}.

Notice that, for every $I\finsub \Nbbb_{\geq1}$ with $q\coloneqq |I|,$ the set $\tilde{\Ocal}^{1}|_{I}$ contains $q+\binom{q}{2}$ variants of the graphs $(K_{5}^{-},\pi_{1})$ and $(K_{3,3}^{-},\pi_{3})$ each.
Moreover, $(K_{4},\pi_{2})$ and $(K_{2,3},\pi_{4})$ give rise to a total of $q+2\binom{q}{2}$ colorful graphs each.
In total, $\big|\tilde{\Ocal}^{1}|_{I}\big|=4q+6\binom{q}{2}.$

 \paragraph{Color-segmentation.}
We call a set $X\subseteq V(H)$ \emph{connected} if $H[X]$ is connected, and we say that two disjoint sets $X,Y\subseteq V(H)$ are \emph{adjacent} if $H$ has an edge with one endpoint in $X$ and the other in $Y.$
We say that a colorful graph $(H,\psi)$ is \emph{color-segmented}
if the three following conditions are satisfied:
\begin{itemize}
\item[A.] there are no four pairwise disjoint non-empty connected sets $B_{1},B_{2},B_{3},B_{4}\subseteq V(H),$ where $B_{i}$ and $B_{i+1}$ are adjacent for every $i\in[4]$ and indices are taken modulo $4,$ and no colors $c_{i}\in\psi(B_{i}),$ $i\in[4],$ such that $\{c_{1},c_{3}\}\cap\{c_{2},c_{4}\}=\emptyset,$
\item[B.] there are no three pairwise disjoint, pairwise adjacent, non-empty connected sets $B_{1},B_{2},B_{3}\subseteq V(H)$ such that $|\psi(B_{i})|\geq 2$ for every $i\in[3],$ and 
\item[C.] there are no four pairwise disjoint non-empty connected sets $B_{0},B_{1},B_{2},B_{3}\subseteq V(H)$ such that, for every $i\in[3],$ the sets $B_{0}$ and $B_{i}$ are adjacent and $|\psi(B_{i})|\geq 2.$
\end{itemize}

\begin{figure}[ht]
 \centering
\scalebox{.95}{%
\resizebox{\textwidth}{!}{%
\begin{tikzpicture}[x=1pt,y=1pt,line width=2pt,
    v/.style ={circle,draw,line width=1.5pt,inner sep=0pt,minimum size=8.4pt,fill=#1},
    b/.style ={circle,draw,line width=1pt,inner sep=0pt,minimum size=7.5pt,fill=black},
    fam/.style={anchor=east,inner sep=0pt,font=\LARGE}]
  \draw[clYpale,fill=clYbg,line width=2pt,rounded corners=7pt,
        dash pattern=on 6.5pt off 4.5pt] (230.0,-174.0) rectangle (506.0,-82.0);
  \draw[clYpale,fill=clYbg,line width=2pt,rounded corners=7pt,
        dash pattern=on 6.5pt off 4.5pt] (230.0,-273.0) rectangle (874.0,-180.0);
  \draw[clYpale,fill=clYbg,line width=2pt,rounded corners=7pt,
        dash pattern=on 6.5pt off 4.5pt] (230.0,-373.0) rectangle (874.0,-280.0);
  \draw (0.0,-94.0) -- (34.0,-128.0) -- (0.0,-162.0) -- (-34.0,-128.0) -- cycle;
  \node[v=clM] at (0.0,-94.0) {};
  \node[v=clB] at (34.0,-128.0) {};
  \node[v=clM] at (0.0,-162.0) {};
  \node[v=clB] at (-34.0,-128.0) {};
  \draw (92.0,-94.0) -- (126.0,-128.0) -- (92.0,-162.0) -- (58.0,-128.0) -- cycle;
  \node[v=clM] at (92.0,-94.0) {};
  \node[v=clY] at (126.0,-128.0) {};
  \node[v=clM] at (92.0,-162.0) {};
  \node[v=clY] at (58.0,-128.0) {};
  \draw (184.0,-94.0) -- (218.0,-128.0) -- (184.0,-162.0) -- (150.0,-128.0) -- cycle;
  \node[v=clB] at (184.0,-94.0) {};
  \node[v=clY] at (218.0,-128.0) {};
  \node[v=clB] at (184.0,-162.0) {};
  \node[v=clY] at (150.0,-128.0) {};
  \draw (276.0,-94.0) -- (310.0,-128.0) -- (276.0,-162.0) -- (242.0,-128.0) -- cycle;
  \node[v=clM] at (276.0,-94.0) {};
  \node[v=clY] at (310.0,-128.0) {};
  \node[v=clB] at (276.0,-162.0) {};
  \node[v=clY] at (242.0,-128.0) {};
  \draw (368.0,-94.0) -- (402.0,-128.0) -- (368.0,-162.0) -- (334.0,-128.0) -- cycle;
  \node[v=clM] at (368.0,-94.0) {};
  \node[v=clB] at (402.0,-128.0) {};
  \node[v=clY] at (368.0,-162.0) {};
  \node[v=clB] at (334.0,-128.0) {};
  \draw (460.0,-94.0) -- (494.0,-128.0) -- (460.0,-162.0) -- (426.0,-128.0) -- cycle;
  \node[v=clB] at (460.0,-94.0) {};
  \node[v=clM] at (494.0,-128.0) {};
  \node[v=clY] at (460.0,-162.0) {};
  \node[v=clM] at (426.0,-128.0) {};
  \draw (0.0,-191.0) -- (-34.0,-261.5) -- (34.0,-261.5) -- cycle;
  \halfv{0.0}{-191.0}{5.63}{clM}{clB}
  \halfv{-34.0}{-261.5}{5.63}{clM}{clB}
  \halfv{34.0}{-261.5}{5.63}{clM}{clB}
  \draw (92.0,-191.0) -- (58.0,-261.5) -- (126.0,-261.5) -- cycle;
  \halfv{92.0}{-191.0}{5.63}{clM}{clY}
  \halfv{58.0}{-261.5}{5.63}{clM}{clY}
  \halfv{126.0}{-261.5}{5.63}{clM}{clY}
  \draw (184.0,-191.0) -- (150.0,-261.5) -- (218.0,-261.5) -- cycle;
  \halfv{184.0}{-191.0}{5.63}{clB}{clY}
  \halfv{150.0}{-261.5}{5.63}{clB}{clY}
  \halfv{218.0}{-261.5}{5.63}{clB}{clY}
  \draw (276.0,-191.0) -- (242.0,-261.5) -- (310.0,-261.5) -- cycle;
  \halfv{276.0}{-191.0}{5.63}{clM}{clB}
  \halfv{242.0}{-261.5}{5.63}{clM}{clB}
  \halfv{310.0}{-261.5}{5.63}{clM}{clY}
  \draw (368.0,-191.0) -- (334.0,-261.5) -- (402.0,-261.5) -- cycle;
  \halfv{368.0}{-191.0}{5.63}{clM}{clB}
  \halfv{334.0}{-261.5}{5.63}{clM}{clB}
  \halfv{402.0}{-261.5}{5.63}{clB}{clY}
  \draw (460.0,-191.0) -- (426.0,-261.5) -- (494.0,-261.5) -- cycle;
  \halfv{460.0}{-191.0}{5.63}{clM}{clB}
  \halfv{426.0}{-261.5}{5.63}{clM}{clY}
  \halfv{494.0}{-261.5}{5.63}{clM}{clY}
  \draw (552.0,-191.0) -- (518.0,-261.5) -- (586.0,-261.5) -- cycle;
  \halfv{552.0}{-191.0}{5.63}{clM}{clB}
  \halfv{518.0}{-261.5}{5.63}{clM}{clY}
  \halfv{586.0}{-261.5}{5.63}{clB}{clY}
  \draw (644.0,-191.0) -- (610.0,-261.5) -- (678.0,-261.5) -- cycle;
  \halfv{644.0}{-191.0}{5.63}{clM}{clB}
  \halfv{610.0}{-261.5}{5.63}{clB}{clY}
  \halfv{678.0}{-261.5}{5.63}{clB}{clY}
  \draw (736.0,-191.0) -- (702.0,-261.5) -- (770.0,-261.5) -- cycle;
  \halfv{736.0}{-191.0}{5.63}{clM}{clY}
  \halfv{702.0}{-261.5}{5.63}{clM}{clY}
  \halfv{770.0}{-261.5}{5.63}{clB}{clY}
  \draw (828.0,-191.0) -- (794.0,-261.5) -- (862.0,-261.5) -- cycle;
  \halfv{828.0}{-191.0}{5.63}{clM}{clY}
  \halfv{794.0}{-261.5}{5.63}{clB}{clY}
  \halfv{862.0}{-261.5}{5.63}{clB}{clY}
  \draw (0.0,-328.0) -- (0.0,-291.0);
  \draw (0.0,-328.0) -- (-33.5,-362.0);
  \draw (0.0,-328.0) -- (33.5,-362.0);
  \node[b] at (0.0,-328.0) {};
  \halfv{0.0}{-291.0}{5.63}{clM}{clB}
  \halfv{-33.5}{-362.0}{5.63}{clM}{clB}
  \halfv{33.5}{-362.0}{5.63}{clM}{clB}
  \draw (92.0,-328.0) -- (92.0,-291.0);
  \draw (92.0,-328.0) -- (58.5,-362.0);
  \draw (92.0,-328.0) -- (125.5,-362.0);
  \node[b] at (92.0,-328.0) {};
  \halfv{92.0}{-291.0}{5.63}{clM}{clY}
  \halfv{58.5}{-362.0}{5.63}{clM}{clY}
  \halfv{125.5}{-362.0}{5.63}{clM}{clY}
  \draw (184.0,-328.0) -- (184.0,-291.0);
  \draw (184.0,-328.0) -- (150.5,-362.0);
  \draw (184.0,-328.0) -- (217.5,-362.0);
  \node[b] at (184.0,-328.0) {};
  \halfv{184.0}{-291.0}{5.63}{clB}{clY}
  \halfv{150.5}{-362.0}{5.63}{clB}{clY}
  \halfv{217.5}{-362.0}{5.63}{clB}{clY}
  \draw (276.0,-328.0) -- (276.0,-291.0);
  \draw (276.0,-328.0) -- (242.5,-362.0);
  \draw (276.0,-328.0) -- (309.5,-362.0);
  \node[b] at (276.0,-328.0) {};
  \halfv{276.0}{-291.0}{5.63}{clM}{clB}
  \halfv{242.5}{-362.0}{5.63}{clM}{clB}
  \halfv{309.5}{-362.0}{5.63}{clM}{clY}
  \draw (368.0,-328.0) -- (368.0,-291.0);
  \draw (368.0,-328.0) -- (334.5,-362.0);
  \draw (368.0,-328.0) -- (401.5,-362.0);
  \node[b] at (368.0,-328.0) {};
  \halfv{368.0}{-291.0}{5.63}{clM}{clB}
  \halfv{334.5}{-362.0}{5.63}{clM}{clB}
  \halfv{401.5}{-362.0}{5.63}{clB}{clY}
  \draw (460.0,-328.0) -- (460.0,-291.0);
  \draw (460.0,-328.0) -- (426.5,-362.0);
  \draw (460.0,-328.0) -- (493.5,-362.0);
  \node[b] at (460.0,-328.0) {};
  \halfv{460.0}{-291.0}{5.63}{clM}{clB}
  \halfv{426.5}{-362.0}{5.63}{clM}{clY}
  \halfv{493.5}{-362.0}{5.63}{clM}{clY}
  \draw (552.0,-328.0) -- (552.0,-291.0);
  \draw (552.0,-328.0) -- (518.5,-362.0);
  \draw (552.0,-328.0) -- (585.5,-362.0);
  \node[b] at (552.0,-328.0) {};
  \halfv{552.0}{-291.0}{5.63}{clM}{clB}
  \halfv{518.5}{-362.0}{5.63}{clM}{clY}
  \halfv{585.5}{-362.0}{5.63}{clB}{clY}
  \draw (644.0,-328.0) -- (644.0,-291.0);
  \draw (644.0,-328.0) -- (610.5,-362.0);
  \draw (644.0,-328.0) -- (677.5,-362.0);
  \node[b] at (644.0,-328.0) {};
  \halfv{644.0}{-291.0}{5.63}{clM}{clB}
  \halfv{610.5}{-362.0}{5.63}{clB}{clY}
  \halfv{677.5}{-362.0}{5.63}{clB}{clY}
  \draw (736.0,-328.0) -- (736.0,-291.0);
  \draw (736.0,-328.0) -- (702.5,-362.0);
  \draw (736.0,-328.0) -- (769.5,-362.0);
  \node[b] at (736.0,-328.0) {};
  \halfv{736.0}{-291.0}{5.63}{clM}{clY}
  \halfv{702.5}{-362.0}{5.63}{clM}{clY}
  \halfv{769.5}{-362.0}{5.63}{clB}{clY}
  \draw (828.0,-328.0) -- (828.0,-291.0);
  \draw (828.0,-328.0) -- (794.5,-362.0);
  \draw (828.0,-328.0) -- (861.5,-362.0);
  \node[b] at (828.0,-328.0) {};
  \halfv{828.0}{-291.0}{5.63}{clM}{clY}
  \halfv{794.5}{-362.0}{5.63}{clB}{clY}
  \halfv{861.5}{-362.0}{5.63}{clB}{clY}
\end{tikzpicture}}
}
 \caption{The $26$ colorful graphs in $\Ocal^{2}|_{[3]}$. The  colorful graphs in $\Ocal^{2}|_{[3]}\setminus \tilde{\Ocal}^{2}$ are highlighted in the dashed boxes. By picking the first colorful graph of each row we collect the colorful graphs in $\Ocal^{2}|_{[2]}$.}
 
 \label{fig_Oq2}%
\end{figure}
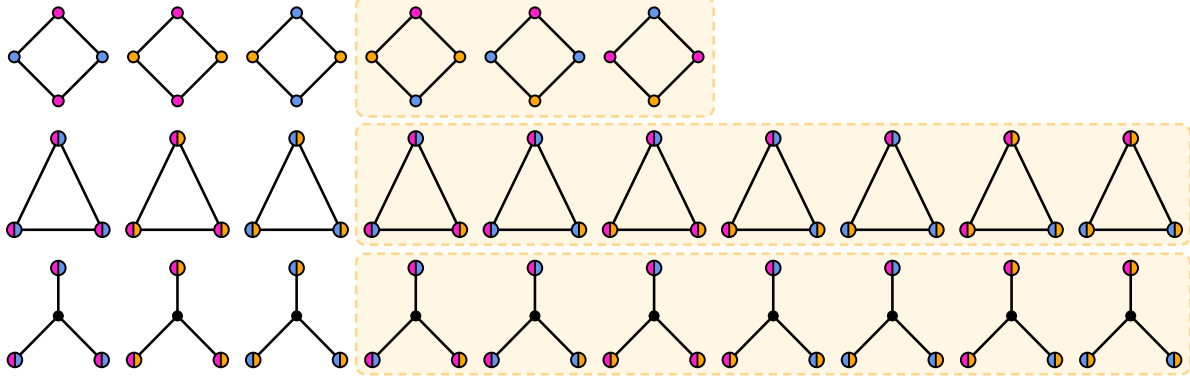

Let $\Ocal^{2}$ be the set containing
\begin{itemize}
\item every colorful graph $(C_{4},\sigma_1),$ where each vertex carries exactly one color and such that, for any two non-adjacent vertices, their palettes are disjoint from those carried by the other two.
 
\item every colorful graph $(K_{3},\sigma_2)$ where for each $v\in V(K_{3}),$ $|\sigma_{2}(v)|=2.$

\item every colorful graph $(K_{1,3},\sigma_3)$ where for each leaf $v$ of $K_{1,3},$ $|\sigma_{3}(v)|=2,$ and the center of $K_{1,3}$ has an empty palette.
\end{itemize}

See \zcref{fig_Oq2} for an illustration.

We now characterize color-segmented colorful graphs and establish closure.

\begin{observation}\label{obs_Oq2anti_chain}
The family $\Ocal^{2}$ is an anti-chain for the colorful minor relation.
\end{observation}

\begin{proof}
Let $(Z,\zeta),(Z',\zeta')\in\Ocal^{2}$ be such that $(Z,\zeta)\leq(Z',\zeta').$
As colorful minors do not increase the number of vertices and every member of $\Ocal^{2}$ has three or four vertices, we distinguish two cases.

Assume first that $Z=K_{3}.$
A model of $K_{3}$ consists of three pairwise adjacent branch sets, hence $Z'\neq K_{1,3},$ as $K_{1,3}$ is a forest and every colorful minor of a forest is a forest.
If $Z'=C_{4},$ then three singleton branch sets would induce a subgraph of $C_{4}$ on three vertices, which has at most two edges and is therefore not a triangle; hence one branch set consists of two vertices and the other two are singletons.
A singleton branch set of $(C_{4},\sigma_{1})$ carries exactly one color, while $\zeta$ assigns two colors to every vertex of $K_{3},$ and colors can only be removed, a contradiction.
Hence $Z'=K_{3},$ every branch set is a singleton, and $\zeta(v)\subseteq\zeta'(v)$ for the corresponding vertices.
As both sides have two elements, $(Z,\zeta)$ and $(Z',\zeta')$ are isomorphic.

Assume now that $Z\in\{C_{4},K_{1,3}\}.$
Then $|V(Z)|=|V(Z')|=4,$ so every branch set is a singleton, $Z$ is a spanning subgraph of $Z',$ and $\zeta(v)\subseteq\zeta'(v)$ for the corresponding vertices.
As $C_{4}$ contains a cycle while $K_{1,3}$ does not, and as $K_{1,3}$ has a vertex of degree three while $C_{4}$ has maximum degree two, this yields $Z=Z'.$
Both graphs have the same number of edges, hence $Z$ is $Z'$ itself, and the palettes, having the same size on every vertex, coincide.
\end{proof}

\begin{observation}\label{obs_colorSegmentation}
A colorful graph $(H,\psi)$ is color-segmented if and only if it excludes $\Ocal^{2}$ as a colorful minor.
\end{observation}

\begin{proof}
We first recall that a colorful graph $(Z,\zeta)$ is a colorful minor of $(H,\psi)$ if and only if there are pairwise disjoint non-empty connected sets $B_{v}\subseteq V(H),$ one for every $v\in V(Z),$ such that $B_{u}$ and $B_{v}$ are adjacent for every $uv\in E(Z)$ and $\zeta(v)\subseteq\psi(B_{v})$ for every $v\in V(Z).$
Indeed, contracting each $B_{v}$ to a single vertex yields a colorful graph in which that vertex carries the palette $\psi(B_{v}),$ after which the vertices outside $\bigcup_{v\in V(Z)}B_{v},$ the edges that are not edges of $Z,$ and the superfluous colors are deleted, respectively removed.

With this description, each of the three conditions is the negation of the containment of one of the three types of $\Ocal^{2}.$
The four sets of A are the branch sets of a model of $C_{4},$ in which $B_{i}$ retains the single color $c_{i};$ the resulting colorful graph belongs to $\Ocal^{2}$ precisely because $\{c_{1},c_{3}\}\cap\{c_{2},c_{4}\}=\emptyset,$ which is the condition imposed on $\sigma_{1}.$
The three sets of B are the branch sets of a model of $K_{3},$ in which each $B_{i}$ retains two of its colors, and every such choice yields a member of $\Ocal^{2}$ of the second type.
The four sets of C are the branch sets of a model of $K_{1,3},$ with $B_{0}$ that of the center, from which all colors are removed, and $B_{1},B_{2},B_{3}$ those of the leaves, each retaining two of its colors.
Conversely, the branch sets of a model of a member of $\Ocal^{2}$ of the first, second, or third type witness the failure of A, B, or C respectively.
Hence $(H,\psi)$ satisfies A, B, and C if and only if it excludes $\Ocal^{2}$ as a colorful minor.
\end{proof}

\begin{observation}\label{obs_segmentationClosed}
The class of color-segmented colorful graphs is closed under colorful minors.
\end{observation}

\begin{proof}
By \zcref{obs_colorSegmentation}, color-segmented graphs exclude $\Ocal^2$. Since the 
colorful minor relation is transitive, any colorful minor of such a graph also excludes 
$\Ocal^2$, hence is color-segmented.
\end{proof}

Notice that A, B, and C are stated in terms of pairwise disjoint connected sets and not of single vertices, and that this is what makes them closed under colorful minors.
Stated for single vertices, they would be satisfied by the colorful graph obtained from a cycle on four uncolored vertices by attaching to its vertices, in cyclic order, four pendant vertices carrying the colors $1,$ $2,$ $1,$ and $2,$ whereas contracting the four pendant edges produces a member of $\tilde{\Ocal}^{2}.$\medskip

In analogy to the definition of $\tilde{\Ocal}^{1},$ we define $\tilde{\Ocal}^{2}$ to be the set obtained from $\Ocal^{2}$ by discarding all $(H,\psi)\in\Ocal^{2}$ with $|\psi(V(H))|\geq 3.$

It follows that, for every $I\finsub \Nbbb_{\geq1}$ with $q\coloneqq |I|,$ each of the basic types, $(C_{4},\sigma_1),$ $(K_{3},\sigma_2),$
and $(K_{1,3},\sigma_3),$ contributes a total of $\binom{q}{2}$ colorful graphs to the set $\tilde{\Ocal}^{2}|_{I}.$
Therefore, $\big|\tilde{\Ocal}^{2}|_{I}\big|=3\binom{q}{2}.$

 \paragraph{Excluding three colors from a component.}
We say that a colorful graph $(G,\chi)$ is \emph{component-wise bicolored} if for every component $C$ of $G,$ $|\chi(V(C))|\leq 2.$

We define $\Ocal^{3}$ as the set containing every colorful graph $(K_{1},\pi)$ where $\pi$ assigns three colors to the unique vertex of $K_{1}.$
For instance, a visualization of $\Ocal^{3}|_{[4]}$ is
$$\Ocal^{3}|_{[4]} = \{\YMB,\YGB,\YMG,\GMB\}. $$

Clearly, for every $I\finsub \Nbbb_{\geq1}$ with $q\coloneqq |I|,$ it holds that $\big|\Ocal^{3}|_{I}\big|=\binom{q}{3}.$
The next observations are straightforward.

\begin{observation}\label{obs_SinglecomBicoloredClosed}%
The class of component-wise bicolored colorful graphs is closed under colorful minors.
\end{observation}

\begin{observation}\label{obs_Qq3anti_chain}
The family $\Ocal^{3}$ is an anti-chain for the colorful minor relation.
\end{observation}

\begin{observation}\label{obs_SinglecomBicolored}%
A colorful graph is component-wise bicolored if and only if it does not contain any member of $\Ocal^{3}$ as a colorful minor.
\end{observation}

 \paragraph{Further restricting the colors of the components.}
We call a colorful graph $(G,\chi)$ {\emph{multicolored}} if $|\chi(V(G))|\geq 2.$ 

We say that a colorful graph $(G,\chi)$ is \emph{single-component bicolored} if for any two distinct components $C_1$ and $C_2$ of $G$ there are no four distinct colors $i_1,i_2,i_3,i_4$ such that $i_1,i_2\in \chi(V(C_1))$ and $i_3,i_4\in \chi(V(C_2)).$

We define $\Ocal^{4}$ to be the set containing every colorful graph $(2\cdot K_{1},\tau)$ where $\tau$ is such that we have $|\tau(u)|=2$ for both $u\in V(2\cdot K_1)$ and $|\tau(V(2\cdot K_1))|=4.$

As an example, consider the sets
\begin{align*}
\Ocal^{4}|_{[3]} &= \emptyset,\\
\Ocal^{4}|_{[4]} &= \{ \NodeYM~~\!\NodeBG, \NodeYB~~\!\NodeMG, \NodeYG~~\!\NodeMB \},\ \text{and}\\
\Ocal^{4}|_{[5]} &= \{ \NodeYM~~\!\NodeBG, \NodeYM~~\!\NodeBP, \NodeYM~~\!\NodeGP, \NodeYB~~\!\NodeMG, \NodeYB~~\!\NodeMP, \NodeYB~~\!\NodeGP,
              \NodeYG~~\!\NodeMB, \NodeYG~~\!\NodeMP,\\
          &\phantom{{}\!\!=\{}\ \NodeYG~~\!\NodeBP, \NodeYP~~\!\NodeMB, \NodeYP~~\!\NodeMG, \NodeYP~~\!\NodeBG,
              \NodeMB~~\!\NodeGP, \NodeMG~~\!\NodeBP, \NodeMP~~\!\NodeBG \}.
\end{align*}
Observe that, for every $I\finsub \Nbbb_{\geq1}$ with $q\coloneqq |I|,$ it holds that
$\big|\Ocal^{4}|_{I}\big|=\frac{1}{2}\binom{4}{2}\cdot\binom{q}{4}=3\binom{q}{4}.$

The next three observations are immediate.

\begin{observation}\label{obs_comBicolorClosed}%
The class of single-component bicolored colorful graphs is closed under colorful minors.
\end{observation}

\begin{observation}\label{obs_comBicoloranti_chain}
The family $\Ocal^{4}$ is an anti-chain for the colorful minor relation.
\end{observation}

\begin{observation}\label{obs_comBicolored}%
A colorful graph $(H,\psi)$ is single-component bicolored if and only if $(H,\psi)$ does not contain any member of $\Ocal^{4}$ as a colorful minor.\end{observation}

\subsection{Crucial colorful graphs}
\label{subsec_crucial_define}
Let $(H,\psi)$ be a colorful graph.
We say that $(H,\psi)$ is \emph{crucial} if the following conditions are satisfied:
\begin{enumerate}
\item $(H,\psi)$ is color-facial,
\item $(H,\psi)$ is color-segmented,
\item $(H,\psi)$ is component-wise bicolored, and 
\item $(H,\psi)$ is single-component bicolored.
\end{enumerate}

We use $\Qcal$ for the domain containing all 
colorful graphs that are crucial.

Recalling \zcref{obs_colorFacialClosed,obs_segmentationClosed,obs_SinglecomBicoloredClosed,obs_comBicolorClosed}, one can see immediately that crucial colorful graphs are also closed under taking colorful minors.

\begin{observation}\label{obs_CrucialClosed}%
The class of crucial colorful graphs is closed under colorful minors.
\end{observation}

We now need to justify why we excluded some colorful graphs from ${\Ocal}^{i}$ in order to define $\tilde{\Ocal}^{i},$ for $i\in[2].$

\begin{lemma}
\label{lemma_Q12qNotMinimal}%
The family $\tilde{\Ocal}^{1}$ contains exactly the members of ${\Ocal}^{1}$ that do not contain any member of ${\Ocal}^{2}\cup {\Ocal}^{3}$ as a colorful minor.
Moreover, $\tilde{\Ocal}^{2}$ contains exactly the members of ${\Ocal}^{2}$ that do not contain any member of ${\Ocal}^{3}$ as a colorful minor.
\end{lemma}

\begin{proof}
By the definition of $\tilde{\Ocal}^{1},$ if a colorful graph $(H,\psi)\in{\Ocal}^{1}$ is not contained in $\tilde{\Ocal}^{1},$ this is because either it is not component-wise bicolored or it is $(K_{4},\pi_{2})$ where two vertices carry the color $c_{1}$ and the other two carry a different color $c_{2}.$
In the first case, $(H,\psi)$ contains a colorful graph from $\Ocal^{3}$ as a colorful minor, while in the second, the colorful graph $(K_{4},\pi_{2})$ contains the colorful graph $(C_{4},\sigma_1)\in\Ocal^{2}$ as a colorful minor, where the two non-adjacent vertices $a,c\in V(C_4)$ satisfy $\sigma_{1}(a)=\sigma_{1}(c)=c_{1}$ and the other two, say $b,d,$ satisfy $\sigma_{1}(b)=\sigma_{1}(d)=c_{2}.$
Notice now that none of the colorful graphs in $\tilde{\Ocal}^{1}$ contain any of the 
colorful graphs in ${\Ocal}^{2}$ as a colorful minor.
This is because every cycle of $(K_{5}^{-},\pi_{1}),$ $(K_{4},\pi_{2}),$ $(K_{3,3}^{-},\pi_{3}),$ and $(K_{2,3},\pi_{4})\in \tilde{\Ocal}^{1}$ carries at most two colors and one of those colors appears at most once.
Also, as the same colorful graphs carry at most two colors, they cannot contain any of the colorful graphs in ${\Ocal}^{3}$ as colorful minors.

For the second statement, recall that in the definition of $\tilde{\Ocal}^{2}$ we discarded all members of ${\Ocal}^{2}$ carrying more than two colors and these are exactly those that contain some member of ${\Ocal}^{3}$ as a colorful minor.
\end{proof}

\subsection{Obstructions to cruciality}
\label{subsect_obs_cru}

We now define 
\begin{align*}
\Ocal\coloneqq\Ocal^{0}\cup \tilde{\Ocal}^{1}\cup\tilde{\Ocal}^{2}\cup\Ocal^{3}\cup\Ocal^{4}.
\end{align*}
Notice that the sets $\tilde{\Ocal}^{1},$ $\tilde{\Ocal}^{2},$ $\Ocal^{3},$ and $\Ocal^{4},$ and therefore also $\Ocal,$ are infinite sets of colorful graphs; however, for every $I\finsub \Nbbb_{\geq1},$ the set $\Ocal_{I}=\Ocal|_{I}$ is finite, as we compute below.

The following lemma is the main ingredient for the equivalence of the statements \ref{it_s_cru} and \ref{it_s_O} of \zcref{th_EP_single}.
\begin{lemma}\label{obs_CrucialMinimal}%
A colorful graph is crucial if and only if it does not contain any member of $\Ocal$ as a colorful minor. 
\end{lemma}

\begin{proof}
For the forward direction, let $(G,\chi)$ be a crucial colorful graph and assume, towards a contradiction, that it contains some member $(Z,\zeta)$ of $\Ocal$ as a colorful minor.
By \zcref{obs_CrucialClosed}, $(Z,\zeta)$ is crucial as well.
However, no member of $\Ocal$ is crucial.
Indeed, as every colorful graph contains itself as a colorful minor, \zcref{lemma_colorFacial} implies that no member of $\Ocal^{0}\cup \tilde{\Ocal}^{1}$ is color-facial, \zcref{obs_colorSegmentation} implies that no member of $\tilde{\Ocal}^{2}$ is color-segmented, \zcref{obs_SinglecomBicolored} implies that no member of $\Ocal^{3}$ is component-wise bicolored, and \zcref{obs_comBicolored} implies that no member of $\Ocal^{4}$ is single-component bicolored.
This contradiction proves that no crucial colorful graph contains a member of $\Ocal$ as a colorful minor.

For the converse direction, let $(G,\chi)$ be a colorful graph that does not contain any member of $\Ocal$ as a colorful minor.
We first prove that $(G,\chi)$ does not contain any member of $\Ocal^{0}\cup\Ocal^{1}\cup\Ocal^{2}\cup\Ocal^{3}\cup\Ocal^{4}$ as a colorful minor.
This is immediate for the members of $\Ocal^{0},$ $\Ocal^{3},$ and $\Ocal^{4},$ as these three sets are subsets of $\Ocal.$
Next, consider some $(Z,\zeta)\in\Ocal^{2}\setminus\tilde{\Ocal}^{2}.$
By \zcref{lemma_Q12qNotMinimal}, $(Z,\zeta)$ contains some member of $\Ocal^{3}$ as a colorful minor and therefore, by the transitivity of the colorful minor relation, $(G,\chi),$ which excludes every member of $\Ocal^{3},$ cannot contain $(Z,\zeta)$ as a colorful minor.
Hence $(G,\chi)$ does not contain any member of $\Ocal^{2}$ as a colorful minor.
Finally, consider some $(Z,\zeta)\in\Ocal^{1}\setminus\tilde{\Ocal}^{1}.$
Again by \zcref{lemma_Q12qNotMinimal}, $(Z,\zeta)$ contains some member of $\Ocal^{2}\cup\Ocal^{3}$ as a colorful minor and, as we have already proved that $(G,\chi)$ excludes every member of $\Ocal^{2}\cup\Ocal^{3},$ we conclude that $(G,\chi)$ does not contain any member of $\Ocal^{1}$ as a colorful minor.
Now \zcref{lemma_colorFacial} implies that $(G,\chi)$ is color-facial, \zcref{obs_colorSegmentation} that it is color-segmented, \zcref{obs_SinglecomBicolored} that it is component-wise bicolored and \zcref{obs_comBicolored} that it is single-component bicolored.
Hence, $(G,\chi)$ is crucial.
\end{proof}

We are also able to obtain the size of the sets $\Ocal_{I}$ precisely.
Our previous discussions on the families $\Ocal^{i}$ and $\tilde{\Ocal}^{i}$ imply that, for every $I\finsub \Nbbb_{\geq1}$ with $q\coloneqq |I|,$
\begin{align*}
|\Ocal_{I}|=2+4q+6\binom{q}{2}+3\binom{q}{2}+\binom{q}{3}
+3\binom{q}{4}= \frac{1}{24} \left( 3q^4 - 14q^3 + 129q^2 - 22q + 48 \right).
\end{align*}

Interestingly, these numbers, being polynomially bounded, are rather tame.
Indeed, for some small values of $q$ we get $|\Ocal_{[0]}|=2,$ $|\Ocal_{[1]}|=6,$ $|\Ocal_{[2]}|=19,$ $|\Ocal_{[3]}|=42,$ $|\Ocal_{[4]}|=79,$ $|\Ocal_{[5]}|=137,$ and $|\Ocal_{[6]}|=226.$

Using \zcref{obs_Oq1anti_chain,obs_Oq2anti_chain,obs_Qq3anti_chain,obs_comBicoloranti_chain} and the fact that none of the colorful
graphs in one of the sets in $\{\Ocal^{0}, \tilde{\Ocal}^{1},\tilde{\Ocal}^{2},\Ocal^{3},\Ocal^{4}\}$ is a colorful minor 
of a colorful graph of another set, we may also conclude the following.

\begin{corollary}\label{cor_obsQ}
The set $\Ocal$
is the obstruction set of the domain $\Qcal$ of crucial colorful graphs, i.e., $\obs(\Qcal)=\Ocal$.
\end{corollary}

\subsection{The domain \texorpdfstring{$\Ucal$}{U} and crucial colorful graphs}
\label{subsect_U_crucial}

We conclude by proving that the domain $\Ucal,$ defined in \zcref{subsec_Results}, is precisely the domain $\Qcal$ of crucial colorful graphs and by proving that every $I$-colorful member of $\Ucal$ is a colorful minor of every large enough $(I,k)$-segregated grid.
We use the following result of \cite{ProtopapasTW2025Colorful}.
Recall that, for every $I\finsub \Nbbb_{\geq1}$ with $|I|\leq 2$ and every $k\in\Nbbb_{\geq 1},$ there is a unique $(I,k)$-segregated grid.

\begin{proposition}[\cite{ProtopapasTW2025Colorful}]
\label{prop_gridFolding}
For every crucial colorful graph $(H,\psi)$ with $|\psi(H)|\leq 2,$ there is some $k\in\Nbbb_{\geq 1}$ such that $(H,\psi)$ is a colorful minor of the $(\psi(H),k)$-segregated grid.
\end{proposition}

We first observe that the graphs generating $\Ucal$ are themselves crucial.

\begin{observation}
\label{obs_generatorsCrucial}
For every $q\in\Nbbb_{\geq 3},$ every $k\in\Nbbb_{\geq 1},$ every $i\in[q],$ and every $Z\subseteq[q]$ with $|Z|=3,$ the colorful graphs $U^{q}_{i,k}$ and $T^{q}_{Z,k}$ are crucial.
\end{observation}

\begin{proof}
Each of $U^{q}_{i,k}$ and $T^{q}_{Z,k}$ is the disjoint union of $(I,k)$-segregated grids where $|I|\leq 2$ (see \zcref{fig_UsGrids,fig_Tgrids}).
Every such grid has a plane embedding where the first column, which contains all colored vertices, lies on the boundary of the outer face.
Embedding the grids of the disjoint union side by side, we obtain a plane embedding where all colored vertices lie on the boundary of a single face.
Hence both graphs are color-facial.

In each of A, B, and C the sets involved have connected union, as $B_{i}$ and $B_{i+1}$ are adjacent in A, the three sets are pairwise adjacent in B, and $B_{0}$ is adjacent to each of the others in C.
They therefore lie in a single component, so it suffices to verify the three conditions for one $(I,k)$-segregated grid $(G,\chi)$ with $|I|\leq 2.$

If $|I|\leq 1,$ then $|\chi(B)|\leq 1$ for every $B\subseteq V(G),$ so B and C hold, and there are no colors $c_{1},c_{2},c_{3},c_{4}$ in $I$ with $\{c_{1},c_{3}\}\cap\{c_{2},c_{4}\}=\emptyset,$ so A holds as well.
Assume now that $I=\{x,y\}$ and let $X$ and $Y$ be the sets of vertices colored $x$ and $y$ respectively; these are the two blocks of the first column and they contain all colored vertices of $(G,\chi).$
Consider the embedding of $G$ in a closed disk $\Delta$ whose boundary contains the outer cycle of the grid.
The first column is a subpath of that cycle, met in the order $v_{1},\dots,v_{2k},$ so we may choose two points $\alpha$ and $\omega$ of $\partial\Delta,$ neither of them a vertex, such that the two arcs $A$ and $A'$ of $\partial\Delta$ that they determine satisfy $X\subseteq A$ and $Y\subseteq A'$: for $\alpha$ we take a point between $v_{k}$ and $v_{k+1}$ and for $\omega$ a point of the part of the outer cycle that carries no colors.

Let $G^{+}$ be the graph obtained from $G$ by adding two vertices $a$ and $b,$ joining $a$ to every vertex of $G$ on $A,$ joining $b$ to every vertex of $G$ on $A',$ and adding the edge $ab.$
The complement of $\Delta$ in the sphere is an open disk, and an arc $\gamma$ joining $\alpha$ to $\omega$ inside it splits it into two disks bounded by $A\cup\gamma$ and by $A'\cup\gamma.$
Placing $a$ in the first and $b$ in the second, both stars can be drawn without crossings, and the edge $ab$ can be drawn in a neighborhood of $\alpha,$ which carries no vertex.
Hence $G^{+}$ is planar.
Note that every $B\subseteq V(G)$ meeting $X$ is adjacent to $a$ in $G^{+},$ and every $B$ meeting $Y$ is adjacent to $b.$

Suppose that A fails, as witnessed by $B_{1},B_{2},B_{3},B_{4}$ and $c_{1},c_{2},c_{3},c_{4}.$
As $\{c_{1},c_{3}\}\cap\{c_{2},c_{4}\}=\emptyset$ and every color is $x$ or $y,$ we may assume that $c_{1}=c_{3}=x$ and $c_{2}=c_{4}=y.$
Then $B_{1}$ and $B_{3}$ meet $X$ while $B_{2}$ and $B_{4}$ meet $Y,$ so contracting each $B_{\ell}$ in $G^{+}$ yields $K_{3,3}$ with parts $\{B_{1},B_{3},b\}$ and $\{B_{2},B_{4},a\},$ contradicting the planarity of $G^{+}.$

Suppose that B fails, as witnessed by $B_{1},B_{2},B_{3}.$
Then $\chi(B_{\ell})=\{x,y\}$ for every $\ell\in[3],$ so each $B_{\ell}$ meets both $X$ and $Y$ and is therefore adjacent in $G^{+}$ to both $a$ and $b.$
Contracting each $B_{\ell}$ yields $K_{5}$ on $\{B_{1},B_{2},B_{3},a,b\},$ again a contradiction.

Suppose finally that C fails, as witnessed by $B_{0},B_{1},B_{2},B_{3}.$
As above, each of $B_{1},B_{2},B_{3}$ is adjacent in $G^{+}$ to both $a$ and $b,$ and each is adjacent to $B_{0}$ by assumption.
Contracting each $B_{\ell}$ yields $K_{3,3}$ with parts $\{B_{1},B_{2},B_{3}\}$ and $\{B_{0},a,b\},$ a contradiction.

Hence both graphs are color-segmented.

Both graphs are component-wise bicolored, as every component carries at most two colors.
Finally, consider two distinct components $C_{1}$ and $C_{2}.$
In $U^{q}_{i,k},$ the palettes of $C_{1}$ and $C_{2}$ are subsets of $\{i,h_{1}\}$ and $\{i,h_{2}\}$ for some $h_{1},h_{2}\in[q]\setminus\{i\},$ while in $T^{q}_{Z,k}$ the palette of every multicolored component is a two-element subset of $Z.$
In both cases, at most three distinct colors appear in $\chi(V(C_{1}))\cup\chi(V(C_{2})),$ so no four distinct colors can be distributed among $C_{1}$ and $C_{2}$ as in the definition of single-component bicolored colorful graphs.
Hence both graphs are single-component bicolored and, therefore, crucial.
\end{proof}

The next observation isolates the combinatorial dichotomy behind the definition of $\Ucal.$

\begin{observation}
\label{obs_starTriangle}
Let $\Fcal$ be a non-empty family of two-element subsets of $\Nbbb$ where every two members intersect.
Then either some color is contained in every member of $\Fcal,$ or there is a set $Z\subseteq\Nbbb$ with $|Z|=3$ such that every member of $\Fcal$ is a subset of $Z.$
\end{observation}

\begin{proof}
Assume that no color is contained in every member of $\Fcal$ and pick some $A=\{a,b\}\in\Fcal.$
As $a$ is not contained in every member of $\Fcal,$ there is some $C\in\Fcal$ with $a\notin C$ and, as $A\cap C\neq\emptyset,$ we have $C=\{b,c\}$ for some color $c\notin\{a,b\}.$
Symmetrically, there is some $D\in\Fcal$ with $b\notin D$ and, as $D$ intersects both $A$ and $C,$ we have $D=\{a,c\}.$
Set $Z\coloneqq\{a,b,c\}$ and consider some $E\in\Fcal.$
If $E\not\subseteq Z,$ then $|E\cap Z|\leq 1$ and therefore $E$ is disjoint from one of $A,$ $C,$ and $D,$ a contradiction.
\end{proof}

We first observe that segregated grids contain all their ``sub-grids'' as colorful minors.

\begin{observation}
\label{obs_subgridMinor}
For every $I'\subseteq I\finsub \Nbbb_{\geq1}$ with $|I'|\leq 2$ and every $k'\leq k,$ the $(I',k')$-segregated grid is a colorful minor of every $(I,k)$-segregated grid.
\end{observation}

\begin{proof}
We use the following operation: given two consecutive rows of a grid, contracting the matching formed by the vertical edges between them yields a grid with one row less, where every merged vertex receives the union of the palettes of the two vertices it comes from.
Let $(G,\chi)$ be an $(I,k)$-segregated grid and remove from it all colors of $I\setminus I'.$
If $I'=\emptyset,$ it remains to delete boundary rows and columns until a $(k'\times k')$-grid remains.
Otherwise, apply the above operation to merge every row whose first vertex is uncolored into a neighboring row, until every remaining row starts with a colored vertex, and then merge, within each of the $|I'|$ remaining blocks, $k-k'$ of its rows into neighboring rows of the same block.
The first column of the resulting grid consists of $|I'|$ blocks of $k'$ consecutive vertices, one for each color of $I',$ and it remains to delete boundary columns until a $(|I'|k'\times |I'|k')$-grid remains.
As, for $|I'|\leq 2,$ there is a unique $(I',k')$-segregated grid, the claim follows.
\end{proof}

We are now ready to prove the announced characterization of $\Ucal.$

\begin{lemma}
\label{lemma_U_equals_Q}
A colorful graph is crucial if and only if it belongs to $\Ucal.$
In other words, $\Qcal=\Ucal.$
\end{lemma}

\begin{proof}
By \zcref{obs_generatorsCrucial}, every colorful graph in $\bigcup_{q\in\Nbbb_{\geq3}}\bigcup_{k\in\Nbbb_{\geq1}}\Ucal^{q}_{k}$ is crucial and therefore, by \zcref{obs_CrucialClosed}, so is every colorful graph in $\Ucal.$

For the other direction, let $(H,\chi)$ be a crucial colorful graph, let $(H_{1},\chi),\ldots,(H_{r},\chi)$ be its components and choose some $q\in\Nbbb_{\geq3}$ with $\chi(H)\subseteq[q].$
Besides \zcref{obs_subgridMinor}, we use the easily verifiable fact that if each colorful graph of a finite collection is a colorful minor of the corresponding member of another collection, then the disjoint union of the former is a colorful minor of the disjoint union of the latter.

As $(H,\chi)$ is component-wise bicolored, $|\chi(V(H_{\ell}))|\leq 2$ for every $\ell\in[r].$
Let $\Fcal\coloneqq\{\chi(V(H_{\ell}))\mid \ell\in[r] \text{ and } |\chi(V(H_{\ell}))|=2\}.$
As $(H,\chi)$ is single-component bicolored, every two members of $\Fcal$ intersect: two disjoint members would be the palettes of two distinct components carrying four distinct colors, distributed as in the definition of single-component bicolored colorful graphs.
If $\Fcal=\emptyset,$ we fix an arbitrary $j\in[q]$ and proceed as in Case 1 below; otherwise, we apply \zcref{obs_starTriangle} to $\Fcal$ and distinguish the two corresponding cases.

\medskip
\noindent\textit{Case 1: there is a color $j$ contained in every member of $\Fcal.$}
Fix some $h_{0}\in[q]\setminus\{j\}.$
For every $h\in[q]\setminus\{j\},$ let $(J_{h},\chi)$ be the disjoint union of all components $(H_{\ell},\chi)$ with $\chi(V(H_{\ell}))\in\{\{j,h\},\{h\}\}$ and, in the case where $h=h_{0},$ additionally of all components with $\chi(V(H_{\ell}))\subseteq\{j\}.$
This defines a partition of the components of $(H,\chi)$ and each $(J_{h},\chi)$ is a colorful minor of $(H,\chi),$ hence crucial by \zcref{obs_CrucialClosed}, with palette a subset of $\{j,h\}.$
By \zcref{prop_gridFolding}, \zcref{obs_subgridMinor}, and the fact above, there is a $k\in\Nbbb_{\geq1}$ such that, for every $h\in[q]\setminus\{j\},$ the colorful graph $(J_{h},\chi)$ is a colorful minor of the $(\{j,h\},k)$-segregated grid.
Therefore $(H,\chi)$ is a colorful minor of $U^{q}_{j,k}.$

\medskip
\noindent\textit{Case 2: there is a set $Z\subseteq[q]$ with $|Z|=3$ such that every member of $\Fcal$ is a subset of $Z.$}
Fix some $d_{0}\in[q].$
For every $P\subseteq Z$ with $|P|=2,$ let $(J_{P},\chi)$ be the disjoint union of all components $(H_{\ell},\chi)$ with $\chi(V(H_{\ell}))=P$ and, for every $d\in[q],$ let $(J_{d},\chi)$ be the disjoint union of all components with $\chi(V(H_{\ell}))=\{d\}$ and, in the case where $d=d_{0},$ additionally of all components with $\chi(V(H_{\ell}))=\emptyset.$
Again, this defines a partition of the components of $(H,\chi)$ into colorful minors of $(H,\chi)$ that are crucial and whose palettes contain at most two colors.
By \zcref{prop_gridFolding}, \zcref{obs_subgridMinor}, and the fact above, there is a $k\in\Nbbb_{\geq1}$ such that $(H,\chi)$ is a colorful minor of $T^{q}_{Z,k}.$

\medskip
In both cases, $(H,\chi)\in\Ucal_{q}\subseteq\Ucal.$
\end{proof}

The next lemma shows that every large enough segregated grid over $I$ hosts the disjoint unions of small segregated grids that arise from crucial colorful graphs.
Notice that the assumption that the two-element sets $I_{t}$ pairwise intersect cannot be dropped: two pieces whose palettes are disjoint two-element sets that interleave along the first column can never be found in pairwise disjoint subgraphs of a segregated grid, by planarity.

\begin{figure}[t]
\centering
\scalebox{1.15}{%
 \scalebox{.9}{%
\begin{tikzpicture}[x=6.8681pt,y=6.8172pt,
    gline/.style ={line width=0.808pt},
    band/.style  ={line width=3.230pt,draw=blue!40},
    dot/.style   ={circle,fill,draw,line width=0.606pt,inner sep=0pt,minimum size=2.253pt},
    cdot/.style  ={circle,fill=#1,draw,line width=0.606pt,inner sep=0pt,minimum size=3.400pt},
    brc/.style   ={decorate,line width=0.606pt,decoration={brace,amplitude=#1}}]
  \begin{scope}[shift={(29.55,0)}]
    \fill[pkgray] (-0.47,15.63) rectangle (14.31,19.35);
    \fill[pkgray] (7.70,3.63) rectangle (15.30,19.34);
    \fill[pkgray] (-0.44,3.63) rectangle (14.31,7.36);
    \fill[pkgray] (-0.38,19.62) rectangle (16.91,23.47);
    \fill[pkgray] (15.71,2.01) rectangle (23.34,20.73);
    \fill[pkgray] (-0.02,-0.43) rectangle (16.01,3.36);
    \fill[blue!25] (-0.47,7.65) rectangle (7.36,15.34);
    \fill[blue!25] (7.70,3.64) rectangle (15.28,7.45);
    \fill[blue!25] (-0.48,3.64) rectangle (0.55,7.36);
    \fill[blue!25] (7.70,15.63) rectangle (15.30,19.35);
    \fill[blue!25] (-0.49,15.64) rectangle (0.56,19.35);
    \fill[blue!25] (15.72,-0.43) rectangle (23.33,3.38);
    \fill[blue!25] (-0.48,-0.43) rectangle (0.55,3.35);
    \fill[blue!25] (15.72,19.61) rectangle (23.34,23.48);
    \fill[blue!25] (-0.48,19.63) rectangle (0.56,23.48);
    \foreach \r in {0,1,2,3,20,21,22,23}{\draw[band] (0.45,\r) -- (15.92,\r);}
    \foreach \r in {4,5,6,7,16,17,18,19}{\draw[band] (0.48,\r) -- (7.63,\r);}
    \foreach \c in {16,...,23}{\draw[band] (\c,3.35) -- (\c,19.62);}
    \foreach \c in {8,...,15}{\draw[band] (\c,7.37) -- (\c,15.63);}
    \foreach \r in {0,...,23}{\draw[gline] (0,\r) -- (23,\r);}
    \foreach \c in {0,...,23}{\draw[gline] (\c,0) -- (\c,23);}
    \foreach \c in {1,...,7}{\foreach \r in {8,...,15}{\node[dot] at (\c,\r) {};}}
    \foreach \c in {8,...,15}{\foreach \r in {4,...,7}{\node[dot] at (\c,\r) {};}}
    \foreach \c in {8,...,15}{\foreach \r in {16,...,19}{\node[dot] at (\c,\r) {};}}
    \foreach \c in {16,...,23}{\foreach \r in {0,...,3}{\node[dot] at (\c,\r) {};}}
    \foreach \c in {16,...,23}{\foreach \r in {20,...,23}{\node[dot] at (\c,\r) {};}}
    \foreach \r in {12,...,23}{\node[cdot=pkmagenta] at (0,\r) {};}
    \foreach \r in {0,...,11}{\node[cdot=pkcyan] at (0,\r) {};}
    \draw[brc=3.84pt] (-.582,-0.43) -- (-.582,3.32);
    \node[anchor=east,font=\footnotesize] at (-1.25,1.44) {$k$};
    \draw[brc=3.84pt] (-.582,3.68) -- (-.582,7.44);
    \node[anchor=east,font=\footnotesize] at (-1.25,5.56) {$k$};
    \draw[brc=3.84pt] (-.582,7.64) -- (-.582,11.39);
    \node[anchor=east,font=\footnotesize] at (-1.25,9.52) {$k$};
    \draw[brc=3.84pt] (-.582,11.60) -- (-.582,15.35);
    \node[anchor=east,font=\footnotesize] at (-1.25,13.47) {$k$};
    \draw[brc=3.84pt] (-.582,15.67) -- (-.582,19.42);
    \node[anchor=east,font=\footnotesize] at (-1.25,17.55) {$k$};
    \draw[brc=3.84pt] (-.582,19.56) -- (-.582,23.31);
    \node[anchor=east,font=\footnotesize] at (-1.25,21.43) {$k$};
    \draw[brc=8.08pt] (-2.32,11.53) -- (-2.32,23.37);
    \node[anchor=east,font=\footnotesize] at (-3.62,17.45) {$mk$};
    \draw[brc=4.64pt] (7.29,-.57) -- (-0.23,-.57);
    \node[anchor=north,font=\small] at (3.53,-1.35) {$2k$};
    \draw[brc=4.64pt] (15.32,-.57) -- (7.80,-.57);
    \node[anchor=north,font=\small] at (11.56,-1.35) {$2k$};
    \draw[brc=4.64pt] (23.35,-.57) -- (15.83,-.57);
    \node[anchor=north,font=\small] at (19.59,-1.35) {$2k$};
  \end{scope}
\end{tikzpicture}}
}
\caption{The $(\{1,2\},mk)$-segregated grid and $m$ pairwise disjoint subgraphs in it, each containing the $(\{1,2\},k)$-segregated grid as a colorful minor, for $k=4$ and $m=3$: the highlighted regions are formed by the teeth, arms, and bodies in the proof of \zcref{lemma_gridHosting}.}
\label{fig_1segregatedGridsPack}
\end{figure}
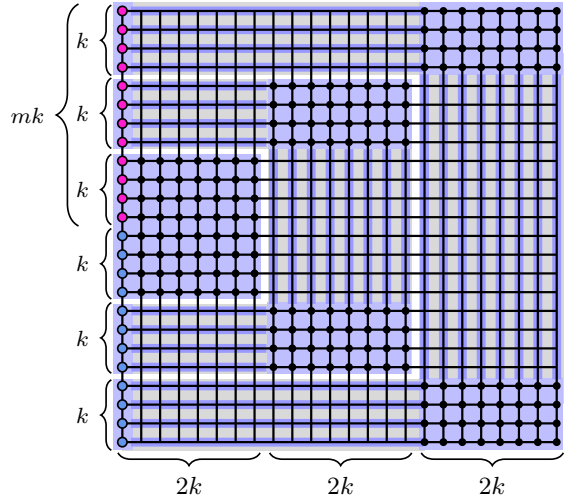

\begin{lemma}
\label{lemma_gridHosting}
Let $I\finsub\Nbbb_{\geq1}$ with $q\coloneqq|I|\geq 1,$ let $k,m\in\Nbbb_{\geq1},$ and let $(W,\omega)$ be the disjoint union of $m$ colorful graphs $(W_{1},\omega),\ldots,(W_{m},\omega)$ where, for every $t\in[m],$ the colorful graph $(W_{t},\omega)$ is an $(I_{t},k)$-segregated grid for some $I_{t}\subseteq I$ with $1\leq|I_{t}|\leq 2,$ and where every two of the sets $I_{1},\ldots,I_{m}$ with two elements intersect.
Then $(W,\omega)$ is a colorful minor of every $(I,K)$-segregated grid with $K\geq 2mkq+1.$
\end{lemma}

\begin{proof}
Let $(G,\chi)$ be an $(I,K)$-segregated grid with $K\geq 2mkq+1,$ together with a plane embedding where the first column, seen as a path $P,$ lies on the boundary of the outer face and, for every $a\in I,$ call the set of vertices of $P$ carrying the color $a$ the \emph{block} of $a.$
We refer to the colorful graphs $(W_{1},\omega),\ldots,(W_{m},\omega)$ as the \emph{pieces}.
For every $t\in[m]$ and every $a\in I_{t},$ we choose an interval of $k$ consecutive vertices of the block of $a,$ called a \emph{tooth} of the $t$-th piece, so that all chosen teeth are pairwise disjoint; this is possible, as every block consists of $K\geq mk$ vertices.
The \emph{span} of a piece is the minimal subpath of $P$ containing its teeth.
We first claim that the teeth can be chosen so that every two spans are either vertex-disjoint or such that one of them contains all teeth of the other.
Indeed, apply \zcref{obs_starTriangle} to the two-element sets among $I_{1},\ldots,I_{m},$ if there are at least two of them; if all of these sets contain a common color $j,$ split the corresponding pieces into two groups, according to whether the block of their second color appears before or after the block of $j$ on $P,$ let the $j$-teeth of the first group precede those of the second and, within each group, nest the pieces by assigning the $j$-teeth in the reverse of the order in which the teeth of the second colors appear on $P$ (see \zcref{fig_manu_ole}).
If, instead, these sets are the three two-element subsets of a set $\{a,b,c\}$ whose blocks appear in this order on $P,$ choose the teeth of the pieces with $I_{t}=\{a,c\}$ extremally, that is, before all other teeth in the block of $a$ and after all other teeth in the block of $c,$ nest these pieces among themselves, nest the pieces with $I_{t}=\{a,b\}$ and, separately, those with $I_{t}=\{b,c\}$ among themselves and let the $b$-teeth of the former precede the $b$-teeth of the latter (see \zcref{fig_Thosting}).
The teeth of the pieces with $|I_{t}|=1$ may be placed in any pairwise disjoint position, as their spans consist of a single tooth and therefore contain no teeth of other pieces.

For every $t\in[m],$ let now $d_{t}$ be the number of pieces whose teeth are all contained in the span of the $t$-th piece, other than the $t$-th piece itself, and let $D_{t}$ be the set of columns of $G$ with indices in $[2+2kd_{t},1+2k(d_{t}+1)].$
As $d_{t}\leq m-1,$ the largest index occurring in $D_{t}$ is at most $1+2km,$ which is at most the number $qK$ of columns of $G,$ by the choice of $K.$
We define $G_{t}$ as the subgraph of $G$ induced by the union of the teeth of the $t$-th piece, the \emph{arms} of the $t$-th piece, that is, for every vertex $v$ of a tooth of the $t$-th piece, the vertices of the row of $v$ lying in the columns $2$ up to $\max(D_{t}),$ and the \emph{body} of the $t$-th piece, that is, the vertices lying in the columns of $D_{t}$ whose rows meet the span of the $t$-th piece; see \zcref{fig_1segregatedGridsPack}.
We claim that $G_{1},\ldots,G_{m}$ are pairwise vertex-disjoint.
Teeth are pairwise disjoint by construction, arms of distinct pieces lie in distinct rows, and bodies of distinct pieces lie either in distinct columns or, if $d_{s}=d_{t},$ in disjoint spans, hence in distinct rows.
Finally, the arm of a vertex $v$ of the $t$-th piece meets the columns of $D_{s}$ of the body of a different piece only when $d_{s}\leq d_{t},$ and in this case the span of the $s$-th piece avoids $v$: if the two spans are disjoint this is clear and otherwise, as $d_{s}\leq d_{t},$ all teeth of the $s$-th piece are contained in the span of the $t$-th piece, so, by the choice of the teeth above, the span of the $s$-th piece lies strictly between the teeth of the $t$-th piece and, in particular, avoids $v.$

It remains to observe that, for every $t\in[m],$ the colorful graph $(G_{t},\chi)$ contains $(W_{t},\omega)$ as a colorful minor: contracting every arm onto its tooth vertex makes the vertices of the teeth adjacent, in the columns of $D_{t},$ to the rows of the body corresponding to their own rows; contracting away the rows of the body that correspond to no tooth, as in the proof of \zcref{obs_subgridMinor} and deleting all superfluous vertices then yields the $(I_{t},k)$-segregated grid.
Therefore $(W,\omega)$ is a colorful minor of $(G,\chi).$
\end{proof}

\begin{figure}[t]
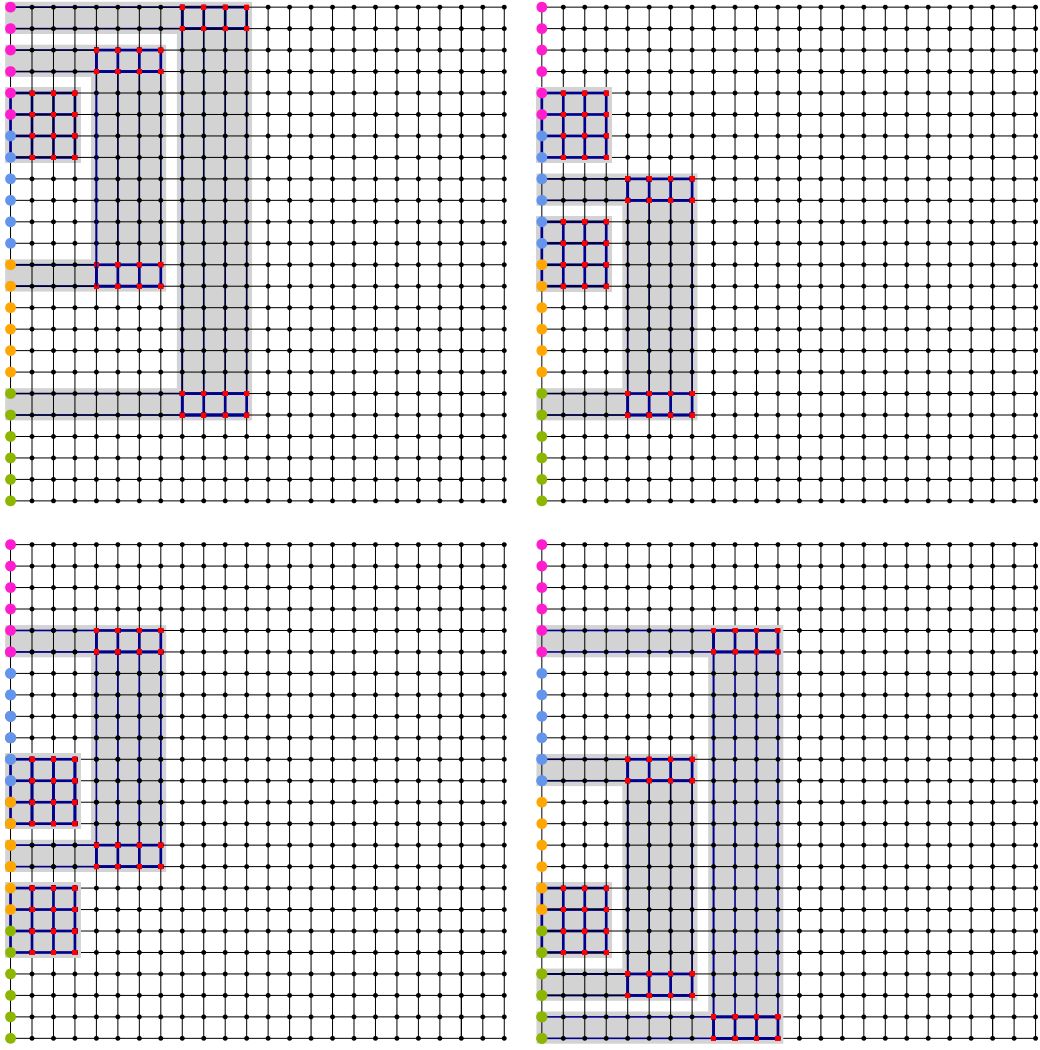

\centering
\scalebox{1}{%
\scalebox{0.505}{%
}
}
\caption{The graphs $U^{4}_{i,2},$ for $i\in[4],$ seen as colorful minors of the $([4],6)$-segregated grid, following the construction in the proof of \zcref{lemma_gridHosting}.}
\label{fig_manu_ole}
\end{figure}

\begin{figure}[t]
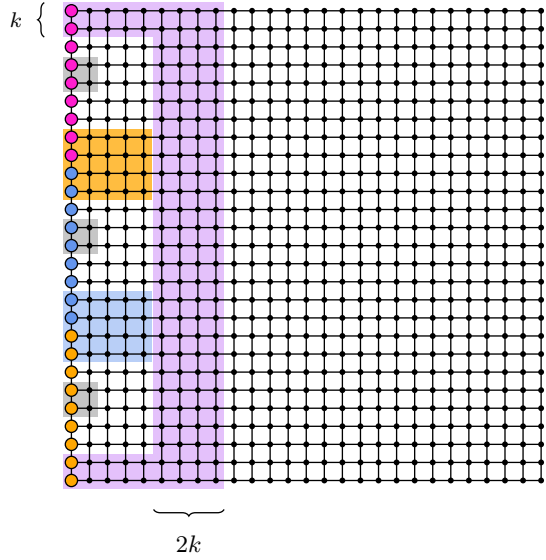

\centering
\scalebox{1}{%
\scalebox{1.0}{%
%
}
}
\caption{The colorful graph $T^{3}_{[3],2}$ as a colorful minor of the $([3],9)$-segregated grid, following the construction in the proof of \zcref{lemma_gridHosting}: the three pair grids and the three singleton grids of $T^{3}_{[3],2}$ appear in pairwise disjoint subgraphs and the piece with palette $\{1,3\},$ having the largest depth, is routed around all other pieces.}
\label{fig_Thosting}
\end{figure}

\begin{observation}
\label{obs_U_decomposition}
Let $I\finsub\Nbbb_{\geq1}$ with $q\coloneqq|I|\geq1$ and let $(W,\omega)$ be an $I$-colorful graph in $\Ucal.$
Then there are some $k\in\Nbbb_{\geq1}$ and at most $q+3$ colorful graphs, each of which is an $(I',k)$-segregated grid for some $I'\subseteq I$ with $1\leq|I'|\leq2,$ 
 such that every two of their palettes that have two elements intersect and such that their disjoint union contains $(W,\omega)$ as a colorful minor.
\end{observation}

\begin{proof}
By \zcref{lemma_U_equals_Q}, $(W,\omega)$ is crucial and, as $(W,\omega)$ is component-wise bicolored, the palette of every component of $(W,\omega)$ is a subset of $I$ with at most two elements.
If $q=1,$ then, by \zcref{prop_gridFolding}, $(W,\omega)$ is a colorful minor of the $(I,k)$-segregated grid for some $k\in\Nbbb_{\geq1}$ and this single colorful graph is as required.
Assume now that $q\geq 2$ and argue as in the proof of \zcref{lemma_U_equals_Q}: the two-element palettes of the components of $(W,\omega)$ pairwise intersect and, by \zcref{obs_starTriangle}, they either all contain a common color $j\in I,$ in which case we partition the components of $(W,\omega)$ into the classes of Case 1 of that proof, with $[q]$ replaced by $I$ (see \zcref{fig_manu_ole}), or they are the two-element subsets of a set $Z\subseteq I$ with $|Z|=3,$ in which case we partition them into the classes of Case 2 (see \zcref{fig_Thosting}).
In both cases we obtain, by \zcref{prop_gridFolding} and \zcref{obs_subgridMinor}, a $k\in\Nbbb_{\geq1}$ and at most $q+3$ colorful graphs whose palettes are subsets of $I$ with at most two elements and pairwise intersect whenever they have two elements.
It remains to observe that their palettes may be taken non-empty.
Indeed, a class all of whose components carry no color yields, by \zcref{prop_gridFolding}, a colorful graph that is a colorful minor of the $(\emptyset,k)$-segregated grid and, by \zcref{obs_subgridMinor}, the latter is a colorful minor of the $(I',k)$-segregated grid for every non-empty $I'\subseteq I;$ choosing for $I'$ the palette that the class in question would have received had one of its components been colored keeps the required intersection property.
\end{proof}

\begin{lemma}
\label{lemma_U_in_grids}
For every $I\finsub \Nbbb_{\geq1}$ and every $I$-colorful graph $(W,\omega)\in\Ucal,$ there is some $K_{W}\in\Nbbb_{\geq1}$ such that $(W,\omega)$ is a colorful minor of every $(I,K)$-segregated grid with $K\geq K_{W}.$
\end{lemma}

\begin{proof}
Set $q\coloneqq|I|.$
If $I=\emptyset,$ then $(W,\omega)$ is a planar graph and the claim follows from the well-known fact that every planar graph is a minor of every large enough grid.
Assume now that $I\neq\emptyset$ and let $k\in\Nbbb_{\geq1}$ and the at most $q+3$ colorful graphs be as in \zcref{obs_U_decomposition}.
As their disjoint union contains $(W,\omega)$ as a colorful minor, the claim follows from \zcref{lemma_gridHosting} with $K_{W}\coloneqq 2(q+3)kq+1.$ 
\end{proof}

\section{The Erd\H{o}s-P{\'o}sa property}
\label{sec_EP}

In this section we prove the implication from \ref{it_s_cru} to \ref{it_s_ep} of \zcref{th_EP_single}: every crucial colorful graph has the Erd\H{o}s-P{\'o}sa property (\zcref{thm_EP_positive}).
Combined with the results of \zcref{sec_obstructions}, this establishes all implications of \zcref{th_EP_single} except for the one from \ref{it_s_ep} to \ref{it_s_cru}, which is proved in \zcref{sec_EP_to_crucial}.

\subsection{Models, color restriction, and a reduction to bounded torso treewidth}
\label{subsec_EP_prelim}

Let $(H,\psi)$ and $(G,\chi)$ be colorful graphs with $V(H)\neq\emptyset.$
A \emph{model} of $(H,\psi)$ in $(G,\chi)$ is a family $\{D_{u}\}_{u\in V(H)}$ of pairwise disjoint non-empty sets of vertices of $G,$ called the \emph{branch sets} of the model, such that $G[D_{u}]$ is connected for every $u\in V(H),$ there is an edge of $G$ with one endpoint in $D_{u}$ and one endpoint in $D_{v}$ for every edge $uv\in E(H)$ and, for every $u\in V(H)$ and every $i\in\psi(u),$ some vertex of $D_{u}$ carries the color $i.$

\begin{observation}
\label{obs_minorModels}
Let $(H,\psi)$ and $(G,\chi)$ be colorful graphs with $V(H)\neq\emptyset.$
Then $(H,\psi)$ is a colorful minor of $(G,\chi)$ if and only if there is a model of $(H,\psi)$ in $(G,\chi).$
\end{observation}

\begin{proof}
If there is a model $\{D_{u}\}_{u\in V(H)},$ then we may delete all vertices of $G$ outside its branch sets, contract, within each branch set $D_{u},$ all edges of a spanning tree of $G[D_{u}],$ and then delete all superfluous edges and remove all superfluous colors.
The palette of the vertex resulting from $D_{u}$ is the union of the palettes of the vertices of $D_{u},$ hence contains $\psi(u),$ and the resulting colorful graph is $(H,\psi).$

For the converse, we proceed by induction on the number of operations transforming $(G,\chi)$ into $(H,\psi).$
If no operation is applied, then the singletons $\{u\},$ $u\in V(H),$ form a model.
Otherwise, let $(G',\chi')$ be the colorful graph obtained from $(G,\chi)$ by the first operation; by the induction hypothesis, there is a model $\{D_{u}\}_{u\in V(H)}$ of $(H,\psi)$ in $(G',\chi').$
If the operation is the deletion of a vertex or of an edge, or the removal of a color, then $\{D_{u}\}_{u\in V(H)}$ is also a model in $(G,\chi).$
If the operation contracts an edge $vw$ of $G$ into a vertex $x_{vw},$ then replacing, in every branch set containing it, the vertex $x_{vw}$ by the two vertices $v$ and $w$ yields a model in $(G,\chi)$: the modified branch sets remain connected, every adjacency via $x_{vw}$ is realized via $v$ or $w,$ and every color carried by $x_{vw}$ is carried by $v$ or by $w.$
\end{proof}

For a colorful graph $(G,\chi)$ and a set $I\finsub\Nbbb_{\geq1},$ we denote by $(G,\chi)\cap I$ the colorful graph $(G,\chi_{I})$ where $\chi_{I}(v)\coloneqq\chi(v)\cap I$ for every $v\in V(G).$

\begin{observation}
\label{obs_colorRestriction}
Let $(H,\psi)$ and $(G,\chi)$ be colorful graphs, where $V(H)\neq\emptyset,$ and let $Q\coloneqq\psi(H).$
Then, for every $k\in\Nbbb,$ the packings of $(H,\psi)$ in $(G,\chi)$ of size $k$ are precisely the packings of $(H,\psi)$ in $(G,\chi)\cap Q$ of size $k$ and the coverings of $(H,\psi)$ in $(G,\chi)$ of size $k$ are precisely the coverings of $(H,\psi)$ in $(G,\chi)\cap Q$ of size $k.$
\end{observation}

\begin{proof}
By \zcref{obs_minorModels}, it suffices to observe that, for every subgraph $J$ of $G,$ a family is a model of $(H,\psi)$ in $(J,\chi)$ if and only if it is a model of $(H,\psi)$ in $(J,\chi)\cap Q.$
Indeed, the only condition of a model that involves colors requires, for every $u\in V(H)$ and every $i\in\psi(u)\subseteq Q,$ that some vertex of $D_{u}$ carries the color $i,$ and, for every $i\in Q,$ a vertex of $J$ carries $i$ in $(J,\chi)$ if and only if it carries $i$ in $(J,\chi)\cap Q.$
\end{proof}

The next lemma reduces the study of packings and coverings of a crucial colorful graph to hosts of bounded torso treewidth.

\begin{lemma}
\label{lemma_EP_reduce}
For every crucial colorful graph $(H,\psi)$ with $V(H)\neq\emptyset$ and every $k\in\Nbbb_{\geq1},$ there is some $c\in\Nbbb$ such that every colorful graph $(G,\chi)$ satisfies one of the following, where $Q\coloneqq\psi(H)$:
\begin{enumerate}
\item\label{it_reduce_pack} $(G,\chi)$ has a packing of $(H,\psi)$ of size $k,$ or
\item\label{it_reduce_torso} there is a set $X\subseteq V(G)$ such that the treewidth of the torso of $X$ in $G$ is at most $c$ and, for every component $J$ of $G-X,$ the colorful graph $(J,\chi)\cap Q$ is $Q$-restricted.
\end{enumerate}
\end{lemma}

\begin{proof}
Set $q\coloneqq|Q|.$
By \zcref{lemma_U_equals_Q}, $(H,\psi)\in\Ucal$ and, as $(H,\psi)$ is $Q$-colorful, \zcref{lemma_U_in_grids} provides some $K_{H}\in\Nbbb_{\geq1}$ such that $(H,\psi)$ is a colorful minor of every $(Q,K)$-segregated grid with $K\geq K_{H}.$

We first claim that there are some $K'\in\Nbbb_{\geq1}$ and a colorful graph $(W_{k},\omega_{k})$ such that $(W_{k},\omega_{k})$ is the disjoint union of $k$ colorful graphs, each of which contains $(H,\psi)$ as a colorful minor and such that $(W_{k},\omega_{k})$ is a colorful minor of every $(Q,K)$-segregated grid with $K\geq K'.$
If $q=0,$ then the $(Q,K)$-segregated grid is the $(K\times K)$-grid without colors and, for $K\geq kK_{H},$ it contains $k$ pairwise vertex-disjoint $(K_{H}\times K_{H})$-grids, lying in pairwise disjoint sets of columns; as each of them is a $(Q,K_{H})$-segregated grid, the claim follows with $K'\coloneqq kK_{H}$ and $(W_{k},\omega_{k})$ the disjoint union of $k$ copies of the $(Q,K_{H})$-segregated grid.
Assume now that $q\geq1.$
By \zcref{obs_U_decomposition}, applied to $(H,\psi),$ there are some $k_{0}\in\Nbbb_{\geq1}$ and at most $q+3$ colorful graphs, each of which is an $(I',k_{0})$-segregated grid for some $I'\subseteq Q$ with $1\leq|I'|\leq2,$ whose two-element palettes pairwise intersect and whose disjoint union contains $(H,\psi)$ as a colorful minor.
Let $(W_{k},\omega_{k})$ be the disjoint union of $k$ copies of this collection.
It consists of at most $k(q+3)$ segregated grids as above whose two-element palettes still pairwise intersect and therefore, by \zcref{lemma_gridHosting}, $(W_{k},\omega_{k})$ is a colorful minor of every $(Q,K)$-segregated grid with $K\geq K'\coloneqq2k(q+3)k_{0}q+1.$ 
This proves the claim.

Set $c\coloneqq\sg(q,K')$ and let $(G,\chi)$ be a colorful graph; write $(G,\chi_{Q})\coloneqq(G,\chi)\cap Q.$
If $\chi_{Q}(G)\subsetneq Q,$ then every component $J$ of $G$ satisfies $\chi_{Q}(J)\subseteq\chi_{Q}(G)\subsetneq Q,$ so $(J,\chi)\cap Q$ is $Q$-restricted and \ref{it_reduce_torso} holds with $X=\emptyset.$
Otherwise, $(G,\chi_{Q})$ is $Q$-colorful.
If the $Q$-torso treewidth of $(G,\chi_{Q})$ is smaller than $c,$ then a witnessing set $X$ yields \ref{it_reduce_torso}.
Otherwise, by \zcref{prop_torsoGrid}, $(G,\chi_{Q})$ contains some $(Q,K')$-segregated grid as a colorful minor and therefore, by the claim and the transitivity of the colorful minor relation, $(W_{k},\omega_{k})$ is a colorful minor of $(G,\chi_{Q}).$
Fix, by \zcref{obs_minorModels}, a model of $(W_{k},\omega_{k})$ in $(G,\chi_{Q})$ and, for every $t\in[k],$ let $G_{t}$ be the subgraph of $G$ induced by the union of the branch sets of the vertices of the $t$-th of the $k$ colorful graphs forming $(W_{k},\omega_{k}).$
The subgraphs $G_{1},\ldots,G_{k}$ are pairwise vertex-disjoint and, for every $t\in[k],$ the corresponding sub-family of branch sets is a model in $(G_{t},\chi_{Q})$ of the $t$-th of these colorful graphs, which contains $(H,\psi)$ as a colorful minor.
Hence, by \zcref{obs_colorRestriction} applied to the host $(G_{t},\chi)$ with packings of size one, the colorful graph $(H,\psi)$ is a colorful minor of $(G_{t},\chi)$ for every $t\in[k],$ so that $(G_{1},\ldots,G_{k})$ is a packing of $(H,\psi)$ in $(G,\chi)$ of size $k$ and \ref{it_reduce_pack} holds.
\end{proof}

\subsection{The classic scheme}
\label{subsec_EP_scheme}

The engine of our covering arguments is the following lemma, which converts the outcome of \zcref{lemma_EP_reduce} into a covering by a standard argument on a tree decomposition of the torso.

\begin{lemma}
\label{lemma_classic_scheme}
Let $I\finsub\Nbbb_{\geq1},$ let $t\in\Nbbb$ and $k\in\Nbbb_{\geq1},$ let $(H,\psi)$ be a colorful graph with $V(H)\neq\emptyset$ and $\psi(H)\subseteq I,$ and let $(G,\chi)$ be a colorful graph with $\chi(G)\subseteq I,$ together with a set $X\subseteq V(G)$ such that every component of $(G-X,\chi)$ is $I$-restricted and the treewidth of the torso of $X$ in $G$ is at most $t.$
Then one of the following is true:
\begin{enumerate}
\item\label{it_scheme_kill} there is a set $S\subseteq V(G)$ of size at most $(t+1)(k-1)$ such that no component of $(G-S,\chi)$ contains $(H,\psi)$ as a colorful minor,
\item\label{it_scheme_restricted} $I\neq\emptyset,$ $\psi(H)\subsetneq I,$ and there is a set $S\subseteq V(G)$ of size at most $(t+1)(k-1)$ such that every component of $(G-S,\chi)$ is $I$-restricted or does not contain $(H,\psi)$ as a colorful minor, or
\item\label{it_scheme_pack} there is a packing of $(H,\psi)$ in $(G,\chi)$ of size $k.$
\end{enumerate}
\end{lemma}

\begin{proof}
Throughout the proof we keep $(H,\psi)$ and $I$ fixed and, for every colorful graph $(G',\chi')$ with $G'\subseteq G,$ we call a subgraph $F$ of $G'$ a \emph{witness} in $(G',\chi')$ if $F$ is connected, the colorful graph $(F,\chi')$ is not $I$-restricted and $(H,\psi)$ is a colorful minor of $(F,\chi').$
We write $\mathsf{p}'(G',\chi')$ for the maximum number of pairwise vertex-disjoint witnesses in $(G',\chi')$ and $\mathsf{c}'(G',\chi')$ for the minimum size of a set $S\subseteq V(G')$ meeting every witness in $(G',\chi').$
Two observations will be used repeatedly.
First, since every witness is connected and since a component containing a witness is itself a witness, a set $S$ meets every witness in $(G',\chi')$ if and only if every component of $(G'-S,\chi')$ is $I$-restricted or does not contain $(H,\psi)$ as a colorful minor.
Second, every family of pairwise vertex-disjoint witnesses is a packing of $(H,\psi),$ so
\begin{align}
\label{eq_scheme_packing}
\text{if } \mathsf{p}'(G,\chi)\geq k, \text{ then \ref{it_scheme_pack} holds.}
\end{align}

\medskip\noindent\textbf{A tree decomposition adapted to $X.$}
We first record the decomposition that the hypothesis on $X$ provides.

\begin{claim}
\label{cl_scheme_decomposition}
There is a tree decomposition $(T,\beta)$ of $G$ of adhesion at most $t+1$ and a set $L$ of leaves of $T$ such that $|\beta(v)|\leq t+1$ for every $v\in V(T)\setminus L,$ the set $V(T)\setminus L$ is non-empty and, for every $d\in L$ with neighbor $v,$ the colorful graph $(G[\beta(d)\setminus\beta(v)],\chi)$ is $I$-restricted.
\end{claim}

\begin{proof}[Proof of the claim]
Let $(T_{0},\beta_{0})$ be a tree decomposition of $\torso(G,X)$ of width at most $t.$
For every component $C$ of $G-X,$ the set $N_{G}(C)$ is a clique of $\torso(G,X)$ and is therefore contained in $\beta_{0}(v_{C})$ for some $v_{C}\in V(T_{0}).$
Obtain $T$ from $T_{0}$ by adding, for every such component $C,$ a new leaf $d_{C}$ adjacent to $v_{C},$ and set $\beta(d_{C})\coloneqq V(C)\cup N_{G}(C)$ and $\beta(v)\coloneqq\beta_{0}(v)$ for every $v\in V(T_{0}).$
Then $(T,\beta)$ is a tree decomposition of $G$ and, setting $L\coloneqq\{d_{C}\mid C \text{ a component of } G-X\},$ we have $V(T)\setminus L\supseteq V(T_{0})\neq\emptyset$ and $|\beta(v)|\leq t+1$ for every $v\in V(T)\setminus L.$
Every bag of a node $d_{C}\in L$ meets the bag of its neighbor in $N_{G}(C),$ which has size at most $t+1,$ so the adhesion of $(T,\beta)$ is at most $t+1.$
Finally, $\beta(d_{C})\setminus\beta(v_{C})\subseteq V(C),$ so $(G[\beta(d_{C})\setminus\beta(v_{C})],\chi)$ is a subgraph of a component of $G-X$ and, as being $I$-restricted is preserved under taking subgraphs, it is $I$-restricted.
\end{proof}

\medskip\noindent\textbf{The key inequality.}
We prove that
\begin{align}
\label{eq_scheme_main}
\mathsf{c}'(G',\chi')\ \leq\ (t+1)\cdot\mathsf{p}'(G',\chi')
\end{align}
holds for every colorful graph $(G',\chi')$ with $G'\subseteq G$ that admits a tree decomposition and a set of leaves with the properties of \zcref{cl_scheme_decomposition}, by induction on $p\coloneqq\mathsf{p}'(G',\chi').$

If $p=0,$ then $(G',\chi')$ has no witness at all, so $S\coloneqq\emptyset$ meets every witness and $\mathsf{c}'(G',\chi')=0,$ as required.
Assume now that $p\geq1$ and let $(T,\beta)$ and $L$ be as in \zcref{cl_scheme_decomposition} for $(G',\chi').$
Root $T$ at a node of $V(T)\setminus L.$
For $v\in V(T),$ let $T_{v}$ be the subtree of $T$ rooted at $v$ and let $G'_{v}\coloneqq G'[\bigcup_{u\in V(T_{v})}\beta(u)].$
Since $G'_{\mathrm{root}}=G'$ contains a witness, we may choose $v\in V(T)$ of maximum depth such that $G'_{v}$ contains a witness of $(G',\chi'),$ and we fix such a witness $F_{0}\subseteq G'_{v}.$
Let
$$S_{0}\coloneqq\begin{cases} \beta(v) &\text{if } v\notin L,\\ \beta(v)\cap\beta(v') &\text{if } v\in L \text{ and } v' \text{ is the parent of } v,\end{cases}$$
so that $|S_{0}|\leq t+1,$ by the bound on the size of the bags of $V(T)\setminus L$ in the first case and by the bound on the adhesion of $(T,\beta)$ in the second.

\begin{claim}
\label{cl_scheme_nowitness}
No witness of $(G',\chi')$ is a subgraph of $G'_{v}-S_{0}.$
\end{claim}

\begin{proof}[Proof of the claim]
Suppose first that $v\notin L,$ so that $S_{0}=\beta(v),$ and let $F$ be a witness with $F\subseteq G'_{v}-\beta(v).$
Every component of $G'_{v}-\beta(v)$ is a subgraph of $G'_{c}$ for some child $c$ of $v$ and $F$ is connected, so $F\subseteq G'_{c}$ for some child $c$ of $v.$
This contradicts the choice of $v,$ as $c$ has larger depth than $v.$
Suppose now that $v\in L,$ with parent $v'.$
Then $G'_{v}-S_{0}=G'[\beta(v)\setminus\beta(v')]$ and $(G'[\beta(v)\setminus\beta(v')],\chi')$ is $I$-restricted by \zcref{cl_scheme_decomposition}.
As being $I$-restricted is preserved under taking subgraphs, no subgraph of $G'_{v}-S_{0}$ is a witness.
\end{proof}

Since $S_{0}$ separates $V(G'_{v})\setminus S_{0}$ from $V(G')\setminus V(G'_{v}),$ every connected subgraph of $G'-S_{0}$ is either a subgraph of $G'_{v}-S_{0}$ or vertex-disjoint from $G'_{v}.$
By \zcref{cl_scheme_nowitness}, no witness of $(G'-S_{0},\chi')$ is of the first kind, so every witness of $(G'-S_{0},\chi')$ is vertex-disjoint from $G'_{v}$ and hence from $F_{0}.$
Consequently, adding $F_{0}$ to any family of pairwise vertex-disjoint witnesses of $(G'-S_{0},\chi')$ yields a family of pairwise vertex-disjoint witnesses of $(G',\chi'),$ which gives
\begin{align}
\label{eq_scheme_drop}
\mathsf{p}'(G'-S_{0},\chi')\ \leq\ p-1.
\end{align}
Moreover, deleting the vertices of $S_{0}$ from every bag of $(T,\beta)$ yields a tree decomposition of $G'-S_{0}$ with the same set $L$ of leaves, whose adhesion and non-leaf bags only shrink and whose leaf parts, being subgraphs of the previous ones, remain $I$-restricted.
Hence the induction hypothesis applies to $(G'-S_{0},\chi')$ and, using \eqref{eq_scheme_drop},
$$\mathsf{c}'(G',\chi')\ \leq\ |S_{0}|+\mathsf{c}'(G'-S_{0},\chi')\ \leq\ (t+1)+(t+1)(p-1)\ =\ (t+1)p,$$
where the first inequality holds because, if $S_{1}$ meets every witness of $(G'-S_{0},\chi'),$ then $S_{0}\cup S_{1}$ meets every witness of $(G',\chi').$
This proves \eqref{eq_scheme_main}.

\medskip\noindent\textbf{Conclusion.}
By \eqref{eq_scheme_packing}, we may assume that $\mathsf{p}'(G,\chi)\leq k-1.$
Then \zcref{cl_scheme_decomposition} and \eqref{eq_scheme_main} provide a set $S\subseteq V(G)$ with $|S|\leq(t+1)(k-1)$ meeting every witness of $(G,\chi),$ that is, such that
\begin{align}
\label{eq_scheme_outcome}
\text{every component of $(G-S,\chi)$ is $I$-restricted or does not contain $(H,\psi)$ as a colorful minor.}
\end{align}
If no $I$-restricted colorful graph with palette a subset of $I$ contains $(H,\psi)$ as a colorful minor, then, by \eqref{eq_scheme_outcome}, no component of $(G-S,\chi)$ contains $(H,\psi)$ as a colorful minor and \ref{it_scheme_kill} holds.
Otherwise, $(H,\psi)$ is a colorful minor of some $I$-restricted colorful graph $(R,\psi')$ with $\psi'(R)\subseteq I.$
As $V(H)\neq\emptyset,$ also $V(R)\neq\emptyset$ and, as the only $\emptyset$-restricted colorful graph is the one with no vertices, we obtain $I\neq\emptyset.$
Moreover, $\psi(H)\subseteq\psi'(R)\subsetneq I$ and \eqref{eq_scheme_outcome} yields \ref{it_scheme_restricted}.
\end{proof}

\subsection{The connected case}
\label{subsec_EP_conn}

\begin{theorem}
\label{thm_EP_connected}
Every connected crucial colorful graph $(H,\psi)$ with $V(H)\neq\emptyset$ has the Erd\H{o}s-P{\'o}sa property.
\end{theorem}

\begin{proof}
Set $Q\coloneqq\psi(H)$ and $f(0)\coloneqq0$ and, given $k\in\Nbbb_{\geq1},$ let $c$ be as in \zcref{lemma_EP_reduce} for $(H,\psi)$ and $k,$ and set $f(k)\coloneqq(c+1)(k-1).$
Let $(G,\chi)$ be a colorful graph without a packing of $(H,\psi)$ of size $k$; we prove that $(G,\chi)$ has a covering of $(H,\psi)$ of size at most $f(k).$
As the empty family is a packing of size $0,$ we have $k\geq1.$
By \zcref{obs_colorRestriction}, the colorful graph $(G,\chi_{Q})\coloneqq(G,\chi)\cap Q$ has no packing of $(H,\psi)$ of size $k$ either and it suffices to find a covering of $(H,\psi)$ in $(G,\chi_{Q})$ of size at most $f(k).$
We apply \zcref{lemma_EP_reduce} to $(H,\psi),$ $(G,\chi),$ and $k.$
Its outcome \ref{it_reduce_pack} is excluded, as $(G,\chi)$ has no packing of $(H,\psi)$ of size $k,$ so its outcome \ref{it_reduce_torso} provides a set $X\subseteq V(G)$ such that the treewidth of the torso of $X$ in $G$ is at most $c$ and every component of $(G-X,\chi_{Q})$ is $Q$-restricted.
We apply \zcref{lemma_classic_scheme}, with $I\coloneqq Q,$ to $(H,\psi),$ $(G,\chi_{Q}),$ $X,$ and $k.$
The outcome \ref{it_scheme_pack} is excluded, as there is no packing of $(H,\psi)$ in $(G,\chi_{Q})$ of size $k,$ and the outcome \ref{it_scheme_restricted} asserts that $\psi(H)\subsetneq Q,$ which is false.
Hence \ref{it_scheme_kill} holds and provides a set $S$ of size at most $(c+1)(k-1)$ such that no component of $(G-S,\chi_{Q})$ contains $(H,\psi)$ as a colorful minor.
As $H$ is connected and $V(H)\neq\emptyset,$ the union of the branch sets of any model of $(H,\psi)$ induces a connected subgraph and therefore lies in a single component of $G-S.$
Hence, by \zcref{obs_minorModels}, $(G-S,\chi_{Q})$ does not contain $(H,\psi)$ as a colorful minor and $S$ is a covering of $(H,\psi)$ in $(G,\chi_{Q})$ of size at most $f(k).$
\end{proof}
\subsection{The general case}
\label{subsec_EP_general}

We now lift \zcref{thm_EP_connected} to crucial colorful graphs with an arbitrary number of components.
The additional difficulty is that a copy of a disconnected pattern may be dispersed around the host: distinct components of $(H,\psi)$ may be realized in distinct components of $(G,\chi),$ so removing a set of vertices after whose deletion no single component contains all of $(H,\psi)$ --- which is what the machinery of the connected case provides --- does not yet prevent copies from being recombined out of pieces lying in different components.
The proof is by induction on the number $r$ of components of $H$ and revolves around the pieces of $(H,\psi),$ the sub-patterns formed by the groups of a partition of its components.
After a first deletion as above, every surviving copy of $(H,\psi)$ is split by such a partition into at least two pieces, each realized in its own component of the host.
The heart of the argument is based on the interleaving  between packings and coverings for these pieces.
If every piece that could still occur is contained in many components of the host, then $k$ pairwise disjoint copies of $(H,\psi)$ can be assembled out of disjoint copies of complementary pieces, which is impossible in the absence of a packing.
Hence some piece is contained in few components and, inside each of them, the induction hypothesis either produces $k$ disjoint copies of that piece in that single component --- copies which serve the assembly of copies of $(H,\psi)$ just as well --- or a small covering that eliminates the piece from that component.
Treating one piece per round and never revisiting a treated piece, the process examines fewer than $2^{r}$ pieces in total and, once no copy of $(H,\psi)$ survives, the accumulated deletions form the desired covering.

\begin{theorem}
\label{thm_EP_positive}
Every crucial colorful graph has the Erd\H{o}s-P{\'o}sa property.
\end{theorem}

\begin{proof}
Let $(H,\psi)$ be a crucial colorful graph.
If $V(H)=\emptyset,$ then $(H,\psi)$ is a colorful minor of every colorful graph, so no colorful graph has a covering of $(H,\psi);$ on the other hand, as observed in \zcref{subsec_EP_def}, every colorful graph then has packings of $(H,\psi)$ of every size.
The defining implication of the EP-property is therefore vacuous and $(H,\psi)$ has the EP-property with any gap.
Assume now that $V(H)\neq\emptyset,$ let $H_{1},\ldots,H_{r}$ be the components of $H,$ and proceed by induction on $r.$
If $r=1,$ the claim is \zcref{thm_EP_connected}.

Assume that $r\geq2.$
For every non-empty $R\subseteq[r],$ let $(H_{R},\psi_{R})$ be the disjoint union of the colorful graphs $(H_{i},\psi),$ $i\in R;$ in particular, $(H_{[r]},\psi_{[r]})=(H,\psi).$
Every $(H_{R},\psi_{R})$ is a colorful minor of $(H,\psi)$ and therefore, by \zcref{obs_CrucialClosed}, crucial.
Let $\Rcal$ be the set of all non-empty proper subsets of $[r],$ so that $|\Rcal|=2^{r}-2.$
For every $R\in\Rcal,$ the colorful graph $(H_{R},\psi_{R})$ has fewer than $r$ components, so the induction hypothesis provides a gap $g_{R}$ for its Erd\H{o}s-P{\'o}sa property.
Given $k\in\Nbbb_{\geq1},$ let $c$ be as in \zcref{lemma_EP_reduce} for $(H,\psi)$ and $k,$ and set
$$f(k)\coloneqq(c+1)(k-1)+2^{2r}\cdot r\cdot k\cdot\sum_{R\in\Rcal}g_{R}(k);$$
set also $f(0)\coloneqq0.$
Let $(G,\chi)$ be a colorful graph without a packing of $(H,\psi)$ of size $k$; we prove that $(G,\chi)$ has a covering of $(H,\psi)$ of size at most $f(k).$
As the empty family is a packing of size $0,$ we have $k\geq1.$
By \zcref{obs_colorRestriction}, the colorful graph $(G,\chi')\coloneqq(G,\chi)\cap Q,$ where $Q\coloneqq\psi(H),$ has no packing of $(H,\psi)$ of size $k$ either and the coverings of $(H,\psi)$ in $(G,\chi')$ of any given size are precisely the coverings of $(H,\psi)$ in $(G,\chi)$; we may therefore work with $(G,\chi').$
Exactly as in the proof of \zcref{thm_EP_connected}, \zcref{lemma_EP_reduce} and \zcref{lemma_classic_scheme} provide a set $S_{0}$ of size at most $(c+1)(k-1)$ such that no component of $(G-S_{0},\chi')$ contains $(H,\psi)$ as a colorful minor.
Set $(G_{1},\chi'_{1})\coloneqq(G-S_{0},\chi').$
For every $R\in\Rcal,$ let $\Kcal_{R}$ be the set of components of $(G_{1},\chi'_{1})$ that contain $(H_{R},\psi_{R})$ as a colorful minor.
The set $\Kcal_{R}$ measures how widespread the piece $(H_{R},\psi_{R})$ is in $G_{1},$ and the construction below hinges on whether these sets are large or small.

\medskip\noindent\textbf{The partition property.}
Since no component of $G_{1}$ contains $(H,\psi),$ every copy of $(H,\psi)$ that survives further deletions must be scattered over several components; the following makes this precise.
Let $S\subseteq V(G_{1})$ and suppose that $(H,\psi)$ is a colorful minor of $(G_{1}-S,\chi'_{1}).$
In a model of $(H,\psi),$ the branch sets of the vertices of each component $H_{i}$ induce a connected subgraph and therefore lie in a single component of $G_{1}-S.$
Grouping the indices of $[r]$ according to that component, we obtain a partition $\Pcal$ of $[r]$ into non-empty parts and pairwise distinct components $C_{R},$ $R\in\Pcal,$ of $G_{1}-S$ such that $(H_{R},\psi_{R})$ is a colorful minor of $(C_{R},\chi'_{1})$ for every $R\in\Pcal.$
As every component of $G_{1}-S$ is a subgraph of a component of $G_{1}$ and no component of $G_{1}$ contains $(H,\psi),$ the partition $\Pcal$ has at least two parts; in particular, every part of $\Pcal$ belongs to $\Rcal.$

\medskip\noindent\textbf{The assembly argument.}
The partition property splits surviving copies of $(H,\psi)$ into pieces; the fact recorded next runs in the opposite direction, assembling $k$ disjoint copies of $(H,\psi)$ out of disjoint copies of pieces.
It will be used twice below, both times in order to derive a contradiction with the absence of a packing.
Let $\Pcal'$ be a partition of $[r]$ all of whose parts belong to $\Rcal$ and let
$P^{j}_{R},$ $R\in\Pcal'$ and $j\in[k],$ be subgraphs of $G_{1}$ such that
\begin{enumerate}
\item\label{it_ass_contains} for every $R\in\Pcal'$ and every $j\in[k],$ the colorful graph
$(P^{j}_{R},\chi'_{1})$ contains $(H_{R},\psi_{R})$ as a colorful minor,
\item\label{it_ass_row} for every $R\in\Pcal',$ the subgraphs $P^{1}_{R},\ldots,P^{k}_{R}$ are
pairwise vertex-disjoint and
\item\label{it_ass_col} for every two parts $R\neq R'$ of $\Pcal'$ and all $i,j\in[k],$ the
subgraphs $P^{i}_{R}$ and $P^{j}_{R'}$ are contained in different components of $G_{1}.$
\end{enumerate}
We claim that no such family exists.
For this, we set, for every $j\in[k],$
\begin{align}
\label{eq_ass_union}
U^{j}\ \coloneqq\ \bigcup_{R\in\Pcal'}P^{j}_{R}
\end{align}
and we prove that $U^{1},\ldots,U^{k}$ is a packing of $(H,\psi)$ in $(G,\chi')$ of size $k,$ contrary to the fact that $(G,\chi')$ has no such packing.
It may help to picture the family as an array whose rows are indexed by the parts of $\Pcal'$ and whose columns are indexed by $[k].$
The condition \ref{it_ass_row} concerns the rows: the subgraphs $P^{1}_{R},\ldots,P^{k}_{R}$ are $k$ places where one and the same piece $(H_{R},\psi_{R})$ is found.
The union in \eqref{eq_ass_union} is taken along a column: $U^{j}$ collects one place for each piece and therefore has a chance of hosting all of $(H,\psi).$

\begin{claim}
\label{cl_ass_contains}
For every $j\in[k],$ the colorful graph $(U^{j},\chi'_{1})$ contains $(H,\psi)$ as a colorful minor.
\end{claim}

\begin{proof}[Proof of the claim]
Fix $j\in[k].$
By \ref{it_ass_contains} and \zcref{obs_minorModels}, we may choose, for every $R\in\Pcal',$ a model
$\Mcal_{R}$ of $(H_{R},\psi_{R})$ in $(P^{j}_{R},\chi'_{1}).$
The branch sets of $\Mcal_{R}$ lie in $P^{j}_{R}$ and, by \ref{it_ass_col}, the subgraphs $P^{j}_{R},$
$R\in\Pcal',$ are pairwise vertex-disjoint, so the branch sets of $\Mcal_{R}$ and those of
$\Mcal_{R'}$ are disjoint whenever $R\neq R'.$
Hence $\bigcup_{R\in\Pcal'}\Mcal_{R}$ is a family of pairwise disjoint connected sets of vertices of
$U^{j}.$
As $\Pcal'$ is a partition of $[r],$ the colorful graph $(H,\psi)$ is the disjoint union of the
colorful graphs $(H_{R},\psi_{R}),$ $R\in\Pcal',$ and every edge of $H$ joins two vertices of the same
component of $H,$ hence of the same $(H_{R},\psi_{R}).$
Therefore every adjacency and every color required by $(H,\psi)$ is already realized inside one of the
models $\Mcal_{R}$ and $\bigcup_{R\in\Pcal'}\Mcal_{R}$ is a model of $(H,\psi)$ in
$(U^{j},\chi'_{1}).$
\end{proof}

\begin{claim}
\label{cl_ass_disjoint}
The subgraphs $U^{1},\ldots,U^{k}$ are pairwise vertex-disjoint.
\end{claim}

\begin{proof}[Proof of the claim]
Let $j\neq j'$ and consider $P^{j}_{R}$ and $P^{j'}_{R'}$ for some $R,R'\in\Pcal'.$
If $R=R',$ these two subgraphs are vertex-disjoint by \ref{it_ass_row} and, if $R\neq R',$ they are
contained in different components of $G_{1}$ by \ref{it_ass_col} and are, in particular,
vertex-disjoint.
As $U^{j}$ and $U^{j'}$ are unions of such subgraphs, they are vertex-disjoint.
\end{proof}

By \zcref{cl_ass_contains,cl_ass_disjoint}, the subgraphs $U^{1},\ldots,U^{k}$ are pairwise
vertex-disjoint subgraphs of $G_{1},$ hence of $G,$ each containing $(H,\psi)$ as a colorful minor,
so they form a packing of $(H,\psi)$ in $(G,\chi')$ of size $k,$ a contradiction.
Consequently, no family satisfying \ref{it_ass_contains}, \ref{it_ass_row} and \ref{it_ass_col}
exists.

In the two applications below, the subgraphs $P^{1}_{R},\ldots,P^{k}_{R}$ attached to a part
$R\in\Pcal'$ are obtained in one of the following two ways.
Either a single component $K$ of $G_{1}$ has a packing of $(H_{R},\psi_{R})$ of size $k$ and we let
$P^{1}_{R},\ldots,P^{k}_{R}$ be the $k$ members of this packing, or $k$ distinct components of $G_{1}$
each contain $(H_{R},\psi_{R})$ as a colorful minor and we let $P^{1}_{R},\ldots,P^{k}_{R}$ be these
components.
In both ways the conditions \ref{it_ass_contains} and \ref{it_ass_row} hold automatically and every
subgraph attached to $R$ is contained in a component of $G_{1},$ namely in the single component $K$ in
the first way and in its own component in the second.
The condition \ref{it_ass_col} therefore amounts to requiring that no component of $G_{1}$ is used for
two distinct parts of $\Pcal',$ and this is the only condition that will have to be verified.

\medskip\noindent\textbf{The iterative construction.}
The assembly argument converts abundance into a packing: if every relevant piece is contained in many
components of $G_{1},$ disjoint copies of $(H,\psi)$ can be assembled, which is impossible.
Scarcity, in turn, is what the induction hypothesis can exploit: a piece contained in few components
can be eliminated from all of them at bounded cost.
The construction below alternates between these two mechanisms, treating one scarce piece per round.
We first observe that some piece is scarce at the outset.

\begin{claim}
\label{cl_constr_start}
There is some $R_{1}\in\Rcal$ with $|\Kcal_{R_{1}}|<rk.$
\end{claim}

\begin{proof}[Proof of the claim]
Suppose towards a contradiction that $|\Kcal_{R'}|\geq rk$ for every $R'\in\Rcal$ and fix an arbitrary
$R\in\Rcal.$
As $R$ is a non-empty proper subset of $[r],$ so is $[r]\setminus R,$ hence $[r]\setminus R\in\Rcal$
as well and $\Pcal'\coloneqq\{R,[r]\setminus R\}$ is a partition of $[r]$ all of whose parts belong to
$\Rcal.$
Recall also that $r\geq2,$ so that $rk\geq2k.$

We choose $k$ distinct components $P^{1}_{R},\ldots,P^{k}_{R}$ in $\Kcal_{R},$ which is possible as
$|\Kcal_{R}|\geq2k\geq k,$ and then $k$ distinct components
$P^{1}_{[r]\setminus R},\ldots,P^{k}_{[r]\setminus R}$ in $\Kcal_{[r]\setminus R}$ that are different
from all of the former.
The latter choice is possible as well: the two sets $\Kcal_{R}$ and $\Kcal_{[r]\setminus R}$ need not
be disjoint, so at most $k$ members of $\Kcal_{[r]\setminus R}$ have already been used and
$|\Kcal_{[r]\setminus R}|-k\geq2k-k=k$ members remain available.

The three conditions of the assembly argument hold for the family so obtained: the condition
\ref{it_ass_contains} holds by the definition of $\Kcal_{R}$ and of $\Kcal_{[r]\setminus R},$ the
condition \ref{it_ass_row} holds as the $k$ components chosen for a part are distinct components of
$G_{1}$ and are therefore pairwise vertex-disjoint and the condition \ref{it_ass_col} holds as the
$2k$ chosen components are pairwise distinct, so that a component chosen for $R$ and a component
chosen for $[r]\setminus R$ are different components of $G_{1}.$
This contradicts the assembly argument.
\end{proof}

We set $\Rcal_{1}\coloneqq\{R_{1}\}$ and construct, for $z=1,2,\ldots,$ a set
$\Rcal_{z}=\{R_{1},\ldots,R_{z}\}\subseteq\Rcal$ of pairwise distinct members, a set $\Qcal_{z}$ of
components of $(G_{1},\chi'_{1}),$ and a set $S_{z}\subseteq V(G_{1}),$ subject to the size bounds
\begin{align}
\label{eq_constr_sizes}
|\Qcal_{z}|\ \leq\ z\cdot rk
\qquad\text{and}\qquad
|S_{z}|\ \leq\ z^{2}\cdot rk\cdot\sum_{i\in[z]}g_{R_{i}}(k)
\end{align}
and maintaining the following two properties:
\begin{enumerate}
\item\label{it_constr_a} for every $R\in\Rcal_{z}$ and every $K\in\Qcal_{z},$ either $(K,\chi'_{1})$
has a packing of $(H_{R},\psi_{R})$ of size $k,$ or $(K-S_{z},\chi'_{1})$ does not contain
$(H_{R},\psi_{R})$ as a colorful minor and
\item\label{it_constr_b} for every $R\in\Rcal_{z}$ and every component $K$ of $G_{1}$ with
$K\notin\Qcal_{z},$ the colorful graph $(K,\chi'_{1})$ does not contain $(H_{R},\psi_{R})$ as a
colorful minor.
\end{enumerate}
The members of $\Rcal_{z}$ are the pieces treated so far; they are merely pairwise distinct subsets of
$[r]$ and may well intersect.
The set $\Qcal_{z}$ lists every component of $G_{1}$ that contains a treated piece, the property
\ref{it_constr_b} saying that no component outside the list does and $S_{z}$ is the accumulated
deletion, the property \ref{it_constr_a} saying that every listed component has either been cleaned of
each treated piece by $S_{z}$ or contains $k$ disjoint copies of it on its own --- and in the latter
case the component is an asset rather than a threat, as these copies feed the assembly argument.

\medskip\noindent\textit{The base of the construction.}
For $z=1,$ we set $\Qcal_{1}\coloneqq\Kcal_{R_{1}},$ so that $|\Qcal_{1}|<rk$ by
\zcref{cl_constr_start} and the first bound of \eqref{eq_constr_sizes} holds.
The property \ref{it_constr_b} holds by the definition of $\Kcal_{R_{1}},$ as a component of $G_{1}$
outside $\Kcal_{R_{1}}$ does not contain $(H_{R_{1}},\psi_{R_{1}})$ as a colorful minor.
For each of the fewer than $rk$ members $K$ of $\Qcal_{1},$ we apply the Erd\H{o}s-P{\'o}sa property
of $(H_{R_{1}},\psi_{R_{1}})$ to $(K,\chi'_{1})$ and $k,$ obtaining a packing of
$(H_{R_{1}},\psi_{R_{1}})$ of size $k$ or a covering of size at most $g_{R_{1}}(k),$ and we let
$S_{1}$ be the union of the coverings obtained in this way.
Then $|S_{1}|\leq rk\cdot g_{R_{1}}(k),$ which is the second bound of \eqref{eq_constr_sizes} for
$z=1,$ and the property \ref{it_constr_a} holds: for $K\in\Qcal_{1},$ either the first alternative was
returned, or the covering returned for $K$ is a subset of $S_{1},$ so that $(K-S_{1},\chi'_{1})$ does
not contain $(H_{R_{1}},\psi_{R_{1}})$ as a colorful minor.

\medskip\noindent\textit{The step of the construction.}
Assume that $\Rcal_{z},$ $\Qcal_{z},$ and $S_{z}$ have been constructed.
If $(H,\psi)$ is not a colorful minor of $(G_{1}-S_{z},\chi'_{1}),$ we stop.
In that case $G-(S_{0}\cup S_{z})=G_{1}-S_{z}$ and therefore $S_{0}\cup S_{z}$ is a covering of
$(H,\psi)$ in $(G,\chi')$ of size at most
$$(c+1)(k-1)+z^{2}\cdot rk\cdot\sum_{i\in[z]}g_{R_{i}}(k)\ \leq\ f(k),$$
where the inequality holds as $z\leq|\Rcal|\leq2^{r},$ so that $z^{2}\leq2^{2r},$ and as
$\sum_{i\in[z]}g_{R_{i}}(k)\leq\sum_{R\in\Rcal}g_{R}(k),$ the members $R_{1},\ldots,R_{z}$ of $\Rcal$
being pairwise distinct.

So suppose that $(H,\psi)$ is a colorful minor of $(G_{1}-S_{z},\chi'_{1});$ the remainder of the step
is carried out under this supposition.
By the partition property, applied with $S\coloneqq S_{z},$ there are a partition $\Pcal$ of $[r]$
into at least two parts from $\Rcal$ and pairwise distinct components $C_{R},$ $R\in\Pcal,$ of
$G_{1}-S_{z}$ such that $(H_{R},\psi_{R})$ is a colorful minor of $(C_{R},\chi'_{1})$ for every
$R\in\Pcal.$
Distinct parts of $\Pcal$ may well sit in distinct components of $G_{1}-S_{z}$ that belong to the same
component of $G_{1},$ so we coarsen $\Pcal$ accordingly.
For every $R\in\Pcal,$ let $K_{R}$ be the component of $G_{1}$ containing $C_{R},$ let $\Lcal$ be the
image of the map $R\mapsto K_{R},$ and, for every $K\in\Lcal,$ let
$$R^{*}_{K}\ \coloneqq\ \bigcup\{R\in\Pcal\mid K_{R}=K\}.$$

\begin{claim}
\label{cl_constr_coarse}
The sets $R^{*}_{K},$ $K\in\Lcal,$ form a partition of $[r]$ into non-empty parts and, for every
$K\in\Lcal,$ the colorful graph $(K-S_{z},\chi'_{1})$ contains $(H_{R^{*}_{K}},\psi_{R^{*}_{K}})$ as a
colorful minor, $R^{*}_{K}\in\Rcal,$ and $K\in\Kcal_{R^{*}_{K}}.$
\end{claim}

\begin{proof}[Proof of the claim]
Every part of $\Pcal$ is assigned to exactly one member of $\Lcal,$ namely to $K_{R},$ so the sets
$R^{*}_{K}$ are non-empty, pairwise disjoint, and their union is $\bigcup_{R\in\Pcal}R=[r].$
Fix $K\in\Lcal.$
The sets $C_{R},$ $R\in\Pcal$ with $K_{R}=K,$ are pairwise distinct components of $G_{1}-S_{z}$ and
are therefore pairwise vertex-disjoint subgraphs of $K-S_{z},$ each containing the corresponding
$(H_{R},\psi_{R})$ as a colorful minor; hence their union contains the disjoint union of these
colorful graphs, which is $(H_{R^{*}_{K}},\psi_{R^{*}_{K}}),$ as a colorful minor  and so does
$(K-S_{z},\chi'_{1}).$
In particular, $(K,\chi'_{1})$ contains $(H_{R^{*}_{K}},\psi_{R^{*}_{K}})$ as a colorful minor.
Were $R^{*}_{K}=[r],$ then, as $(H_{[r]},\psi_{[r]})=(H,\psi),$ the component $K$ of $G_{1}$ would
contain $(H,\psi)$ as a colorful minor, which is excluded by the choice of $S_{0}.$
Hence $R^{*}_{K}$ is a non-empty proper subset of $[r],$ that is, $R^{*}_{K}\in\Rcal,$ and
$K\in\Kcal_{R^{*}_{K}}$ by the definition of $\Kcal_{R^{*}_{K}}.$
\end{proof}

We call a member $K$ of $\Lcal$ \emph{old} if $R^{*}_{K}\in\Rcal_{z},$ that is, if its piece has
already been treated and \emph{new} otherwise.

\begin{claim}
\label{cl_constr_old}
For every old $K\in\Lcal$ we have $K\in\Qcal_{z}$ and $(K,\chi'_{1})$ has a packing of
$(H_{R^{*}_{K}},\psi_{R^{*}_{K}})$ of size $k.$
\end{claim}

\begin{proof}[Proof of the claim]
Let $K\in\Lcal$ be old.
By \zcref{cl_constr_coarse}, $(K,\chi'_{1})$ contains
$(H_{R^{*}_{K}},\psi_{R^{*}_{K}})$ as a colorful minor and, as $R^{*}_{K}\in\Rcal_{z},$ the property
\ref{it_constr_b} yields $K\in\Qcal_{z}.$
By the property \ref{it_constr_a}, either $(K,\chi'_{1})$ has a packing of
$(H_{R^{*}_{K}},\psi_{R^{*}_{K}})$ of size $k,$ or $(K-S_{z},\chi'_{1})$ does not contain
$(H_{R^{*}_{K}},\psi_{R^{*}_{K}})$ as a colorful minor; the second alternative contradicts
\zcref{cl_constr_coarse}.
\end{proof}

The next claim is the engine of the step: it locates the piece to be treated next.

\begin{claim}
\label{cl_constr_next}
There is some $R_{z+1}\in\Rcal\setminus\Rcal_{z}$ with $|\Kcal_{R_{z+1}}\setminus\Qcal_{z}|<rk.$
\end{claim}

\begin{proof}[Proof of the claim]
Suppose towards a contradiction that $|\Kcal_{R}\setminus\Qcal_{z}|\geq rk$ for every
$R\in\Rcal\setminus\Rcal_{z}.$
We build a family as in the assembly argument for the partition
$\Pcal'\coloneqq\{R^{*}_{K}\mid K\in\Lcal\}$ of $[r],$ whose parts belong to $\Rcal$ by
\zcref{cl_constr_coarse}.

For an old $K\in\Lcal,$ we let $P^{1}_{R^{*}_{K}},\ldots,P^{k}_{R^{*}_{K}}$ be the $k$ members of the
packing of $(H_{R^{*}_{K}},\psi_{R^{*}_{K}})$ inside $K$ provided by \zcref{cl_constr_old}; all of them
are contained in the single component $K,$ which belongs to $\Qcal_{z}.$
For the new members of $\Lcal$ we proceed one by one.
As the sets $R^{*}_{K}$ form a partition of $[r]$ into non-empty parts, we have $|\Lcal|\leq r,$ so at
most $(r-1)k$ components have been chosen when a new $K$ is treated; since
$|\Kcal_{R^{*}_{K}}\setminus\Qcal_{z}|\geq rk$ by our supposition, the piece $R^{*}_{K}$ being a
member of $\Rcal\setminus\Rcal_{z},$ at least $rk-(r-1)k=k$ members of
$\Kcal_{R^{*}_{K}}\setminus\Qcal_{z}$ are still unused and we may let
$P^{1}_{R^{*}_{K}},\ldots,P^{k}_{R^{*}_{K}}$ be $k$ of them, distinct from each other and from all
components chosen before.

We check the three conditions of the assembly argument.
The condition \ref{it_ass_contains} holds for old parts by \zcref{cl_constr_old} and for new parts by
the definition of $\Kcal_{R^{*}_{K}}.$
The condition \ref{it_ass_row} holds for old parts as the members of a packing are pairwise
vertex-disjoint and for new parts as distinct components of $G_{1}$ are pairwise vertex-disjoint.
For the condition \ref{it_ass_col}, every subgraph of the family is contained in a component of
$G_{1}$ and it suffices to check that no component of $G_{1}$ is used for two distinct parts of
$\Pcal'.$
Two old parts use the distinct components $K$ and $K'$ of $\Lcal;$ two new parts use components chosen
to be distinct from all components chosen before; and an old part uses a component of $\Qcal_{z},$
while a new part uses components of $\Kcal_{R^{*}_{K}}\setminus\Qcal_{z},$ which lie outside
$\Qcal_{z}.$
This contradicts the assembly argument.
\end{proof}

We set $\Rcal_{z+1}\coloneqq\Rcal_{z}\cup\{R_{z+1}\}$ and
$\Qcal_{z+1}\coloneqq\Qcal_{z}\cup\Kcal_{R_{z+1}},$ so that, by \zcref{cl_constr_next},
$$|\Qcal_{z+1}|\ =\ |\Qcal_{z}|+|\Kcal_{R_{z+1}}\setminus\Qcal_{z}|\ \leq\ z\cdot rk+rk\ =\ (z+1)\cdot
rk,$$
which is the first bound of \eqref{eq_constr_sizes} for $z+1.$
Note that only the number of \emph{new} members of $\Kcal_{R_{z+1}}$ is bounded, the set
$\Kcal_{R_{z+1}}$ itself possibly being much larger.
For every $K\in\Qcal_{z+1},$ we apply the Erd\H{o}s-P{\'o}sa property of
$(H_{R_{z+1}},\psi_{R_{z+1}})$ to $(K,\chi'_{1})$ and $k,$ obtaining a packing of size $k$ or a
covering of size at most $g_{R_{z+1}}(k),$ we let $S'_{z+1}$ be the union of the coverings obtained in
this way and we set $S_{z+1}\coloneqq S_{z}\cup S'_{z+1}.$
Then $|S'_{z+1}|\leq(z+1)\cdot rk\cdot g_{R_{z+1}}(k)$ and therefore
\begin{align*}
|S_{z+1}|\ \leq\ |S_{z}|+(z+1)\cdot rk\cdot g_{R_{z+1}}(k)
\ &\leq\ z^{2}\cdot rk\cdot\sum_{i\in[z]}g_{R_{i}}(k)+(z+1)\cdot rk\cdot g_{R_{z+1}}(k)\\
&\leq\ (z+1)^{2}\cdot rk\cdot\sum_{i\in[z+1]}g_{R_{i}}(k),
\end{align*}
where the last inequality uses $z^{2}\leq(z+1)^{2}$ and $z+1\leq(z+1)^{2}.$
This is the second bound of \eqref{eq_constr_sizes} for $z+1.$

It remains to verify the two properties for $\Rcal_{z+1},$ $\Qcal_{z+1},$ and $S_{z+1}.$
For the property \ref{it_constr_b}, let $R\in\Rcal_{z+1}$ and let $K$ be a component of $G_{1}$ with
$K\notin\Qcal_{z+1}.$
If $R\in\Rcal_{z},$ then $K\notin\Qcal_{z},$ as $\Qcal_{z}\subseteq\Qcal_{z+1},$ and the property
\ref{it_constr_b} for $z$ applies; if $R=R_{z+1},$ then $K\notin\Kcal_{R_{z+1}},$ as
$\Kcal_{R_{z+1}}\subseteq\Qcal_{z+1},$ which is the assertion.
For the property \ref{it_constr_a}, let $R\in\Rcal_{z+1}$ and $K\in\Qcal_{z+1}.$
If $R\in\Rcal_{z}$ and $K\in\Qcal_{z},$ the property held for $z$ and is preserved: the first
alternative does not refer to $S_{z},$ and the second is preserved as $S_{z}\subseteq S_{z+1}.$
If $R\in\Rcal_{z}$ and $K\in\Qcal_{z+1}\setminus\Qcal_{z},$ then $(K,\chi'_{1})$ does not contain
$(H_{R},\psi_{R})$ as a colorful minor by the property \ref{it_constr_b} for $z,$ so the second
alternative holds.
If $R=R_{z+1},$ the property holds for every $K\in\Qcal_{z+1}$ by the applications of the
Erd\H{o}s-P{\'o}sa property just made, the covering returned for $K$ being a subset of $S_{z+1}.$

\medskip\noindent\textit{Termination.}
As the members $R_{1},R_{2},\ldots$ are pairwise distinct elements of $\Rcal,$ the construction
performs at most $|\Rcal|$ steps.
Moreover, the whole analysis of a step, including \zcref{cl_constr_coarse,cl_constr_old,cl_constr_next},
was carried out under the supposition that $(H,\psi)$ is a colorful minor of $(G_{1}-S_{z},\chi'_{1}).$
If the construction reaches the step $z=|\Rcal|,$ then $\Rcal_{z}=\Rcal$ and the conclusion of
\zcref{cl_constr_next} is unsatisfiable, the set $\Rcal\setminus\Rcal_{z}$ being empty; hence the
supposition fails at that step, that is, $(H,\psi)$ is not a colorful minor of
$(G_{1}-S_{z},\chi'_{1}),$ and the construction stops.
In all cases the construction therefore stops at some $z\leq|\Rcal|$ with $(H,\psi)$ not a colorful
minor of $(G_{1}-S_{z},\chi'_{1})$ and, as observed at the beginning of the step, $S_{0}\cup S_{z}$ is
then a covering of $(H,\psi)$ in $(G,\chi')$ of size at most $f(k),$ hence, as noted at the outset,
also a covering of $(H,\psi)$ in $(G,\chi).$
\end{proof}

No attempt has been made to optimize the gap.
For a connected pattern it is $(c+1)(k-1),$ where $c$ is given by \zcref{lemma_EP_reduce} and therefore by the function $\sg$ of \zcref{prop_torsoGrid}, while for a general pattern the recursion on the number of components multiplies the gaps obtained for the pieces at every step.
The bound is thus inherited from the min-max duality between $I$-torso treewidth and segregated grids and improving it would require an improvement there first.

\section{The lower bound}
\label{sec_EP_to_crucial}

In this section we prove the remaining implication of \zcref{th_EP_single}, from \ref{it_s_ep} to \ref{it_s_cru}: a colorful graph that has the Erd\H{o}s-P{\'o}sa property is crucial.
Equivalently, and this is the form in which we argue, a colorful graph that is not crucial does not have the Erd\H{o}s-P{\'o}sa property.

Refuting the Erd\H{o}s-P{\'o}sa property of a colorful graph $(H,\psi)$ means exhibiting hosts in which packing and covering are as far apart as one wishes: a bound $c,$ depending on $(H,\psi)$ alone and colorful graphs of arbitrarily large covering number in which no more than $c$ copies of $(H,\psi)$ fit disjointly.
All our hosts are built from the pattern itself, by the following two-step recipe.
First, $(H,\psi)$ is drawn in a surface in such a way that its colored vertices all lie on the boundary of a single hole.
Second, $2k$ copies of that drawing are superimposed and the crossings between different copies are turned into vertices.
The superposition contains $2k$ copies of $(H,\psi)$, no vertex of which is used by more than two of them, so its covering number grows with $k;$ the whole difficulty is to show that its packing number does not.
That is where non-cruciality enters: each of the four ways in which a colorful graph can fail to be crucial forces the copies of $(H,\psi)$ inside the superposition to compete for a resource that the surface makes scarce.
For a non-color-facial pattern the scarce resource is Euler genus, since every disjoint copy consumes a fixed amount of it.
For the remaining three failures the pattern is color-facial, the surface is a disk, and the scarce resource is the boundary: disjoint copies of the pattern occupy pairwise different portions of the boundary and a disk admits only linearly many pairwise non-crossing objects attached to its boundary.
For a pattern that is not single-component bicolored the copies must additionally interleave two linkages of different colors and the scarce resource is the width of the region in which the two linkages cross.

\medskip
One warning about the shape of the argument is in order, as it dictates the organization below.
Non-cruciality is detected through \zcref{obs_CrucialMinimal}, which supplies a member $(Z,\zeta)$ of $\Ocal$ that is a colorful minor of $(H,\psi),$ and the geometry above is really geometry of the small graph $(Z,\zeta).$
It is therefore tempting to refute the Erd\H{o}s-P{\'o}sa property of $(Z,\zeta)$ and to transfer the refutation to $(H,\psi).$
This does not work and the reason is worth recording.
Containment of $(Z,\zeta)$ in $(H,\psi)$ makes copies of $(H,\psi)$ harder to find and easier to destroy, so it moves the two parameters in the same direction:
\begin{align}
\label{eq_lb_monotone}
\pack_{H,\psi}(G,\chi)\ \leq\ \pack_{Z,\zeta}(G,\chi)
\qquad\text{and}\qquad
\cover_{H,\psi}(G,\chi)\ \leq\ \cover_{Z,\zeta}(G,\chi)
\end{align}
for every colorful graph $(G,\chi).$
Only the first of these is of use to us: it lets an upper bound on the packing number of the obstruction be inherited by the pattern.
The second points the wrong way, so a lower bound on the covering number of the obstruction says nothing about the pattern and the class of colorful graphs with the Erd\H{o}s-P{\'o}sa property cannot be shown to be closed under colorful minors by this route; that it is closed is \zcref{cor_EPminorClosed}, which we obtain only as a consequence of \zcref{th_EP_single}.
Accordingly, every covering bound below is proved for the pattern $(H,\psi)$ directly and the obstruction $(Z,\zeta)$ is used only where its shape is genuinely needed, namely in the packing bounds.

\subsection{Preliminaries for the lower bound}
\label{subsec_lb_prelim}

\paragraph{Surfaces.}
A \emph{surface} is a compact connected $2$-manifold without boundary and the \emph{Euler genus} $\eg(\Sigma)$ of a surface $\Sigma$ is $2-\chi_{\text{E}}(\Sigma),$ where $\chi_{\text{E}}$ denotes the Euler characteristic; the sphere is the unique surface of Euler genus $0.$
The \emph{Euler genus} $\eg(G)$ of a graph $G$ is the minimum Euler genus of a surface in which $G$ embeds, so that $\eg(G)=0$ if and only if $G$ is planar.
A \emph{single-boundaried surface} is a surface from which the interior of a closed disk has been removed; the boundary of the removed disk is its \emph{boundary}.
We write $\Delta$ for a closed disk, which is the single-boundaried surface obtained from the sphere and we use the term \emph{disk} for both.

The following is a classical additivity property of the Euler genus; it is stated in
\cite[Theorem 4.4.3]{MoharT01Graphs}, where it is attributed to Stahl and Beineke and to
Miller and we use it in the following iterated form.
It says that non-planar pieces glued along a single vertex accumulate genus, so that a
surface of bounded genus can host only boundedly many of them.

\begin{proposition}[\cite{MoharT01Graphs}]\label{prop_genusAdds}%
Let $Z_{1},\ldots,Z_{r}$ be connected graphs, let $z_{i}\in V(Z_{i})$ for $i\in[r],$ and let
$Z'$ be the graph obtained from the disjoint union of $Z_{1},\ldots,Z_{r}$ by identifying
$z_{1},\ldots,z_{r}$ into a single vertex.
Then $\eg(Z')=\sum_{i\in[r]}\eg(Z_{i}).$
In particular, if every $Z_{i}$ is non-planar, then $\eg(Z')\geq r.$
\end{proposition}

\paragraph{Societies.}
A \emph{society} is a pair $(G,\Omega),$ where $G$ is a graph and $\Omega$ is a cyclic ordering of a subset $V(\Omega)$ of $V(G).$
A \emph{segment} of $\Omega$ is a set of vertices of $V(\Omega)$ that are consecutive in $\Omega.$
A \emph{transaction} in $(G,\Omega)$ is a collection $\Pcal$ of pairwise vertex-disjoint paths of $G$ for which there are two disjoint segments $A$ and $B$ of $\Omega$ such that every path of $\Pcal$ has one endpoint in $A,$ its other endpoint in $B,$ and no internal vertex in $V(\Omega).$
The \emph{depth} of $(G,\Omega)$ is the maximum size of a transaction in $(G,\Omega).$ 
Intuitively, the depth measures how much traffic can be routed across the society from one side of the cyclic ordering to the other; the last case of the proof below bounds a packing number by exhibiting the members of a packing as such traffic.

\paragraph{Realizations.}
Let $(G,\chi)$ be a colorful graph and let $(H,\psi)$ be a colorful minor of $(G,\chi).$
A \emph{realization} of $(H,\psi)$ in $(G,\chi)$ is a subgraph $D$ of $G$ such that $(D,\chi)$ minimally contains $(H,\psi)$ as a colorful minor, that is, such that no proper subgraph of $D$ does.
A packing of $(H,\psi)$ in $(G,\chi)$ of size $k$ may always be assumed to consist of realizations, as every member of a packing contains one.

We record the two elementary facts that were announced in \eqref{eq_lb_monotone} and the reformulation of the failure of the Erd\H{o}s-P{\'o}sa property that we shall verify in each case.

\begin{observation}
\label{obs_packMonotone}
Let $(Z,\zeta)$ and $(H,\psi)$ be colorful graphs with $(Z,\zeta)\leq(H,\psi).$
Then, for every colorful graph $(G,\chi),$ it holds that $\pack_{H,\psi}(G,\chi)\leq\pack_{Z,\zeta}(G,\chi)$ and $\cover_{H,\psi}(G,\chi)\leq\cover_{Z,\zeta}(G,\chi).$
\end{observation}

\begin{proof}
By the transitivity of the colorful minor relation, every subgraph of $G$ that contains $(H,\psi)$ as a colorful minor also contains $(Z,\zeta).$
Hence every packing of $(H,\psi)$ in $(G,\chi)$ is a packing of $(Z,\zeta)$ of the same size, which gives the first inequality and every set $S\subseteq V(G)$ such that $(G-S,\chi)$ does not contain $(Z,\zeta)$ is a set such that $(G-S,\chi)$ does not contain $(H,\psi),$ which gives the second.
\end{proof}

\begin{observation}
\label{obs_EPfails}
A colorful graph $(H,\psi)$ does not have the Erd\H{o}s-P{\'o}sa property if and only if there is some $c\in\Nbbb$ such that, for every $n\in\Nbbb,$ there is a colorful graph $(G,\chi)$ with $\pack_{H,\psi}(G,\chi)\leq c$ and $\cover_{H,\psi}(G,\chi)\geq n.$
\end{observation}

\begin{proof}
If such a $c$ exists, then no function $f$ can satisfy $\cover_{H,\psi}\leq f\circ\pack_{H,\psi}$ on all colorful graphs, as this would force $f(c)\geq n$ for every $n.$
Conversely, if no such $c$ exists, then for every $c\in\Nbbb$ the supremum of $\cover_{H,\psi}(G,\chi)$ over all colorful graphs $(G,\chi)$ with $\pack_{H,\psi}(G,\chi)\leq c$ is finite and the function mapping $c$ to this supremum is a gap for $(H,\psi).$
\end{proof}

\paragraph{Non-crossing families in a disk.}
Three of the cases below reduce to the same planar counting, which we now isolate.
The point is that a disk can accommodate only linearly many pairwise disjoint objects, each of which reaches at least three prescribed stretches of its boundary; the bound follows from Euler's formula applied to the incidences between the objects and the stretches.

\begin{lemma}
\label{lemma_noncrossing}
Let $a\in\Nbbb_{\geq3},$ let $\Delta$ be a closed disk  and let $A_{1},\ldots,A_{n}$ be pairwise disjoint arcs of the boundary of $\Delta.$
Let $D_{1},\ldots,D_{\ell},$ where $\ell\geq1,$ be pairwise disjoint connected sets, each of which is drawn in $\Delta,$ meets the boundary of $\Delta$ only in points of $A_{1}\cup\ldots\cup A_{n},$ and meets at least $a$ of the arcs $A_{1},\ldots,A_{n}.$
Then
\begin{align}
\label{eq_noncrossing}
\ell\cdot(a-2)\ \leq\ 2n-4.
\end{align}
\end{lemma}

\begin{proof}
Let $\Delta'$ be the closed disk complementary to $\Delta$ in the sphere, so that $\Delta$ and $\Delta'$ meet exactly in their common boundary.
For every $j\in[n],$ we place a point $w_{j}$ in the interior of $\Delta'$ and, for every $i\in[\ell]$ such that $D_{i}$ meets $A_{j},$ we choose a point $z_{ij}\in D_{i}\cap A_{j}$ and we join $w_{j}$ to $z_{ij}$ by an arc whose interior lies in the interior of $\Delta'.$
For a fixed $j,$ the points $z_{ij}$ all lie on the arc $A_{j},$ so the arcs joining them to $w_{j}$ may be drawn pairwise disjoint apart from $w_{j}$ and inside an arbitrarily small neighborhood of $A_{j}$ in $\Delta'.$
As $A_{1},\ldots,A_{n}$ are pairwise disjoint, these $n$ neighborhoods may be chosen pairwise disjoint as well.
Together with the sets $D_{1},\ldots,D_{\ell},$ which are pairwise disjoint and lie in $\Delta,$ we thus obtain a family of objects drawn on the sphere no two of which cross.

Contracting each $D_{i}$ to a single point $u_{i},$ which preserves planarity as $D_{i}$ is connected, we obtain a plane graph $B$ with vertex set $\{u_{1},\ldots,u_{\ell}\}\cup\{w_{1},\ldots,w_{n}\}$ in which $u_{i}$ is adjacent to $w_{j}$ if and only if $D_{i}$ meets $A_{j}.$
The graph $B$ is simple and bipartite, it has $\ell+n$ vertices and every $u_{i}$ has degree at least $a,$ so it has at least $a\ell$ edges.
As $n\geq a\geq3,$ we have $\ell+n\geq3,$ and a simple bipartite planar graph on $m\geq3$ vertices has at most $2m-4$ edges.
Hence $a\ell\leq2(\ell+n)-4,$ which is \eqref{eq_noncrossing}.
\end{proof}

We shall only use \zcref{lemma_noncrossing} for $a=3,$ where it gives $\ell\leq2n-4.$

\subsection{The \texorpdfstring{$k$}{k}-multiplication}
\label{subsec_lb_multiplication}

We now describe the operation that produces our hosts.
The intuition is that of overlaying transparencies: one drawing of the pattern in a single-boundaried surface is copied $2k$ times, the copies are placed in general position so that any two of them meet only in finitely many points and each such meeting point is turned into a vertex.
The result contains $2k$ subdivided copies of the pattern   and the general position guarantees that no vertex serves more than two of them, so that few vertices cannot destroy all of them.

\paragraph{$k$-multiplication.}
Let $(G,\chi)$ be a colorful graph and let $\Sigma^{-}$ be a single-boundaried surface in which $G$ is drawn in such a way that all vertices carrying colors lie on the boundary of $\Sigma^{-}.$
The drawing is allowed to have crossings, in which case we require that no edge crosses itself, that no two edges cross more than once, and that no crossing point lies on the boundary.
We refer to the crossing points of this drawing as the \emph{old} crossing points, as further crossings will be created below.
The \emph{$k$-multiplication of $(G,\chi)$ in $\Sigma^{-}$} is the colorful graph $(G^{k},\chi^{k})$ obtained as follows.
We take $k$ drawings of $G$ in $\Sigma^{-},$ each of them a copy of the given one, chosen so that only finitely many points of $\Sigma^{-}$ belong to more than one of them, so that no point belongs to more than two of them, and so that any two distinct copies meet only in interior points of edges; the points belonging to two copies are the \emph{new} crossing points.
We insist, moreover, that no new crossing point lies on the boundary of $\Sigma^{-}$ and that the $k$ copies of a vertex of $G$ drawn on the boundary of $\Sigma^{-}$ are drawn consecutively on it.
Finally, we replace every new crossing point by a vertex of degree $4$ carrying no color and we let every vertex coming from a copy of $(G,\chi)$ inherit its palette.
The old crossing points are not replaced by vertices, so that each old crossing of $G$ gives rise to $k$ crossings in the drawing of $G^{k}.$
In particular, if the given drawing of $G$ has no crossings, then $G^{k}$ is drawn in $\Sigma^{-}$ without crossings.

\begin{claim}
\label{cl_bzsoc_cover}
Let $(G,\chi)$ be drawn in $\Sigma^{-}$ as above, let $(F,\varphi)$ be a colorful minor of $(G,\chi),$ let $k\in\Nbbb_{\geq1},$ and let $(G^{2k},\chi^{2k})$ be the $2k$-multiplication of $(G,\chi)$ in $\Sigma^{-}.$
Then $\cover_{F,\varphi}(G^{2k},\chi^{2k})\geq k.$
\end{claim}

\begin{claimproof}
The graph $G^{2k}$ contains subgraphs $G_{1},\ldots,G_{2k},$ where $G_{i}$ consists of the $i$-th copy of $G$ together with the vertices that replaced the new crossing points lying on it, so that $G_{i}$ is a subdivision of a copy of $G.$
A vertex of $G^{2k}$ coming from a copy of $G$ lies in exactly one of $G_{1},\ldots,G_{2k},$ and a vertex replacing a new crossing point lies in exactly two of them, so no vertex of $G^{2k}$ lies in more than two of $G_{1},\ldots,G_{2k}.$
The vertices that replaced new crossing points carry no color, so contracting them back into a neighbor shows that $(G_{i},\chi^{2k})$ contains $(G,\chi),$ and hence also $(F,\varphi),$ as a colorful minor; we fix a realization $D_{i}$ of $(F,\varphi)$ in $(G_{i},\chi^{2k})$ for every $i\in[2k].$
Now let $S\subseteq V(G^{2k})$ with $|S|\leq k-1.$
Every vertex of $S$ belongs to at most two of $G_{1},\ldots,G_{2k}$ and therefore meets at most two of $D_{1},\ldots,D_{2k},$ so $S$ meets at most $2(k-1)<2k$ of them and some $D_{i}$ is disjoint from $S.$
Hence $(G^{2k}-S,\chi^{2k})$ contains $(F,\varphi)$ as a colorful minor and $S$ is not a covering, which proves that $\cover_{F,\varphi}(G^{2k},\chi^{2k})\geq k.$
\end{claimproof}

\subsection{Proof of the lower bound}
\label{subsec_lb_proof}

\begin{theorem}\label{thm_EPiffCrucial}%
Every colorful graph that has the Erd\H{o}s-P{\'o}sa property is crucial.
\end{theorem}

\begin{proof}
Let $(H,\psi)$ be a colorful graph that is not crucial; we prove that $(H,\psi)$ does not have the Erd\H{o}s-P{\'o}sa property.
We set $Q\coloneqq\psi(H),$ $q\coloneqq|Q|,$ and $h\coloneqq|V(H)|.$
By \zcref{obs_CrucialMinimal}, some member of $\Ocal$ is a colorful minor of $(H,\psi).$
We first dispose of the patterns containing a member of $\Ocal^{0},$ then of those containing a member of $\Ocal^{1}.$
For the remaining ones the member supplied by \zcref{obs_CrucialMinimal} may be taken in $\tilde{\Ocal}^{2}\cup\Ocal^{3}\cup\Ocal^{4},$ and the last four cases are keyed on which of these three sets contains it.
Each case may assume that the previous ones do not apply.
In every case we produce a constant $c,$ depending on $(H,\psi)$ but not on $k,$ and colorful graphs whose packing number is at most $c$ and whose covering number is at least $k$ for every $k\in\Nbbb_{\geq1};$ by \zcref{obs_EPfails}, this is what has to be shown.

\medskip\noindent\textbf{Case 1: $(H,\psi)$ contains a member of $\Ocal^{0}$.} 
Then $H$ contains $K_{5}$ or $K_{3,3}$ as a minor and is therefore not planar.
By the result of Robertson and Seymour \cite{RobertsonS1986Grapha}, the graph $H$ does not have the Erd\H{o}s-P{\'o}sa property with respect to the minor relation, so, arguing as in \zcref{obs_EPfails} for graphs without colors, there is some $c\in\Nbbb$ such that, for every $k\in\Nbbb,$ there is a graph $G_{k}$ with $\pack_{H}(G_{k})\leq c$ and $\cover_{H}(G_{k})\geq k.$
Let $(G_{k},\chi_{k})$ be the $Q$-rainbow colorful graph on $G_{k},$ that is, $\chi_{k}(v)\coloneqq Q$ for every $v\in V(G_{k}).$
For every subgraph $D$ of $G_{k},$ the colorful graph $(D,\chi_{k})$ contains $(H,\psi)$ as a colorful minor if and only if $D$ contains $H$ as a minor: the forward direction is immediate and for the converse it suffices to observe that, in a model of $H$ in $D,$ every branch set is non-empty and hence carries the palette $Q,$ from which the colors not in $\psi(v)$ may be removed.
Therefore $\pack_{H,\psi}(G_{k},\chi_{k})=\pack_{H}(G_{k})\leq c$ and $\cover_{H,\psi}(G_{k},\chi_{k})=\cover_{H}(G_{k})\geq k,$ as required.

\medskip\noindent
For the remaining cases we may assume that $(H,\psi)$ excludes every member of $\Ocal^{0}$ as a colorful minor.

\medskip\noindent\textbf{Case 2: $(H,\psi)$ contains a member of $\Ocal^{1}$.}
The scarce resource here is Euler genus.
We draw $(H,\psi)$ in a surface of minimum Euler genus, remove a disk around the vertex that witnesses the non-planarity of the apexed graph and multiply; a packing of $(H,\psi)$ in the result would produce many non-planar subgraphs sharing a single vertex, which \zcref{prop_genusAdds} forbids.

Let $H^{+}$ be the graph obtained from $H$ by adding a new vertex $v_{\ast}$ adjacent to every vertex $u$ of $H$ with $\psi(u)\neq\emptyset.$
As $(H,\psi)$ contains a member of $\Ocal^{1},$ \zcref{lemma_colorFacial} implies that $(H,\psi)$ is not color-facial and therefore $H^{+}$ is not planar: an embedding of $H^{+}$ in the sphere would, after the removal of $v_{\ast},$ place all colored vertices of $(H,\psi)$ on the boundary of a single face.
Let $\Sigma$ be a surface of minimum Euler genus in which $H^{+}$ embeds and let $g\coloneqq\eg(\Sigma),$ so that $g>0.$
Fixing an embedding of $H^{+}$ in $\Sigma$ and deleting $v_{\ast}$ from it, 
we obtain an embedding of $H$ in $\Sigma$ together with a closed disk $\Delta\subseteq\Sigma$ whose interior is disjoint from the drawing and such that all colored vertices of $(H,\psi)$ lie on the boundary of $\Delta.$
We let $\Sigma^{-}\coloneqq\Sigma\setminus\operatorname{int}(\Delta)$ and, for every $k\in\Nbbb_{\geq1},$ we let $(H^{2k},\psi^{2k})$ be the $2k$-multiplication of $(H,\psi)$ in $\Sigma^{-}.$

By \zcref{cl_bzsoc_cover}, applied with $(F,\varphi)\coloneqq(H,\psi),$ we obtain
\begin{align}
\label{eq_lb_cover_genus}
\cover_{H,\psi}(H^{2k},\psi^{2k})\ \geq\ k\qquad\text{for every }k\in\Nbbb_{\geq1}.
\end{align}

\begin{claim}
\label{cl_lb_genus}
For every $k\in\Nbbb_{\geq1},$ it holds that $\pack_{H,\psi}(H^{2k},\psi^{2k})\leq g.$
\end{claim}

\begin{claimproof}
Set $r\coloneqq g+1$ and assume towards a contradiction that $H^{2k}$ contains pairwise vertex-disjoint realizations $D_{1},\ldots,D_{r}$ of $(H,\psi).$
Let $H^{2k+}$ be the graph obtained from $H^{2k}$ by adding a new vertex $v_{\ast}$ adjacent to all vertices of $H^{2k}$ drawn on the boundary of $\Sigma^{-}.$
As $\Delta$ is a disk disjoint from the drawing of $H^{2k}$ in $\Sigma^{-},$ the vertex $v_{\ast}$ may be drawn in the interior of $\Delta$ and joined to those vertices inside $\Delta,$ so $H^{2k+}$ embeds in $\Sigma$ and $\eg(H^{2k+})\leq g.$

For every $i\in[r],$ let $D_{i}'$ be the union of components of $D_{i}$ that contain at least one colored vertex and let $D_{i}'^{+}$ be obtained from $D_{i}'$ by adding $v_{\ast}$ joined to every vertex of $D_{i}'$ with a non-empty palette under $\psi^{2k}.$
Since $(H,\psi)$ contains a member of $\Ocal^{1}$ (which is connected and colored), its model lies entirely within the colored components of $D_{i},$ so $D_{i}'^{+}$ contains that colored minor.
Moreover $D_{i}'^{+}$ is not planar: were it planar, the removal of $v_{\ast}$ from a plane embedding would show that the colored part of $(D_{i},\psi^{2k})$ is color-facial, hence that $(H,\psi)$ is color-facial by \zcref{obs_colorFacialClosed}, which we have excluded.

The graphs $D_{1}',\ldots,D_{r}'$ are pairwise vertex-disjoint, so $v_{\ast}$ is the only vertex shared by any two of $D_{1}'^{+},\ldots,D_{r}'^{+},$ and their union is obtained from their disjoint union by identifying the $r$ copies of $v_{\ast}.$
Each $D_{i}'^{+}$ is connected (all colored components connect through $v_{\ast}$).
By \zcref{prop_genusAdds}, the union has Euler genus larger than $g,$ and being a subgraph of $H^{2k+}$ it yields $\eg(H^{2k+})>g,$ a contradiction.
\end{claimproof}

Together, \eqref{eq_lb_cover_genus} and \zcref{cl_lb_genus} give what is required, with $c\coloneqq g.$

\medskip\noindent
For the remaining cases we may assume that $(H,\psi)$ excludes every member of $\Ocal^{0}\cup\Ocal^{1}$ as a colorful minor, so that, by \zcref{lemma_colorFacial}, the colorful graph $(H,\psi)$ is color-facial.
As the members of $\tilde{\Ocal}^{1}$ belong to $\Ocal^{1},$ the member of $\Ocal$ supplied by \zcref{obs_CrucialMinimal} now belongs to $\tilde{\Ocal}^{2}\cup\Ocal^{3}\cup\Ocal^{4}.$

\paragraph{A normal form for color-facial patterns.}
Being color-facial, $H$ admits an embedding in a closed disk $\Delta$ in which all vertices with a non-empty palette are drawn on the boundary of $\Delta;$ pushing the remaining vertices into the interior, we may assume that the vertices drawn on the boundary of $\Delta$ are exactly the colored ones and, perturbing the drawing, that every edge meets the boundary of $\Delta$ only in its endpoints.
It is convenient to arrange in addition that every colored vertex carries exactly one color, so that the boundary of $\Delta$ is divided into monochromatic stretches and it becomes meaningful to ask which stretch a given vertex belongs to.
This is achieved by uncontracting.
For every vertex $v$ of $H$ with $\psi(v)=\{i_{1},\ldots,i_{s}\},$ where $s\geq1$ and $i_{1}<\ldots<i_{s},$ we consider a path $P_{v}$ on $s$ vertices, we identify one of its endpoints with $v,$ we draw $P_{v}$ along the boundary of $\Delta,$ and we assign the colors $i_{1},\ldots,i_{s}$ to the vertices of $P_{v}$ in this order, starting from $v.$
The resulting colorful graph $(H^{\circ},\psi^{\circ})$ contains $(H,\psi)$ as a colorful minor, obtained by contracting every $P_{v},$ it is drawn in $\Delta$ without crossings, all its colored vertices lie on the boundary of $\Delta,$ and each of them carries exactly one color.
Note that $H^{\circ}$ has at most $qh$ colored vertices.

For every $k\in\Nbbb_{\geq1},$ we let $(H^{\circ2k},\psi^{\circ2k})$ be the $2k$-multiplication of $(H^{\circ},\psi^{\circ})$ in $\Delta.$
As the drawing of $H^{\circ}$ in $\Delta$ has no crossings, $H^{\circ2k}$ is drawn in $\Delta$ without crossings, that is, it is a plane graph in $\Delta,$ and its vertices drawn on the boundary of $\Delta$ are exactly its colored ones.
By the construction of the multiplication, the $2k$ copies of a colored vertex of $H^{\circ}$ are drawn consecutively on the boundary of $\Delta$ and carry the same single color; we call the set of these $2k$ vertices a \emph{color interval} and, for each of them, we fix a closed arc of the boundary of $\Delta$ that contains exactly its vertices and no other vertex of $H^{\circ2k}.$
As distinct color intervals are separated on the boundary by stretches carrying no vertex, these arcs may be chosen pairwise disjoint and, as $H^{\circ}$ has at most $qh$ colored vertices,
\begin{align}
\label{eq_lb_intervalcount}
\text{the number of color intervals of }(H^{\circ2k},\psi^{\circ2k})\text{ is at most }qh.
\end{align}

Since $(H,\psi)$ is a colorful minor of $(H^{\circ},\psi^{\circ}),$ \zcref{cl_bzsoc_cover}, applied in $\Delta$ with $(H^{\circ},\psi^{\circ})$ in the role of the multiplied colorful graph and with $(F,\varphi)\coloneqq(H,\psi),$ yields
\begin{align}
\label{eq_EP_cover}
\cover_{H,\psi}(H^{\circ2k},\psi^{\circ2k})\ \geq\ k\qquad\text{for every }k\in\Nbbb_{\geq1}.
\end{align}
It therefore remains, in each of the three cases below, to bound the packing number of the colorful graph $(H^{\circ2k},\psi^{\circ2k})$ by a quantity that does not depend on $k.$
Let $(Z,\zeta)$ be the member of $\tilde{\Ocal}^{2}\cup\Ocal^{3}$ under consideration and set
\begin{align}
\label{eq_lb_pack}
\ell\ \coloneqq\ \pack_{Z,\zeta}(H^{\circ2k},\psi^{\circ2k}),
\end{align}
so that $\pack_{H,\psi}(H^{\circ2k},\psi^{\circ2k})\leq\ell$ by \zcref{obs_packMonotone}.
We may assume that $\ell\geq1,$ as otherwise the bound below is trivial and we let $D_{1},\ldots,D_{\ell}$ be pairwise vertex-disjoint realizations of $(Z,\zeta)$ in $(H^{\circ2k},\psi^{\circ2k}).$
Being vertex-disjoint subgraphs of a plane graph whose edges meet the boundary of $\Delta$ only in their endpoints, the $D_{i}$ are pairwise disjoint as drawings and each of them meets the boundary of $\Delta$ only in colored vertices, hence only in points of the arcs fixed above.
The three cases differ only in the following step, which the counting then turns into a bound on $\ell.$

\begin{claim}
\label{cl_lb_count}
If every $D_{i}$ meets at least three color intervals, then $\ell\leq2qh.$
\end{claim}

\begin{claimproof}
By \zcref{lemma_noncrossing}, applied in $\Delta$ with $a\coloneqq3$ and with the arcs fixed for the color intervals, we get $\ell\leq2n-4,$ where $n$ is the number of color intervals; by \eqref{eq_lb_intervalcount} we have $n\leq qh$ and hence $\ell\leq2qh-4\leq2qh.$
\end{claimproof}

\medskip\noindent\textbf{Case 3: $(Z,\zeta)=(C_{4},\sigma_{1})\in\tilde{\Ocal}^{2}$.}
Here $Z=C_{4},$ every vertex carries exactly one color  and, as $(Z,\zeta)\in\tilde{\Ocal}^{2},$ exactly two colors $c$ and $c'$ occur, the two vertices of each diagonal pair carrying the same one.

\begin{claim}
\label{cl_lb_C4}
Every $D_{i}$ meets at least three color intervals.
\end{claim}

\begin{claimproof}
Suppose that some $D\coloneqq D_{i}$ meets at most two of them and let $j,$ $j',$ and $D^{\ast}$ be as in the auxiliary construction.
Fix a model of $(Z,\zeta)$ in $(D,\psi^{\circ2k})$ and let $B_{1},B_{2},B_{3},B_{4}$ be its branch sets, indexed following the cycle $C_{4},$ so that $B_{1}$ and $B_{3}$ carry the color $c$ and $B_{2}$ and $B_{4}$ carry the color $c'.$
They are pairwise disjoint and connected, $B_{m}$ is adjacent to $B_{m+1}$ for every $m,$ indices modulo $4,$ and, by the auxiliary construction, $B_{1}$ and $B_{3}$ are adjacent to $j$ while $B_{2}$ and $B_{4}$ are adjacent to $j'.$
Contracting each of $B_{1},B_{2},B_{3},B_{4}$ to a single vertex $b_{1},b_{2},b_{3},b_{4}$ and contracting the edge $\omega j'$ into a single vertex $\beta,$ we obtain a minor of $D^{\ast}$ in which $b_{1}$ and $b_{3}$ are adjacent to each of $j,$ $b_{2},$ and $b_{4},$ and in which $\beta$ is adjacent to $j$ through $\omega$ and to each of $b_{2}$ and $b_{4}$ through $j'.$
Hence $\{b_{1},b_{3},\beta\}$ and $\{j,b_{2},b_{4}\}$ are the two sides of a $K_{3,3}$-minor of $D^{\ast},$ contradicting \eqref{eq_lb_aux_planar}; see \zcref{fig_C4Interval}.
\end{claimproof}

\begin{figure}[!ht]
\begin{center}
\scalebox{.67}{%
\begin{tikzpicture}[x=1pt,y=1pt,
    bnd/.style   ={draw=black!45,line width=.6pt,dash pattern=on 2pt off 2pt},
    arcM/.style  ={draw=clM,line width=2.8pt,line cap=round},
    arcB/.style  ={draw=clB,line width=2.8pt,line cap=round},
    edge/.style  ={draw=black,line width=1.9227pt,line join=round},
    aux/.style   ={draw=black!75,line width=1.9227pt,line join=round},
    idle/.style  ={draw=black!40,line width=1.9227pt,line join=round,dash pattern=on 2.4pt off 2.4pt},
    con/.style   ={draw=clG!85!black,line width=3.8454pt,line join=round},
    blob/.style  ={draw=black!70,line width=.7pt,fill=black!8},
    mdot/.style  ={circle,inner sep=0pt,minimum size=8pt,fill=clM,draw=black!85,line width=.55pt},
    bdot/.style  ={circle,inner sep=0pt,minimum size=8pt,fill=clB,draw=black!85,line width=.55pt},
    kdot/.style  ={circle,inner sep=0pt,minimum size=8pt,fill=black},
    gdot/.style  ={circle,inner sep=0pt,minimum size=8pt,fill=black!45},
    ldot/.style  ={circle,inner sep=0pt,minimum size=10pt,fill=black!10,draw=black!70,line width=.75pt},
    lab/.style   ={font=\fontsize{14.4}{17}\selectfont},
    slab/.style  ={font=\fontsize{12.8}{15}\selectfont,text=black!60}]

  \draw[bnd] (90,240) circle (64pt);
  \node[slab] at (42,266) {$\Delta$};
  \draw[arcM] (90,240) ++(55:64pt) arc (55:125:64pt);
  \draw[arcB] (90,240) ++(235:64pt) arc (235:305:64pt);
  \node[lab,text=clM] at (32,300) {$J$};
  \node[lab,text=clB] at (30,180) {$J'$};

  \node[blob,ellipse,minimum width=30pt,minimum height=20pt] (A1) at (50,240) {};
  \node[blob,ellipse,minimum width=30pt,minimum height=20pt] (A2) at (90,240) {};
  \node[blob,ellipse,minimum width=30pt,minimum height=20pt] (A3) at (130,240) {};

  \draw[edge] (A1) -- (A2);
  \draw[edge] (A2) -- (A3);
  \draw[edge] (50,230) .. controls (58,212) and (122,212) .. (130,230);
  \draw[edge] (A1) -- (66,299.3);
  \draw[edge] (A3) -- (114,299.3);
  \draw[edge] (A2) -- (90,304);
  \draw[edge] (A1) -- (66,180.7);
  \draw[edge] (A3) -- (114,180.7);
  \fill[white] (90,216.5) circle (4.6pt);
  \draw[edge] (90,230) -- (90,176);

  \node[mdot] at (66,299.3) {}; \node[mdot] at (90,304) {}; \node[mdot] at (114,299.3) {};
  \node[bdot] at (66,180.7) {}; \node[bdot] at (90,176) {}; \node[bdot] at (114,180.7) {};
  \node[lab] at (50,240) {$B_{1}$}; \node[lab] at (90,240) {$B_{2}$}; \node[lab] at (130,240) {$B_{3}$};

  \coordinate (j1)  at (90,332);
  \coordinate (jp1) at (90,148);
  \coordinate (om1) at (172,240);
  \draw[aux] (66,299.3) .. controls (63,320) and (76,332) .. (j1);
  \draw[aux] (90,304) -- (j1);
  \draw[aux] (114,299.3) .. controls (117,320) and (104,332) .. (j1);
  \draw[aux] (66,180.7) .. controls (63,160) and (76,148) .. (jp1);
  \draw[aux] (90,176) -- (jp1);
  \draw[aux] (114,180.7) .. controls (117,160) and (104,148) .. (jp1);
  \draw[con] (j1)  .. controls (140,332) and (172,300) .. (om1);
  \draw[aux] (jp1) .. controls (140,148) and (172,180) .. (om1);
  \node[kdot] at (j1) {}; \node[kdot] at (jp1) {}; \node[kdot] at (om1) {};
  \node[lab] at (74,339) {$j$};
  \node[lab] at (72,141) {$j'$};
  \node[lab] at (185,249) {$\omega$};
  \node[slab] at (90,120) {(i)};

  \coordinate (q1) at (277.2,260.4);
  \coordinate (q2) at (301.2,186.6);
  \coordinate (q3) at (378.8,186.6);
  \coordinate (q4) at (402.8,260.4);
  \coordinate (qb) at (340,306);
  \foreach \a/\b in {q1/q2,q1/q3,q1/q4,q1/qb,q2/q3,q2/q4,q2/qb,q3/q4,q3/qb,q4/qb}
    {\draw[edge] (\a) -- (\b);}
  \node[blob,rounded corners=7pt,minimum width=58pt,minimum height=26pt] at (340,306) {};
  \draw[con] (324,306) -- (356,306);
  \node[kdot] at (324,306) {}; \node[kdot] at (356,306) {};
  \node[slab] at (324,329) {$\omega$}; \node[slab] at (357,329) {$j$};
  \node[ldot] at (q1) {}; \node[ldot] at (q2) {}; \node[ldot] at (q3) {}; \node[ldot] at (q4) {};
  \node[lab] at (257,268) {$b_{1}$};
  \node[lab] at (292,172) {$b_{2}$};
  \node[lab] at (388,172) {$b_{3}$};
  \node[lab] at (420,268) {$j'$};
  \node[lab] at (296,306) {$\beta$};
  \node[slab] at (340,120) {(ii)};

  \draw[bnd] (90,10) circle (64pt);
  \node[slab] at (42,-30) {$\Delta$};
  \draw[arcM] (90,10) ++(55:64pt) arc (55:125:64pt);
  \draw[arcB] (90,10) ++(235:64pt) arc (235:305:64pt);
  \node[lab,text=clM] at (32,70) {$J$};
  \node[lab,text=clB] at (30,-50) {$J'$};

  \node[blob,ellipse,minimum width=30pt,minimum height=20pt] (C1) at (50,10) {};
  \node[blob,ellipse,minimum width=30pt,minimum height=20pt] (C2) at (90,10) {};
  \node[blob,ellipse,minimum width=30pt,minimum height=20pt] (C3) at (130,10) {};
  \node[blob,ellipse,minimum width=34pt,minimum height=20pt] (C0) at (78,46) {};

  \draw[edge] (63,41) -- (56,20);
  \draw[edge] (86,37) -- (90,20);
  \draw[edge] (50,20) -- (62,67.5);
  \draw[edge] (130,20) -- (122,65.4);
  \draw[edge] (98,18) -- (104,72.5);
  \fill[white] (101,43.5) circle (4.6pt);
  \draw[edge] (95,46) .. controls (106,44) and (114,32) .. (124,20);
  \draw[edge] (C1) -- (66,-49.3);
  \draw[edge] (C2) -- (90,-54);
  \draw[edge] (C3) -- (114,-49.3);

  \node[mdot] at (62,67.5) {}; \node[mdot] at (104,72.5) {}; \node[mdot] at (122,65.4) {};
  \node[bdot] at (66,-49.3) {}; \node[bdot] at (90,-54) {}; \node[bdot] at (114,-49.3) {};
  \node[lab] at (50,10) {$B_{1}$}; \node[lab] at (90,10) {$B_{2}$}; \node[lab] at (130,10) {$B_{3}$};
  \node[lab] at (78,46) {$B_{0}$};

  \coordinate (j2)  at (90,102);
  \coordinate (jp2) at (90,-82);
  \coordinate (om2) at (172,10);
  \draw[aux] (62,67.5) .. controls (58,90) and (76,102) .. (j2);
  \draw[aux] (104,72.5) .. controls (106,90) and (98,102) .. (j2);
  \draw[aux] (122,65.4) .. controls (126,90) and (104,102) .. (j2);
  \draw[aux] (66,-49.3) .. controls (63,-70) and (76,-82) .. (jp2);
  \draw[aux] (90,-54) -- (jp2);
  \draw[aux] (114,-49.3) .. controls (117,-70) and (104,-82) .. (jp2);
  \draw[idle] (j2)  .. controls (140,102) and (172,70) .. (om2);
  \draw[idle] (jp2) .. controls (140,-82) and (172,-50) .. (om2);
  \node[kdot] at (j2) {}; \node[kdot] at (jp2) {}; \node[gdot] at (om2) {};
  \node[lab] at (74,109) {$j$};
  \node[lab] at (72,-89) {$j'$};
  \node[lab,text=black!45] at (185,19) {$\omega$};
  \node[slab] at (90,-110) {(iii)};

  \coordinate (r1) at (295,72);
  \coordinate (r2) at (295,10);
  \coordinate (r3) at (295,-52);
  \coordinate (s0) at (400,72);
  \coordinate (sj) at (400,10);
  \coordinate (sp) at (400,-52);
  \foreach \l in {r1,r2,r3}{\foreach \r in {s0,sj,sp}{\draw[edge] (\l) -- (\r);}}
  \node[ldot] at (r1) {}; \node[ldot] at (r2) {}; \node[ldot] at (r3) {};
  \node[ldot] at (s0) {}; \node[ldot] at (sj) {}; \node[ldot] at (sp) {};
  \node[lab] at (274,72) {$b_{1}$};
  \node[lab] at (274,10) {$b_{2}$};
  \node[lab] at (274,-52) {$b_{3}$};
  \node[lab] at (421,72) {$b_{0}$};
  \node[lab] at (419,10) {$j$};
  \node[lab] at (419,-52) {$j'$};
  \node[slab] at (347,-110) {(iv)};
\end{tikzpicture}
}
\end{center}
\caption{The situation excluded by \zcref{cl_lb_K3K13}, for $(Z,\zeta)=(K_{3},\sigma_{2})$ in (i) and (ii) and for $(Z,\zeta)=(K_{1,3},\sigma_{3})$ in (iii) and (iv).
In (i) and (iii), a realization $D$ that meets only two color intervals, whose arcs $J$ and $J'$ carry the colors $c$ and $c'$ respectively, together with the vertices $j,$ $j',$ and $\omega$ added in the complementary disk.
The branch sets $B_{1},$ $B_{2},$ and $B_{3}$ carry both colors and therefore meet both arcs, so each of them is adjacent to $j$ and to $j';$ the two cases differ only in how these three sets are tied to one another, directly in (i) and through the branch set $F_{2}$ of the center of $K_{1,3}$ in (iii).
The drawings are schematic: one crossing is shown in each and it is precisely the impossibility of drawing these configurations in $\Delta$ without crossings that the claim establishes.
In (ii), the resulting $K_{5}$-minor on $\{b_{1},b_{2},b_{3},\beta,j'\},$ where $b_{m}$ is the contraction of $B_{m}$ and $\beta$ that of the green edge $\omega j;$ the vertex $\beta$ supplies the adjacency between $j$ and $j',$ the only one of the ten that is not already present.
In (iv), the resulting $K_{3,3}$-minor with sides $\{b_{1},b_{2},b_{3}\}$ and $\{b_{0},j,j'\};$ here $\omega,$ drawn in gray, is not used.}
\label{fig_K3K13Interval}
\end{figure}
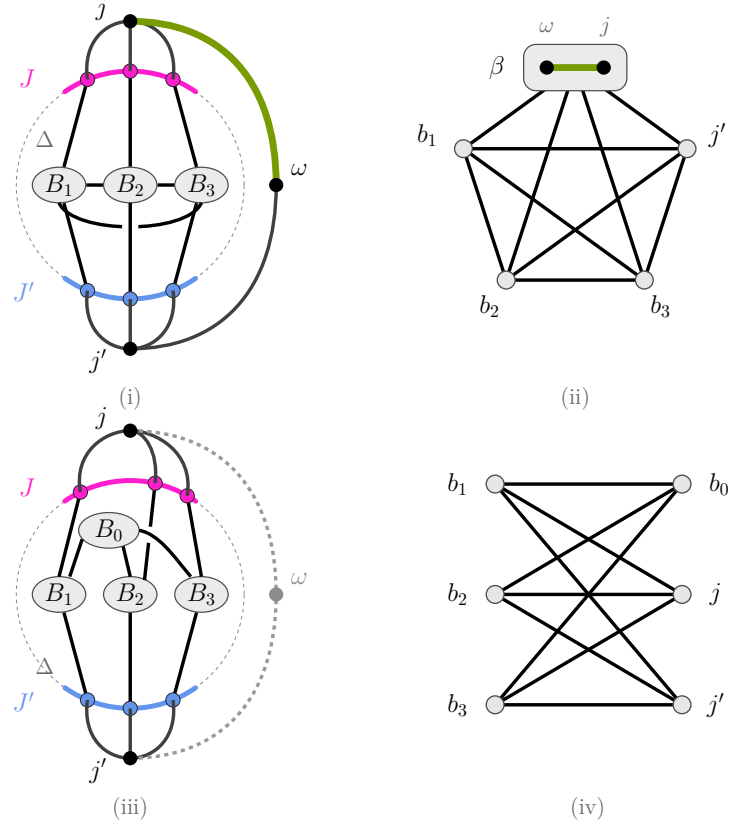

\medskip\noindent\textbf{Case 4: $(Z,\zeta)\in\{(K_{3},\sigma_{2}),(K_{1,3},\sigma_{3})\}\subseteq\tilde{\Ocal}^{2}$.}
By the definitions of $\Ocal^{2}$ and of $\tilde{\Ocal}^{2},$ exactly two colors $c$ and $c'$ occur in $(Z,\zeta)$ and three vertices of $Z$ carry both of them: all three vertices of $K_{3}$ in the first case and the three leaves of $K_{1,3}$ in the second, the center of $K_{1,3}$ carrying the empty palette.
The two cases therefore share the same core, namely three pairwise disjoint connected sets each of which reaches both of the two color intervals and they differ only in how these three sets are tied to one another: directly in the first case and through a fourth set in the second.
Either way of tying them is one link too many for a disk  and this is what the next claim records; see \zcref{fig_K3K13Interval}.

\begin{claim}
\label{cl_lb_K3K13}
Every $D_{i}$ meets at least three color intervals.
\end{claim}

\begin{claimproof}
Suppose that some $D\coloneqq D_{i}$ meets at most two of them and let $j,$ $j',$ $\omega,$ and $D^{\ast}$ be as in the auxiliary construction.
Fix a model of $(Z,\zeta)$ in $(D,\psi^{\circ2k})$ and let $B_{1},B_{2},B_{3}$ be the branch sets of the three vertices of $Z$ that carry both $c$ and $c'.$
They are pairwise disjoint and connected and each of them carries both colors, so, by the auxiliary construction, each of them is adjacent in $D^{\ast}$ both to $j$ and to $j'.$
Contracting every $B_{m}$ to a single vertex $b_{m},$ we therefore obtain, in both cases alike, a minor of $D^{\ast}$ in which each of $b_{1},$ $b_{2},$ and $b_{3}$ is adjacent to $j$ and to $j'.$
The two cases differ only in what is available beyond this.

If $(Z,\zeta)=(K_{3},\sigma_{2}),$ then $B_{1},B_{2},B_{3}$ are pairwise adjacent, so $b_{1},b_{2},b_{3}$ are pairwise adjacent as well and among the ten pairs of vertices of $\{b_{1},b_{2},b_{3},j,j'\}$ only the pair $jj'$ is still missing.
It is supplied by $\omega,$ which is adjacent to both: contracting the edge $\omega j$ into a single vertex $\beta,$ we obtain a vertex adjacent to $b_{1},b_{2},b_{3}$ through $j$ and to $j'$ through $\omega.$
Hence $\{b_{1},b_{2},b_{3},\beta,j'\}$ carries a $K_{5}$-minor of $D^{\ast},$ contradicting \eqref{eq_lb_aux_planar}.

If $(Z,\zeta)=(K_{1,3},\sigma_{3}),$ then the branch set $B_{0}$ of the center of $K_{1,3}$ is non-empty, disjoint from $B_{1},B_{2},B_{3},$ and adjacent to each of them, so contracting it to a single vertex $b_{0}$ provides a third vertex adjacent to all of $b_{1},b_{2},b_{3}.$
Hence $\{b_{1},b_{2},b_{3}\}$ and $\{b_{0},j,j'\}$ are the two sides of a $K_{3,3}$-minor of $D^{\ast},$ again contradicting \eqref{eq_lb_aux_planar}.
Note that $\omega$ plays no role here; it is needed only in the previous case, to join $j$ to $j'.$
\end{claimproof}

\medskip\noindent
By \zcref{cl_lb_O3,cl_lb_C4,cl_lb_K3K13} and \zcref{cl_lb_count}, in each of the Cases 3, 4, and 5 we have $\pack_{H,\psi}(H^{\circ2k},\psi^{\circ2k})\leq\ell\leq2qh$ for every $k\in\Nbbb_{\geq1},$ which together with \eqref{eq_EP_cover} gives what is required, with $c\coloneqq2qh.$

\medskip\noindent
\medskip\noindent\textbf{Case 5: $(Z,\zeta)\in\Ocal^{3}$.}
Here $Z=K_{1}$ and $\zeta$ assigns three colors to its unique vertex.

\begin{claim}
\label{cl_lb_O3}
Every $D_{i}$ meets at least three color intervals.
\end{claim}

\begin{claimproof}
A realization of $(Z,\zeta)$ is a minimal connected subgraph carrying the three prescribed colors, so it contains three vertices carrying three distinct colors.
These are colored vertices, hence they lie on the boundary of $\Delta,$ and, as every color interval is monochromatic, they lie in three distinct color intervals.
\end{claimproof}

\medskip\noindent
The two remaining cases share the following auxiliary construction, which converts the failure of the claim into a forbidden minor.
Let $D$ be one of $D_{1},\ldots,D_{\ell}$ and suppose that $D$ meets at most two color intervals.
As $(Z,\zeta)$ carries exactly two colors, say $c$ and $c',$ and $D$ contains vertices carrying each of them, it meets exactly two color intervals, one of color $c$ and one of color $c';$ we let $J$ and $J'$ be the two disjoint arcs fixed for them.
Let $\Delta'$ be the closed disk complementary to $\Delta$ in the sphere and let $D^{\ast}$ be obtained from $D$ by adding three new vertices $j,$ $j',$ and $\omega,$ where $j$ is joined to every vertex of $D$ lying in $J,$ where $j'$ is joined to every vertex of $D$ lying in $J',$ and where $\omega$ is joined to $j$ and to $j'.$
All three new vertices and all new edges are drawn in the interior of $\Delta',$ which is possible without crossings: the neighbors of $j$ lie on the arc $J$ and those of $j'$ on the disjoint arc $J',$ so the two stars occupy disjoint neighborhoods of $J$ and of $J'$ in $\Delta',$ and $\omega$ may be placed in the region of $\Delta'$ that remains between them.
Hence
\begin{align}
\label{eq_lb_aux_planar}
D^{\ast}\text{ is planar.}
\end{align}
Moreover, every branch set of a model of $(Z,\zeta)$ in $(D,\psi^{\circ2k})$ that carries the color $c$ contains a vertex colored $c,$ which lies on the boundary of $\Delta$ and therefore in $J,$ so that branch set is adjacent to $j$ in $D^{\ast};$ and symmetrically for $c'$ and $j'.$

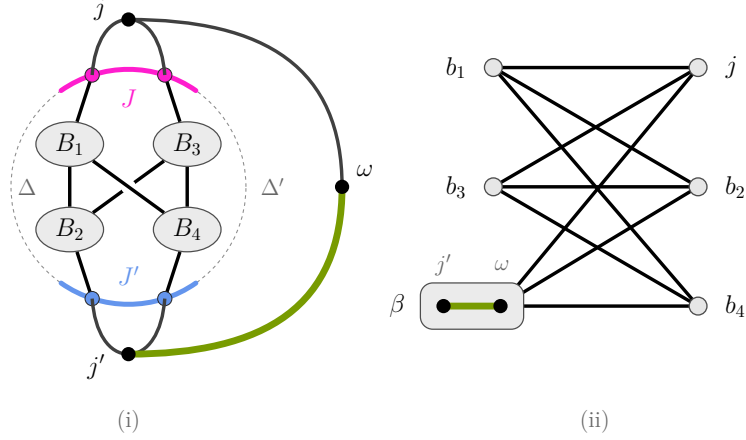
\begin{figure}[ht]
\begin{center}
\scalebox{.67}{%
\begin{tikzpicture}[x=1pt,y=1pt,
    bnd/.style   ={draw=black!45,line width=.6pt,dash pattern=on 2pt off 2pt},
    arcM/.style  ={draw=clM,line width=2.8pt,line cap=round},
    arcB/.style  ={draw=clB,line width=2.8pt,line cap=round},
    edge/.style  ={draw=black,line width=1.9227pt,line join=round},
    aux/.style   ={draw=black!75,line width=1.9227pt,line join=round},
    con/.style   ={draw=clG!85!black,line width=3.8454pt,line join=round},
    blob/.style  ={draw=black!70,line width=.7pt,fill=black!8},
    mdot/.style  ={circle,inner sep=0pt,minimum size=8pt,fill=clM,draw=black!85,line width=.55pt},
    bdot/.style  ={circle,inner sep=0pt,minimum size=8pt,fill=clB,draw=black!85,line width=.55pt},
    kdot/.style  ={circle,inner sep=0pt,minimum size=8pt,fill=black},
    ldot/.style  ={circle,inner sep=0pt,minimum size=10pt,fill=black!10,draw=black!70,line width=.75pt},
    lab/.style   ={font=\fontsize{14.4}{17}\selectfont},
    slab/.style  ={font=\fontsize{12.8}{15}\selectfont,text=black!60}]

  \draw[bnd] (85,88) circle (66pt);
  \node[slab] at (28,88) {$\Delta$};
  \node[slab] at (166,88) {$\Delta'$};

  \draw[arcM] (85,88) ++(55:66pt) arc (55:125:66pt);
  \draw[arcB] (85,88) ++(235:66pt) arc (235:305:66pt);
  \node[lab,text=clM] at (85,138) {$J$};
  \node[lab,text=clB] at (85, 38) {$J'$};

  \coordinate (p1) at (64.6,150.8);
  \coordinate (p3) at (105.4,150.8);
  \coordinate (p2) at (64.6, 25.2);
  \coordinate (p4) at (105.4, 25.2);

  \node[blob,ellipse,minimum width=38pt,minimum height=25pt] (B1) at (52,112) {};
  \node[blob,ellipse,minimum width=38pt,minimum height=25pt] (B3) at (118,112) {};
  \node[blob,ellipse,minimum width=38pt,minimum height=25pt] (B2) at (52,64) {};
  \node[blob,ellipse,minimum width=38pt,minimum height=25pt] (B4) at (118,64) {};

  \draw[edge] (B1) -- (B2);
  \draw[edge] (B3) -- (B4);
  \draw[edge] (B2) -- (B3);
  \fill[white] (85,88) circle (4.2pt);
  \draw[edge] (B4) -- (B1);

  \draw[edge] (B1) -- (p1);
  \draw[edge] (B3) -- (p3);
  \draw[edge] (B2) -- (p2);
  \draw[edge] (B4) -- (p4);
  \node[mdot] at (p1) {};
  \node[mdot] at (p3) {};
  \node[bdot] at (p2) {};
  \node[bdot] at (p4) {};

  \node[lab] at (52,112) {$B_{1}$};
  \node[lab] at (118,112) {$B_{3}$};
  \node[lab] at (52,64) {$B_{2}$};
  \node[lab] at (118,64) {$B_{4}$};

  \coordinate (j)  at (85,182);
  \coordinate (jp) at (85, -6);
  \coordinate (om) at (205,88);
  \draw[aux] (p1) .. controls (66,170) and (74,182) .. (j);
  \draw[aux] (p3) .. controls (104,170) and (96,182) .. (j);
  \draw[aux] (p2) .. controls (66,  6) and (74, -6) .. (jp);
  \draw[aux] (p4) .. controls (104,  6) and (96, -6) .. (jp);
  \draw[aux] (j)  .. controls (162,182) and (205,155) .. (om);
  \draw[con] (jp) .. controls (162, -6) and (205, 21) .. (om);
  \node[kdot] at (j) {};
  \node[kdot] at (jp) {};
  \node[kdot] at (om) {};
  \node[lab] at (69,188) {$j$};
  \node[lab] at (67,-13) {$j'$};
  \node[lab] at (218,97) {$\omega$};

  \node[slab] at (85,-44) {(i)};

  \coordinate (b1) at (290,155);
  \coordinate (b3) at (290, 88);
  \coordinate (be) at (294, 21);
  \coordinate (jj) at (405,155);
  \coordinate (b2) at (405, 88);
  \coordinate (b4) at (405, 21);

  \foreach \l in {b1,b3,be}{\foreach \r in {jj,b2,b4}{\draw[edge] (\l) -- (\r);}}

  \node[blob,rounded corners=7pt,minimum width=58pt,minimum height=26pt] at (278,21) {};
  \draw[con] (262,21) -- (294,21);
  \node[kdot] at (262,21) {};
  \node[kdot] at (294,21) {};
  \node[slab] at (262,44) {$j'$};
  \node[slab] at (295,44) {$\omega$};

  \node[ldot] at (b1) {}; \node[ldot] at (b3) {};
  \node[ldot] at (jj) {}; \node[ldot] at (b2) {}; \node[ldot] at (b4) {};

  \node[lab] at (269,155) {$b_{1}$};
  \node[lab] at (269, 88) {$b_{3}$};
  \node[lab] at (236, 21) {$\beta$};
  \node[lab] at (424,155) {$j$};
  \node[lab] at (426, 88) {$b_{2}$};
  \node[lab] at (426, 21) {$b_{4}$};

  \node[slab] at (347,-44) {(ii)};
\end{tikzpicture}
}
\end{center}
\caption{The situation excluded by \zcref{cl_lb_C4}.
In (i), the branch sets $B_{1},\ldots,B_{4}$ of a model of $(C_{4},\sigma_{1})$ inside a realization $D$ that meets only two color intervals, whose arcs $J$ and $J'$ carry the colors $c$ and $c'$ respectively.
The sets $B_{1}$ and $B_{3}$ carry $c$ and therefore meet $J,$ while $B_{2}$ and $B_{4}$ carry $c'$ and meet $J';$ the four sets are joined in a cycle in this cyclic order.
Outside the disk $\Delta$ lie the three vertices added by the auxiliary construction: $j,$ joined to the vertices of $D$ on $J,$ then $j',$ joined to those on $J',$ and finally $\omega,$ joined to $j$ and to $j'.$
The drawing is schematic: the two diagonals of the cycle are shown crossing and it is precisely the impossibility of drawing this configuration in $\Delta$ without crossings that the claim establishes.
In (ii), the resulting $K_{3,3}$-minor, obtained by contracting each $B_{m}$ to a vertex $b_{m}$ and the green edge $\omega j'$ to a vertex $\beta.$
Its two sides are $\{b_{1},b_{3},\beta\}$ and $\{j,b_{2},b_{4}\}:$ each of $b_{1}$ and $b_{3}$ reaches $j$ through $J$ and reaches $b_{2}$ and $b_{4}$ along the cycle, while $\beta$ reaches $j$ through $\omega$ and reaches $b_{2}$ and $b_{4}$ through $j'.$}
\label{fig_C4Interval}
\end{figure}

For the last case we may assume that $(H,\psi)$ excludes every member of $\Ocal^{0}\cup\Ocal^{1}\cup\tilde{\Ocal}^{2}\cup\Ocal^{3}$ as a colorful minor, so that the member of $\Ocal$ supplied by \zcref{obs_CrucialMinimal} belongs to $\Ocal^{4}.$

\begin{figure}[ht]
 \centering
 \scalebox{0.5019}{%

}

 \caption{A drawing of $(Z^{\circ5},\zeta^{\circ5})$ used in the last part of the proof of \zcref{thm_EPiffCrucial}.}
 \label{fig_vortexConstruction}%
\end{figure}

\medskip\noindent\textbf{Case 6: $(Z,\zeta)\in\Ocal^{4}$.}
Here $Z=2\cdot K_{1}$ and $\zeta$ assigns two colors to each of the two vertices, four distinct colors in total; without loss of generality these are $\textcolor{HotMagenta}{1},$ $\textcolor{CornflowerBlue}{2},$ $\textcolor{ChromeYellow}{3},$ and $\textcolor{AppleGreen}{4},$ where one vertex receives $\textcolor{HotMagenta}{1}$ and $\textcolor{ChromeYellow}{3}$ and the other receives $\textcolor{CornflowerBlue}{2}$ and $\textcolor{AppleGreen}{4}.$
The scarce resource is now neither genus nor boundary but width.
A copy of $(Z,\zeta)$ in a host consists of a connected set joining a vertex colored $1$ to a vertex colored $3$ and a disjoint connected set joining a vertex colored $2$ to a vertex colored $4;$ if the four colors occur on the boundary of a disk in the cyclic order $1,2,3,4,$ then the two sets must cross.
We build hosts in which all such crossings are confined to one small region and bound the packing number by the amount of traffic that region admits.

\medskip
We first record the structure that the excluded obstructions impose.
As $(H,\psi)$ excludes $\Ocal^{3},$ it is component-wise bicolored by \zcref{obs_SinglecomBicolored}.
As it excludes $\tilde{\Ocal}^{2}$ and $\Ocal^{3},$ it excludes all of $\Ocal^{2},$ because by \zcref{lemma_Q12qNotMinimal} every member of $\Ocal^{2}\setminus\tilde{\Ocal}^{2}$ contains a member of $\Ocal^{3};$ hence it is color-segmented by \zcref{obs_colorSegmentation}.
As it excludes $\Ocal^{0}\cup\Ocal^{1},$ it is color-facial by \zcref{lemma_colorFacial} and the same then holds for each of its components.
Altogether:
\begin{align}
\label{eq_lb_structure}
\begin{minipage}{0.9\textwidth}
every component $C$ of $H$ carries at most two distinct colors and admits an embedding in a disk $\Delta_{C}$ in which the vertices drawn on the boundary are exactly the colored ones; and if $C$ carries two colors $i$ and $j,$ then there are no vertices $a,b,c,d$ appearing in this cyclic order on the boundary of $\Delta_{C}$ with $i\in\psi(a)\cap\psi(c)$ and $j\in\psi(b)\cap\psi(d).$
\end{minipage}
\end{align}
The last assertion is the condition A in the definition of color-segmented colorful graphs.

As $(Z,\zeta)$ is a colorful minor of $(H,\psi),$ and by \eqref{eq_lb_structure}, there are two distinct components of $H,$ which we denote by $C_{1,3}$ and $C_{2,4},$ such that $C_{1,3}$ carries the colors $1$ and $3$ and $C_{2,4}$ carries the colors $2$ and $4.$
We apply to $(H,\psi)$ the uncontraction of the previous cases, componentwise, obtaining $(H^{\circ},\psi^{\circ})$ with components $C^{\circ}_{1,3},$ $C^{\circ}_{2,4},$ and possibly others, each colored vertex carrying exactly one color and with $(H,\psi)$ a colorful minor of $(H^{\circ},\psi^{\circ}).$
We write $\Delta_{1,3}$ and $\Delta_{2,4}$ for the disks in which $C^{\circ}_{1,3}$ and $C^{\circ}_{2,4}$ are drawn as in \eqref{eq_lb_structure}.
The components of $H^{\circ}$ other than $C^{\circ}_{1,3}$ and $C^{\circ}_{2,4}$ play no role in the construction, but they must travel with it, as our hosts have to contain all of $(H,\psi);$ we therefore draw each of them, embedded in a disk as in \eqref{eq_lb_structure}, in the interior of $\Delta_{1,3},$ disjointly from $C^{\circ}_{1,3}$ and from one another.
From now on we let $C^{\circ}_{1,3}$ denote the resulting drawing in $\Delta_{1,3},$ so that $H^{\circ}$ is the union of the two drawings in $\Delta_{1,3}$ and in $\Delta_{2,4};$ note that the vertices drawn on the boundary of $\Delta_{1,3}$ are still exactly the vertices colored $1$ or $3.$

We also record the colorful graph $(Z^{\circ},\zeta^{\circ})$ obtained from $(Z,\zeta)$ by the same uncontraction: it satisfies $Z^{\circ}=2\cdot K_{2},$ one of its two components carrying the colors $1$ and $3$ on its two vertices and the other carrying $2$ and $4.$

\paragraph{Cutting the two components.}
Let $E_{1,3}$ be a set of edges of $C^{\circ}_{1,3}$ such that no component of $C^{\circ}_{1,3}-E_{1,3}$ contains vertices of both colors $1$ and $3,$ that is, an edge cut separating the vertices colored $1$ from those colored $3;$ we define $E_{2,4}$ analogously for $C^{\circ}_{2,4}.$
By \eqref{eq_lb_structure} the vertices colored $1$ and those colored $3$ occupy two disjoint arcs of the boundary of $\Delta_{1,3},$ so we may decompose $\Delta_{1,3}$ into three regions $R_{1},$ $\bar{\Delta}_{1,3},$ and $R_{3}$ such that $\bar{\Delta}_{1,3}$ is a disk contained in $\Delta_{1,3}$ whose removal leaves exactly the two components $R_{1}$ and $R_{3},$ such that the boundary of $R_{1}$ (respectively of $R_{3}$) meets the boundary of $\Delta_{1,3}$ exactly in the points where the vertices colored $1$ (respectively $3$) are drawn and such that every edge of $E_{1,3}$ is drawn in the interior of $\bar{\Delta}_{1,3}$ with one endpoint in $R_{1}\cap\bar{\Delta}_{1,3}$ and the other in $R_{3}\cap\bar{\Delta}_{1,3}.$
We define $R_{2},$ $\bar{\Delta}_{2,4},$ and $R_{4}$ inside $\Delta_{2,4}$ in the same way, with $2$ and $4$ in the roles of $1$ and $3.$

\begin{figure}[ht]
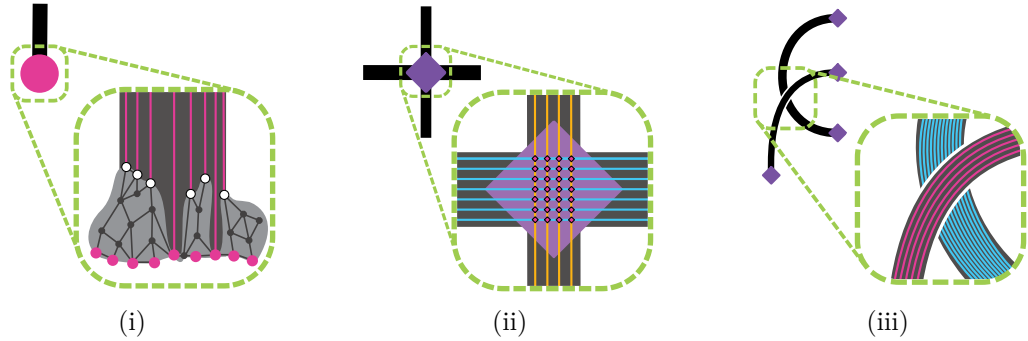

 \centering

}} at (C.center);
 \end{pgfonlayer}
			
 \begin{pgfonlayer}{main}
 \node (C) [v:ghost] {};
 
 \end{pgfonlayer}
 
 \begin{pgfonlayer}{foreground}
 \end{pgfonlayer}

 \end{tikzpicture}
 };

 \node (Rlabel) [v:ghost,position=270:2.35cm from R] {(iii)};

 \node (Mlabel) [v:ghost,position=180:5cm from Rlabel] {(ii)};

 \node (Llabel) [v:ghost,position=180:5cm from Mlabel] {(i)}; 
 
 \end{pgfonlayer}
 
 \begin{pgfonlayer}{foreground}
 \end{pgfonlayer}

 \end{tikzpicture}

 \caption{Magnifications of three crucial parts of the construction of $(Z^{\circ5},\zeta^{\circ5})$ in the proof of \zcref{thm_EPiffCrucial}: (i) a copy of the subgraph with monochromatic boundary, (ii) the gadget inserted to resolve crossings inside the drawing, and (iii) the way large collections of pairwise disjoint paths cross in the \say{vortex}.}
 \label{fig_construction1}
\end{figure}

\paragraph{The skeleton.}
For $k\in\Nbbb_{\geq1},$ let $(Z^{\circ2k},\zeta^{\circ2k})$ be the $2k$-multiplication of $(Z^{\circ},\zeta^{\circ}),$ drawn in a disk so that its colored vertices lie on the boundary in the cyclic order $1,2,3,4,$ that is, so that first come all vertices colored $1,$ then all colored $2,$ then all colored $3,$ and finally all colored $4.$
This colorful graph is the superposition of $2k$ \emph{$(1,3)$-paths} and $2k$ \emph{$(2,4)$-paths}, where a $(1,3)$-path joins a vertex colored $1$ to one colored $3$ and a $(2,4)$-path is defined analogously; we index them following the cyclic order of their endpoints.
Since the four groups of endpoints occur in the order $1,2,3,4,$ every $(1,3)$-path crosses every $(2,4)$-path  and all these crossings can be confined to a single disk $\hat{\Delta}$ whose interior meets the drawing only in edges and whose boundary carries their endpoints; see \zcref{fig_vortexConstruction}.
We write $E_{k}$ for the set of these edges and $\Omega$ for the cyclic ordering in which their endpoints appear on the boundary of $\hat{\Delta},$ and we note that the society $((V(\Omega),E_{k}),\Omega)$ has depth at most $4.$

\begin{figure}[ht]
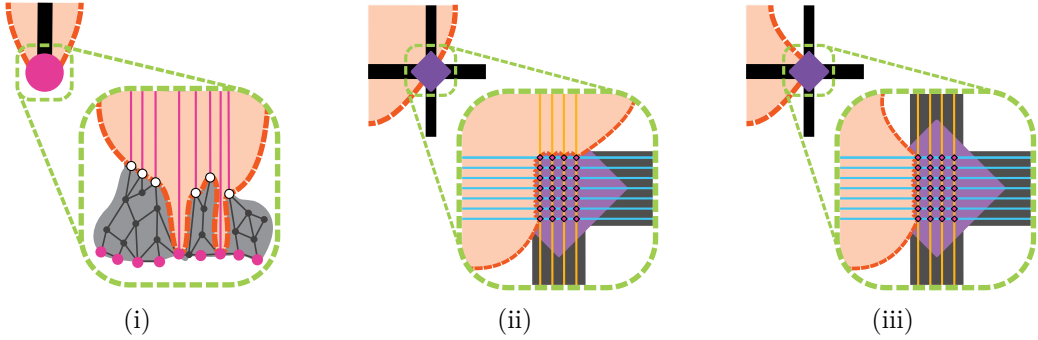

 \centering

}} at (C.center);
 \end{pgfonlayer}
			
 \begin{pgfonlayer}{main}
 \node (C) [v:ghost] {};
 
 \end{pgfonlayer}
 
 \begin{pgfonlayer}{foreground}
 \end{pgfonlayer}

 \end{tikzpicture}
 };

 \node (Rlabel) [v:ghost,position=270:2.35cm from R] {(iii)};

 \node (Mlabel) [v:ghost,position=180:5cm from Rlabel] {(ii)};

 \node (Llabel) [v:ghost,position=180:5cm from Mlabel] {(i)}; 
 
 \end{pgfonlayer}
 
 \begin{pgfonlayer}{foreground}
 \end{pgfonlayer}

 \end{tikzpicture}

 \caption{Magnifications of the vortex boundary in the construction of $(Z^{\circ5},\zeta^{\circ5})$ in the proof of \zcref{thm_EPiffCrucial}: (i) the vortex boundary interacting with one of the monochromatic components, (ii) the way the vortex boundary interacts with a planarization gadget when both flows enter the vortex and (iii) the way the vortex boundary interacts with a planarization gadget when only one of the two flows enters the vortex.}
 \label{fig_construction2}
\end{figure}

\paragraph{The host.}
We now construct the host $(G_{k},\chi_{k}),$ drawn in a disk $\Delta^{\ast},$ by inflating the skeleton.
In the drawing of $(Z^{\circ2k},\zeta^{\circ2k}),$ every $(1,3)$-path is replaced by a copy of the drawing of $C^{\circ}_{1,3}$ in $\Delta_{1,3},$ in such a way that its endpoint colored $1$ is replaced by $R_{1},$ its edge is replaced by $\bar{\Delta}_{1,3},$ and its endpoint colored $3$ is replaced by $R_{3};$ every $(2,4)$-path is replaced by a copy of the drawing of $C^{\circ}_{2,4}$ in $\Delta_{2,4}$ in the same way.
The boundary of $\Delta^{\ast}$ then carries, in cyclic order, first all vertices colored $1,$ then all colored $2,$ then all colored $3,$ and finally all colored $4.$
The drawing is completed by resolving the crossings between two copies: each crossing point of two copies is replaced by a vertex of degree $4,$ as indicated in \zcref{fig_construction1}.(ii), while the crossings indicated in \zcref{fig_construction1}.(iii) are left as they are.
We call \emph{leveled} the vertices of degree $4$ created in this way.
For $i\in[2k],$ we write $W_{i}$ for the union of the copy of $C^{\circ}_{1,3}$ replacing the $i$-th $(1,3)$-path and of the copy of $C^{\circ}_{2,4}$ replacing the $i$-th $(2,4)$-path, together with the leveled vertices on them; see \zcref{fig_construction2}.
Then $G_{k}=\bigcup\{W_{i}\mid i\in[2k]\}.$
Finally, we let $\chi_{k}$ assign to every non-leveled vertex of $G_{k}$ the colors that the corresponding vertex of $H^{\circ}$ carries under $\psi^{\circ},$ and no color to every leveled vertex.

\begin{claim}
\label{cl_lb_vortex_cover}
For every $k\in\Nbbb_{\geq1},$ it holds that $\cover_{H,\psi}(G_{k},\chi_{k})\geq k.$
\end{claim}

\begin{claimproof}
For every $i\in[2k],$ the colorful graph $(W_{i},\chi_{k})$ is a subdivision of a copy of $(H^{\circ},\psi^{\circ})$ in which the subdividing vertices are leveled and hence carry no color, so it contains $(H^{\circ},\psi^{\circ}),$ and therefore also $(H,\psi),$ as a colorful minor; we fix a realization $D_{i}$ of $(H,\psi)$ inside it.
By construction every vertex of $G_{k}$ belongs to at most two of $W_{1},\ldots,W_{2k},$ so a set $S\subseteq V(G_{k})$ with $|S|\leq k-1$ meets at most $2(k-1)<2k$ of $D_{1},\ldots,D_{2k}$ and therefore misses one of them.
Hence no such $S$ is a covering of $(H,\psi)$ in $(G_{k},\chi_{k}).$
\end{claimproof}

\paragraph{Bounding the packing number.}
In $G_{k},$ the disk $\hat{\Delta}$ of the skeleton is replaced by a disk $\hat{\Delta}^{\ast},$ as indicated in \zcref{fig_construction2}, whose interior again meets the drawing only in edges.
Every edge of $E_{k}$ belonging to a $(1,3)$-path is replaced in $\hat{\Delta}^{\ast}$ by the $|E_{1,3}|$ edges of the corresponding copy of the cut $E_{1,3},$ and every edge of $E_{k}$ belonging to a $(2,4)$-path by the $|E_{2,4}|$ edges of the corresponding copy of $E_{2,4}.$
We let $\Omega^{\ast}$ be the cyclic ordering of the vertices of $G_{k}$ drawn on the boundary of $\hat{\Delta}^{\ast}$ and we let $G_{k}'$ be the part of $G_{k}$ drawn in $\hat{\Delta}^{\ast}.$
As every crossing of the skeleton is replaced by the crossings of $|E_{1,3}|$ edges with $|E_{2,4}|$ edges and as the society of the skeleton has depth at most $4,$ we obtain that
\begin{align}
\label{eq_lb_depth}
\text{the depth of }(G_{k}',\Omega^{\ast})\text{ is at most }2\bigl(|E_{1,3}|+|E_{2,4}|\bigr).
\end{align}

\begin{claim}
\label{cl_lb_vortex_pack}
For every $k\in\Nbbb_{\geq1},$ it holds that $\pack_{H,\psi}(G_{k},\chi_{k})\leq2\bigl(|E_{1,3}|+|E_{2,4}|\bigr)+2.$
\end{claim}

\begin{claimproof}
Set $\ell\coloneqq\pack_{Z,\zeta}(G_{k},\chi_{k}),$ so that $\pack_{H,\psi}(G_{k},\chi_{k})\leq\ell$ by \zcref{obs_packMonotone} and let $D_{1},\ldots,D_{\ell}$ be pairwise vertex-disjoint realizations of $(Z,\zeta)$ in $(G_{k},\chi_{k}).$
As every vertex of $G_{k}$ carries at most one color, each $D_{i}$ is the disjoint union of a connected subgraph joining a vertex colored $1$ to a vertex colored $3$ and of a connected subgraph joining a vertex colored $2$ to a vertex colored $4;$ by minimality these are paths.
Collecting them, we obtain a linkage $L_{1,3}$ of $\ell$ pairwise disjoint paths from vertices colored $1$ to vertices colored $3,$ and a linkage $L_{2,4}$ of $\ell$ pairwise disjoint paths from vertices colored $2$ to vertices colored $4,$ with $V(L_{1,3})\cap V(L_{2,4})=\emptyset.$
We may assume that $\ell\geq3,$ as otherwise the asserted bound holds trivially.

Order the paths $P_{1},\ldots,P_{\ell}$ of $L_{1,3}$ following the cyclic order of their endpoints colored $1$ on the boundary of $\Delta^{\ast}.$
Outside $\hat{\Delta}^{\ast}$ the drawing of $G_{k}$ has no crossings, so two disjoint paths may cross only inside $\hat{\Delta}^{\ast}.$
As the four color classes occur on the boundary of $\Delta^{\ast}$ in the cyclic order $1,2,3,4,$ a path of $L_{1,3}$ that avoided $\hat{\Delta}^{\ast}$ altogether would separate, inside $\Delta^{\ast},$ the vertices colored $2$ from those colored $4,$ and would therefore meet every path of $L_{2,4},$ which is impossible.
Consequently each of $P_{1},\ldots,P_{\ell}$ contains at least one vertex of $V(\Omega^{\ast}).$

Let $\Lambda$ be the region obtained from $\Delta^{\ast}$ by removing the interior of $\hat{\Delta}^{\ast}$ together with the drawings of $P_{1}$ and $P_{\ell}.$
Then $\Lambda$ has a connected component $\Delta^{\bullet}$ that is a disk and meets the drawing of $P_{2}.$
Let $a,$ $b,$ and $c$ be the first vertices of $V(\Omega^{\ast})$ met by $P_{1},$ $P_{2},$ and $P_{\ell},$ respectively, when each is traversed from its endpoint colored $1,$ and let $S$ be the maximal segment of $\Omega^{\ast}$ that contains $b$ and contains neither $a$ nor $c.$
Each of $P_{2},\ldots,P_{\ell-1}$ enters $\hat{\Delta}^{\ast}$ through $\Delta^{\bullet}$ and terminates outside $\Delta^{\bullet},$ so each of them meets $V(\Omega^{\ast})$ both inside $S$ and outside $S,$ and the subpaths of $P_{2},\ldots,P_{\ell-1}$ inside $\hat{\Delta}^{\ast}$ form a transaction of $(G_{k}',\Omega^{\ast})$ of size $\ell-2.$
By \eqref{eq_lb_depth}, we get $\ell-2\leq2(|E_{1,3}|+|E_{2,4}|),$ which is the assertion.
\end{claimproof}

As the bound of \zcref{cl_lb_vortex_pack} depends only on $(H,\psi)$ and not on $k,$ \zcref{cl_lb_vortex_cover,cl_lb_vortex_pack} give what is required, with $c\coloneqq2(|E_{1,3}|+|E_{2,4}|)+2.$
This exhausts all cases and completes the proof.
\end{proof}

We are now ready to derive our main result.

\begin{proof}[Proof of \zcref{th_EP_single}]
The equivalence of \ref{it_s_cru} and \ref{it_s_O} is \zcref{obs_CrucialMinimal} and the equivalence of \ref{it_s_cru} and \ref{it_s_U} is \zcref{lemma_U_equals_Q}.
The implication from \ref{it_s_cru} to \ref{it_s_ep} is \zcref{thm_EP_positive} and the implication from \ref{it_s_ep} to \ref{it_s_cru} is \zcref{thm_EPiffCrucial}.
\end{proof}

\begin{proof}[Proof of \zcref{cor_final_EP}]
\label{page_cor_final_EP}
By \zcref{th_EP_single}, $(H,\psi)$ has the Erd\H{o}s-P{\'o}sa property if and only if it is crucial, so it suffices to prove that a $[q]$-rainbow colorful graph $(H,\psi)$ with $V(H)\neq\emptyset$ is crucial if and only if the second, respectively the third, statement holds.
Notice first that, as every vertex of $H$ carries all colors of $[q],$ every component of $H$ carries exactly $q$ colors, so $(H,\psi)$ is component-wise bicolored if and only if $q\leq 2$ and, in this case, it is also single-component bicolored, as at most two distinct colors appear in total.
This proves that $q\leq2$ is necessary for cruciality, so assume from now on that $q\leq2.$

If $q=0,$ then $(H,\psi)$ has no colored vertices, so the conditions A, B, and C in the definition of color-segmented colorful graphs hold trivially and $(H,\psi)$ is color-facial if and only if $H$ is planar.
If $q=1,$ then only one color occurs in $(H,\psi),$ so $|\psi(B)|\leq 1$ for every $B\subseteq V(H)$ and the conditions B and C hold trivially, while the condition A holds as well, as every color $c_{\ell}$ that it offers equals $1,$ whence $\{c_{1},c_{3}\}\cap\{c_{2},c_{4}\}=\{1\}\neq\emptyset.$
Moreover, all vertices of $H$ are colored, so $(H,\psi)$ is color-facial if and only if $H$ is outerplanar.
Finally, if $q=2,$ then every vertex of $H$ carries two colors and therefore, by the condition B, $H$ is acyclic, that is, $H$ is a forest, while the condition C implies that every vertex of $H$ has at most two neighbors.
Hence, if $(H,\psi)$ is crucial then $H$ is a linear forest and, conversely, if $H$ is a linear forest, then $H$ has no cycles, so the conditions A and B hold trivially, the condition C holds as $H$ has maximum degree at most two and $(H,\psi)$ is color-facial, as every linear forest has a plane embedding where all vertices are on the boundary of the outer face.
This proves the equivalence of the first two statements.

For the third one, recall that a graph is planar if and only if it excludes $K_{5}$ and $K_{3,3}$ as minors and that it is outerplanar if and only if it excludes $K_{4}$ and $K_{2,3}$ as minors.
Moreover, a graph is a linear forest if and only if it excludes $K_{3}$ and $K_{1,3}$ as minors: a graph without $K_{3}$-minors is a forest and a forest without $K_{1,3}$-minors has maximum degree at most two, while a linear forest is acyclic and has maximum degree at most two.
As $K_{5-q}$ and $K_{3-q,3}$ are exactly the two excluded minors above for $q\in[0,2],$ the third statement is equivalent to the second one.
\end{proof}

\section{Beyond a single  colorful graph}
\label{sec_beyond}

We conclude by observing that \zcref{th_EP_single} does not extend to the exclusion of more than one colorful graph.
Given a finite set $\Fcal$ of colorful graphs, a \emph{packing of $\Fcal$ in $(G,\chi)$ of size $k$} is a family $(G_{1},\ldots,G_{k})$ of pairwise vertex-disjoint subgraphs of $G$ such that, for every $i\in[k],$ some member of $\Fcal$ is a colorful minor of $(G_{i},\chi),$ and a \emph{covering of $\Fcal$ in $(G,\chi)$ of size $k$} is a set $S\subseteq V(G)$ with $|S|\leq k$ such that no member of $\Fcal$ is a colorful minor of $(G-S,\chi).$
Accordingly, $\Fcal$ \emph{has the Erd\H{o}s-P{\'o}sa property} if there is a function $f\colon\Nbbb\to\Nbbb$ such that, for every $k\in\Nbbb,$ every colorful graph has a packing of $\Fcal$ of size $k$ or a covering of $\Fcal$ of size at most $f(k).$
For $\Fcal=\{(H,\psi)\},$ this is precisely the Erd\H{o}s-P{\'o}sa property of $(H,\psi).$

\subsection{Cruciality fails for many colorful graphs}

The next proposition shows that no analogue of \zcref{th_EP_single} can hold for finite sets of colorful graphs: the set that it provides contains a crucial colorful graph and even one whose palette is that of all its members and still fails to have the Erd\H{o}s-P{\'o}sa property.

\begin{proposition}
\label{prop_no_domain_EP}
There is a finite set $\Fcal$ of colorful graphs, all of palette $\{1\},$ that contains a crucial colorful graph and does not have the Erd\H{o}s-P{\'o}sa property.
\end{proposition}

\begin{proof}
Let $F_{1}\coloneqq(K_{3},\rho_{\{1\}}),$ let $F_{2}$ be the disjoint union $(K_{1},\rho_{\{1\}})\sqcup(K_{5},\rho_{\emptyset}),$ and let $\Fcal\coloneqq\{F_{1},F_{2}\}.$
Both members have palette $\{1\}$ and $F_{1}$ is crucial: it is outerplanar, hence color-facial, it has three vertices only and carries a single color, so it is color-segmented and it is connected with a single color, so it is component-wise and single-component bicolored.

We now prove that $\Fcal$ does not have the Erd\H{o}s-P{\'o}sa property.
By the result of Robertson and Seymour \cite{RobertsonS1986Grapha}, the colorful graph $(K_{5},\rho_{\emptyset})$ does not have the Erd\H{o}s-P{\'o}sa property.
Hence, by \zcref{obs_EPfails}, there is some $k_{0}\in\Nbbb$ such that, for every $n\in\Nbbb,$ there is a graph $H_{n}$ without colors with $\pack_{K_{5},\rho_{\emptyset}}(H_{n})\leq k_{0}$ and $\cover_{K_{5},\rho_{\emptyset}}(H_{n})\geq n.$
Let $(G_{n},\chi_{n})$ be the disjoint union of $H_{n}$ and of $n$ isolated vertices, each carrying the color $1.$

First, $(G_{n},\chi_{n})$ contains no model of $F_{1}$: every vertex carrying the color $1$ is isolated, so a connected set containing one is a single vertex without neighbors, while a model of $F_{1}$ requires three pairwise adjacent branch sets, each carrying the color $1.$
Second, every subgraph of $(G_{n},\chi_{n})$ that contains a member of $\Fcal$ contains $F_{2}$ and hence a model of $K_{5},$ which lies in $H_{n},$ as the isolated colored vertices have no edges; consequently, every packing of $\Fcal$ in $(G_{n},\chi_{n})$ has size at most $\pack_{K_{5},\rho_{\emptyset}}(H_{n})\leq k_{0}.$
Third, if a set $S\subseteq V(G_{n})$ leaves some colored vertex and some model of $K_{5}$ in $H_{n}-S,$ then a model of $F_{2}$ survives in $(G_{n}-S,\chi_{n})$; hence every covering of $\Fcal$ either contains all $n$ colored vertices or meets every model of $K_{5}$ in $H_{n},$ so it has size at least $\min\{n,\cover_{K_{5},\rho_{\emptyset}}(H_{n})\}\geq n.$
Therefore, for $k\coloneqq k_{0}+1,$ the colorful graphs $(G_{n},\chi_{n})$ have no packing of $\Fcal$ of size $k$ and no covering of $\Fcal$ of size smaller than $n,$ so no gap $f(k)$ exists and $\Fcal$ does not have the Erd\H{o}s-P{\'o}sa property.
\end{proof}

\paragraph{Excluding an antichain.}
A set $\Fcal$ may always be assumed to be an antichain for the colorful minor relation: replacing $\Fcal$ by the set of its $\leq$-minimal members changes neither its packings nor its coverings, as every member contains a minimal one.
For $|\Fcal|=1,$ 
the Erd\H{o}s-P{\'o}sa 
is delineated by 
\zcref{th_EP_single}.
\zcref{prop_no_domain_EP} shows that the general case cannot be read in the same way: the presence of a crucial member, even one whose palette is that of the whole set, is not sufficient.
What is intriguing is that the two members it provides are individually as harmless as can be --- the first is crucial and the second becomes crucial once its uncolored component is deleted --- and that it is only their interaction that destroys the property.
A criterion for antichains must therefore be \emph{relational}: it has to compare the members with one another and no condition inspecting them separately can express it.
\medskip

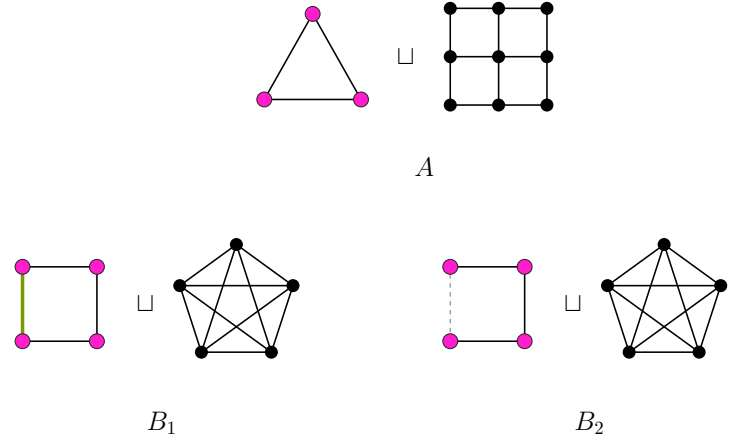
\begin{figure}[!ht]
\begin{center}
\scalebox{.7}{%
\begin{tikzpicture}[x=1pt,y=1pt,
    edge/.style  ={draw=black,line width=.85pt,line join=round},
    hi/.style    ={draw=clG!85!black,line width=1.9pt,line join=round},
    ghost/.style ={draw=black!35,line width=.85pt,dash pattern=on 2.4pt off 2.4pt},
    cdot/.style  ={circle,inner sep=0pt,minimum size=8pt,fill=clM,draw=black!85,line width=.55pt},
    pdot/.style  ={circle,inner sep=0pt,minimum size=7pt,fill=black},
    lab/.style   ={font=\fontsize{14.4}{17}\selectfont}]

  \coordinate (t1) at (176,196);
  \coordinate (t2) at (150,150);
  \coordinate (t3) at (202,150);
  \draw[edge] (t1)--(t2)--(t3)--(t1);
  \node[cdot] at (t1) {}; \node[cdot] at (t2) {}; \node[cdot] at (t3) {};
  \node[lab] at (226,173) {$\sqcup$};
  \foreach \x in {0,1,2} \foreach \y in {0,1,2}
    \coordinate (g\x\y) at ({250+26*\x},{147+26*\y});
  \foreach \y in {0,1,2} \draw[edge] (g0\y)--(g2\y);
  \foreach \x in {0,1,2} \draw[edge] (g\x0)--(g\x2);
  \foreach \x in {0,1,2} \foreach \y in {0,1,2} \node[pdot] at (g\x\y) {};
  \node[lab] at (236,114) {\Large $A$};

  \coordinate (c1) at (20,60); \coordinate (c2) at (60,60);
  \coordinate (c3) at (60,20); \coordinate (c4) at (20,20);
  \draw[edge] (c1)--(c2); \draw[edge] (c2)--(c3); \draw[edge] (c3)--(c4);
  \draw[hi] (c4)--(c1);
  \foreach \p in {c1,c2,c3,c4} \node[cdot] at (\p) {};
  \node[lab] at (86,40) {$\sqcup$};
  \foreach \i in {0,...,4} \coordinate (p\i) at ({135+32*cos(90+72*\i)},{40+32*sin(90+72*\i)});
  \foreach \i/\j in {0/1,0/2,0/3,0/4,1/2,1/3,1/4,2/3,2/4,3/4} \draw[edge] (p\i)--(p\j);
  \foreach \i in {0,...,4} \node[pdot] at (p\i) {};
  \node[lab] at (95,-24) {\Large $B_{1}$};

  \coordinate (d1) at (250,60); \coordinate (d2) at (290,60);
  \coordinate (d3) at (290,20); \coordinate (d4) at (250,20);
  \draw[edge] (d1)--(d2); \draw[edge] (d2)--(d3); \draw[edge] (d3)--(d4);
  \draw[ghost] (d4)--(d1);
  \foreach \p in {d1,d2,d3,d4} \node[cdot] at (\p) {};
  \node[lab] at (316,40) {$\sqcup$};
  \foreach \i in {0,...,4} \coordinate (q\i) at ({365+32*cos(90+72*\i)},{40+32*sin(90+72*\i)});
  \foreach \i/\j in {0/1,0/2,0/3,0/4,1/2,1/3,1/4,2/3,2/4,3/4} \draw[edge] (q\i)--(q\j);
  \foreach \i in {0,...,4} \node[pdot] at (q\i) {};
  \node[lab] at (325,-24) {\Large $B_{2}$};
\end{tikzpicture}
}
\end{center}
\caption{Two pairs of annotated colorful graphs that lie on opposite sides of the Erd\H{o}s-P{\'o}sa boundary and differ in a single edge.
The crucial member $A,$ the disjoint union of the $\{1\}$-rainbow triangle and of the uncolored $(3\times3)$-grid, is the same in both and so is the uncolored component $K_{5}$ of the other member; the colored components are the $\{1\}$-rainbow $C_{4}$ in $B_{1}$ and the $\{1\}$-rainbow $P_{4}$ in $B_{2},$ the edge drawn in green in $B_{1}$ being the only difference and appearing ghosted in $B_{2}.$
Contracting that edge merges two adjacent branch sets of the rainbow $C_{4}$ into a rainbow triangle, so the colored part of $A$ is a colorful minor of the colored part of $B_{1};$ no such relation holds for $B_{2},$ where a triangle is not a minor of a path and $K_{5}$ is not a minor of $P_{4}.$}
\label{fig_OneEdge}
\end{figure}

\subsection{One color is already enough to be open}
The difficulty is not created by the number of colors.
Call a colorful graph \emph{annotated} if its palette is a subset of $\{1\},$ so that every vertex either carries the color $1$ or carries no color at all; this is the setting of rooted minors and a colorful graph of this kind is crucial exactly when it is color-facial.
Even for pairs of annotated colorful graphs the question of \zcref{prob_domains} is open and the following pair of pairs shows how fine the boundary is; see \zcref{fig_OneEdge}.

Write $A\coloneqq(K_{3},\rho_{\{1\}})\sqcup(\Gamma,\rho_{\emptyset}),$ where $\Gamma$ denotes the $(3\times3)$-grid and observe that $A$ is crucial.
Let
$$B_{1}\coloneqq(C_{4},\rho_{\{1\}})\sqcup(K_{5},\rho_{\emptyset})
\qquad\text{and}\qquad
B_{2}\coloneqq(P_{4},\rho_{\{1\}})\sqcup(K_{5},\rho_{\emptyset}).$$
The pairs $\{A,B_{1}\}$ and $\{A,B_{2}\}$ differ in the single edge that closes the path into a cycle and they fall on opposite sides of the boundary: $\{A,B_{1}\}$ satisfies the Erdős-Pósa property, while 
$\{A,B_{2}\}$ does not.
The reason for the difference is relational and it is visible in the figure.
Merging two adjacent branch sets of the rainbow $C_{4}$ yields a rainbow triangle, so the colored part of $A$ is a colorful minor of the colored part of $B_{1},$ and the two members of the first pair are linked; in the second pair no link of either kind is available, as a triangle is not a minor of a path and $K_{5}$ is not a minor of $P_{4}.$
The pair of \zcref{prop_no_domain_EP} is of the second kind: its crucial member is the colored part of $A,$ its other member is $B_{2}$ with the colored component shrunk to a single vertex and no link is available there either.
In particular, no analogue of \zcref{cor_EPminorClosed} can be expected for sets: as $B_{2}$ is obtained from $B_{1}$ by deleting one edge of the rainbow $C_{4},$ it is a colorful minor of $B_{1},$ so replacing a member of $\{A,B_{1}\}$ by a colorful minor of it destroys the property.
The failure for $\{A,B_{2}\}$ is unconditional, while the property for $\{A,B_{1}\}$ rests on the hypothesis on $A$ mentioned above, so this last statement is conditional on that hypothesis.
No condition inspecting $A,$ $B_{1},$ and $B_{2}$ one at a time can distinguish the two pairs, since the member that changes, from $B_{1}$ to $B_{2},$ is on its own neither crucial nor far from being so in either case.

\begin{problem}
\label{prob_domains}
Characterize the finite sets of colorful graphs that have the Erd\H{o}s-P{\'o}sa property.
\end{problem}

\noindent
By the discussion above, the first case to settle is that of sets with a disconnected member and it is already open for pairs of annotated  graphs.

\section*{Declaration on the use of generative AI}
During the preparation of this work the authors used Claude (Anthropic) for
LaTeX editing, for consistency and cross-reference checking, for the independent
machine verification of the counting claims of \zcref{sec_obstructions}, and for
producing the final TikZ versions of some of the the figures. The results of this paper, their
proofs, and the constructions they rest on are the authors' own. All text
prepared with the assistance of the tool was reviewed and edited by the authors,
who take full responsibility for the content of this paper.

%\bibliographystyle{plainurl}
%\bibliography{literature_colorful}

\end{document}